\documentclass[11pt]{amsart}

\usepackage[overload]{textcase}
\usepackage{comment}
\usepackage{hyperref}
\usepackage[capitalise,noabbrev]{cleveref}
\usepackage{amsmath}
\usepackage{amssymb}
\usepackage{enumitem}
\usepackage{thmtools}

\usepackage[labelformat=simple]{subcaption}
\usepackage{tikz}
\usepackage{tkz-graph}
\tikzstyle{VertexStyle} = [shape = circle, draw, fill]
\tikzset{pre/.style={-}}    
\tikzstyle{every node}=[circle, inner sep=0pt, minimum width=4pt]
\usetikzlibrary{backgrounds}
\newtheorem{theorem}{Theorem}[section]
\newtheorem{lemma}[theorem]{Lemma}
\newtheorem{corollary}[theorem]{Corollary}
\newtheorem{proposition}[theorem]{Proposition}
\newtheorem{sublemma}{}[theorem]

\newtheorem{conjecture}[theorem]{Conjecture}

\theoremstyle{definition}
\newtheorem{definition}{Definition}[section]

\newcommand{\dY}{$\Delta$-$Y$~exchange}
\newcommand{\Yd}{$Y$-$\Delta$~exchange}
\newcommand{\ba}{\backslash}

\DeclareMathOperator{\cl}{cl}
\DeclareMathOperator{\fcl}{fcl}
\DeclareMathOperator{\co}{co}
\DeclareMathOperator{\si}{si}
\newcommand{\cocl}{\cl^*}

\newcommand{\seq}[1]{[#1]}
\newcommand{\unfortunate}{$N$-grounded}

\newcommand{\pspider}{elongated-quad $3$-separator}
\newcommand{\spider}{double-quad $3$-separator}

\newcommand{\twisted}{skew-whiff $3$-separator}
\newcommand{\spikelike}{spike-like $3$-separator}

\newcommand{\tvamoslike}{twisted-cube-like $3$-separator}
\newcommand{\aug}{augmented}
\newcommand{\auging}{augmentation}
\newcommand{\planespider}{\aug\ quad $3$-separator}
\newcommand{\psep}{particular $3$-separator}
\newcommand{\Psep}{Particular $3$-separator}
\newcommand{\augsep}{\aug\ $3$-separator}
\newcommand{\prob}{problematic}
\newcommand{\tcn}{nest of twisted cubes}
\newcommand{\quadflower}{quad-flower}

\newenvironment{slproof}[1][Subproof]{\begin{proof}[#1]}{\end{proof}}

\crefformat{sublemma}{#2\rm{#1}#3}
\crefrangeformat{sublemma}{#3\rm{#1}#4--#5\rm{#2}#6}
\crefmultiformat{sublemma}{#2\rm{#1}#3}{ and~#2\rm{#1}#3}{, #2\rm{#1}#3}{ and~#2\rm{#1}#3}
\crefrangeformat{subsublemma}{#3\rm{#1}#4--#5\rm{#2}#6}
\crefmultiformat{subsublemma}{#2\rm{#1}#3}{ and~#2\rm{#1}#3}{, #2\rm{#1}#3}{ and~#2\rm{#1}#3}
\crefrangeformat{enumi}{#3\rm{#1}#4--#5\rm{#2}#6}
\crefname{enumi}{}{}
\crefname{sublemma}{}{}
\Crefname{sublemma}{Claim}{Claims}

\setenumerate{label=\rm(\roman*),midpenalty=2}

\begin{document}

\title[$N$-detachable pairs III: the theorem]{$N$-detachable pairs in $3$-connected matroids III: the theorem}

\thanks{The authors were supported by the New Zealand Marsden Fund.}

\author{Nick Brettell \and Geoff Whittle \and Alan Williams}
\address{School of Mathematics and Statistics, Victoria University of Wellington, New Zealand}
\email{nbrettell@gmail.com}
\email{geoff.whittle@vuw.ac.nz}
\email{ayedwilliams@gmail.com}


\date{\today}

\maketitle

\begin{abstract}
  Let $M$ be a $3$-connected matroid, and let $N$ be a $3$-connected minor of $M$.
  A pair $\{x_1,x_2\} \subseteq E(M)$ is \emph{$N$-detachable} if one of the matroids $M/x_1/x_2$ or $M \backslash x_1 \backslash x_2$ is $3$-connected and has an $N$-minor.
  This is the third and final paper in a series where we prove that if $|E(M)|-|E(N)| \ge 10$, then either $M$ has an $N$-detachable pair after possibly performing a single $\Delta$-$Y$ or $Y$-$\Delta$ exchange, or $M$ is essentially $N$ with a spike attached.
  Moreover, we describe the additional structures that arise if we require only that $|E(M)|-|E(N)| \ge 5$.
\end{abstract}

\pagenumbering{arabic}

\section{Introduction}

Let $M$ be a $3$-connected matroid, and let $N$ be a $3$-connected minor of $M$.
We say that a matroid \emph{has an $N$-minor} if it has an isomorphic copy of $N$ as a minor.
A pair $\{x_1,x_2\} \subseteq E(M)$ is \emph{$N$-detachable} if either 
\begin{enumerate}[label=\rm(\alph*)]
  \item $M/x_1/x_2$ is $3$-connected and has an $N$-minor, or
  \item $M \ba x_1 \ba x_2$ is $3$-connected and has an $N$-minor.
\end{enumerate}

This is the third paper in a series~\cite{paper1,paper2} where we describe the structures that arise when it is not possible to find an $N$-detachable pair in $M$.
We prove the following theorem:

\begin{theorem}
    \label{maintheoremdetailed}
    Let $M$ be a $3$-connected matroid with $|E(M)| \ge 13$, and let $N$ be a $3$-connected minor of $M$ such that $|E(N)| \ge 4$, and $|E(M)|-|E(N)| \ge 5$.
    Then either
    \begin{enumerate}
      \item $M$ has an $N$-detachable pair, 
      \item there is a matroid $M'$ obtained by performing a single $\Delta$-$Y$ or $Y$-$\Delta$ exchange on $M$ such that $M'$ has 
        an $N$-detachable pair, or
      \item there exists $P \subseteq E(M)$ such that at most one element of $E(M)-E(N)$ is not in $P$, and $P$ is either 
        \begin{enumerate}[label=\rm(\alph*)]
          \item an augmented quad $3$-separator, wherein $|E(M)|-|E(N)| = 5$,
          \item a \tvamoslike\ of $M$ or $M^*$, in which case $|E(M)|-|E(N)| = 5$,
          \item a \twisted, in which case $|E(M)|-|E(N)| \le 7$, 
          \item an \pspider, wherein $|E(M)|-|E(N)| \le 7$,
          \item a \spider, in which case $|E(M)|-|E(N)| \le 9$, or
          \item a \spikelike. 
        \end{enumerate}
    \end{enumerate}
\end{theorem}

\noindent
The $3$-separators that appear in \cref{maintheoremdetailed}(iii) are defined momentarily, while $\Delta$-$Y$ and $Y$-$\Delta$ exchange are defined in \cref{secprelims}.

By restricting attention to the case where $|E(M)|-|E(N)| \ge 10$, \cref{maintheoremdetailed} implies the theorem stated in the first paper of the series \cite[Theorem~1.1]{paper1}.
The motivation for proving such a theorem comes from a desire to find exact excluded-minor characterisations of certain minor-closed classes of representable matroids (for more details, see the introduction to \cite{paper1}).

The structure of this paper is as follows.  In \cref{secprelims}, we review relevant preliminaries regarding connectivity and keeping an $N$-minor, and results from the first two papers in the series. 
In \cref{secsetup} we lay the ground work towards a proof of \cref{maintheoremdetailed}.
The main hurdle still to overcome is addressed in \cref{sechard}; there, we consider the case where every pair whose deletion or contraction keeps an $N$-minor results in a non-trivial series or parallel class, respectively.
Throughout the first two papers, we have built up a collection of \psep s that can appear in a matroid with no $N$-detachable pairs.
In \cref{secpseps}, we show that when $M$ has one of these particular $3$-separators, $P$, and no $N$-detachable pairs, then at most one element of $M$ is not in $E(N) \cup P$.
Finally, we put everything together in \cref{finalendgamesec} to prove \cref{usefulone}, which implies \cref{maintheoremdetailed}.

We close this series of papers with some final remarks in \cref{secend}.
There, we address the natural question of what additional structures might arise if we do not allow the single $\Delta$-$Y$ or $Y$-$\Delta$ exchange.
We also present an analogue of \cref{maintheoremdetailed} for graphic matroids, and discuss the implications for simple $3$-connected graphs.

\subsection*{\Psep s}
\label{secobstrs}

\begin{figure}[bp]
    \centering
    \begin{tikzpicture}[rotate=90,scale=0.8,line width=1pt]
      \tikzset{VertexStyle/.append style = {minimum height=5,minimum width=5}}
      \clip (-2.5,-6) rectangle (3.0,2);
      \node at (-1,-1.4) {\large$E-P$};
      \draw (0,0) .. controls (-3,2) and (-3.5,-2) .. (0,-4);
      \draw (0,0) -- (2.5,0.5);
      \draw (0,0) -- (2.25,-0.75);
      \draw (0,0) -- (2,-2);
      \draw (0,0) -- (1,-3);

      \SetVertexNoLabel
      \Vertex[x=1.25,y=0.25,LabelOut=true,L=$q_3$,Lpos=180]{c1}
      \Vertex[x=2.25,y=-0.75,LabelOut=true,L=$q_2$,Lpos=90]{c2}
      \Vertex[x=2.5,y=0.5,LabelOut=true,L=$q_1$,Lpos=180]{c3}
      \Vertex[x=1.5,y=-0.5,LabelOut=true,L=$q_4$,Lpos=135]{c4}

      \Vertex[x=1,y=-1,LabelOut=true,L=$q_1$,Lpos=180]{c5}
      \Vertex[x=2,y=-2,LabelOut=true,L=$q_1$,Lpos=180]{c6}

      \Vertex[x=1,y=-3,LabelOut=true,L=$q_4$,Lpos=135]{c7}
      \Vertex[x=0.67,y=-2,LabelOut=true,L=$q_4$,Lpos=135]{c8}

      \draw (0,0) -- (0,-4);

    \end{tikzpicture}
    \caption{An example of a \spikelike~$P$, 
    which can appear in a matroid $M$ that does not contain any $N$-detachable pairs, even when $|E(M)|-|E(N)| > 9$.}
    \label{figspikelike}
\end{figure}
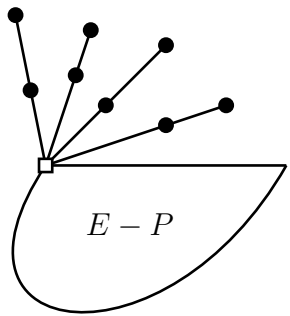

Let $M$ be a matroid with ground set~$E$.
We say that a $4$-element set $Q \subseteq E$ is a \emph{quad} if it is both a circuit and a cocircuit of $M$.
We can also view $Q$ as a \psep\ that can give rise to matroids with no $N$-detachable pairs.
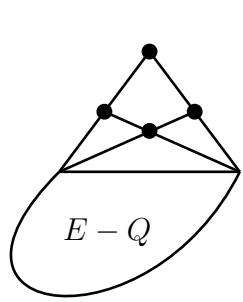
For this reason, we also sometimes refer to $Q$ as a \emph{quad $3$-separator}.

\begin{definition}
\label{def-spike-like3}
Let $P \subseteq E$ be a $2t$-element, rank-$(t+1)$, corank-$(t+1)$, exactly $3$-separating set of $M$, for some $t \ge 3$.
If there exists a partition $\{L_1,\dotsc,L_t\}$ of $P$ such that 
\begin{enumerate}[label=\rm(\alph*)]
  \item $|L_i|=2$ for each $i\in \{1,2,\dotsc,t\}$, and
  \item $L_i\cup L_j$ is a quad for all distinct $i,j\in \{1,2,\dotsc,t\}$, 
\end{enumerate}
then $P$ is a \emph{\spikelike} of $M$.
For each $i \in \{1,\dotsc,t\}$, we say $L_i$ is a \emph{leg} of $P$.
\end{definition}

\begin{definition}
\label{def-twisted3}
  Let $P \subseteq E$ be a $6$-element, rank-$4$, corank-$4$, exactly 3-separating set of $M$. If there exists a labelling $\{s_1,s_2,t_1,t_2,u_1,u_2\}$ of $P$ such that
  \begin{enumerate}[label=\rm(\alph*)]
    \item $\{s_1,s_2,t_2,u_1\}$, $\{s_1,t_1,t_2,u_2\}$, and $\{s_2,t_1,u_1,u_2\}$ are the circuits of $M$ contained in $P$, and
    \item $\{s_1,s_2,t_1,t_2\}$, $\{s_1,s_2,u_1,u_2\}$, and $\{t_1,t_2,u_1,u_2\}$ are the cocircuits of $M$ contained in $P$, 
  \end{enumerate}
  then $P$ is a \emph{\twisted} of $M$.
\end{definition}

\begin{definition}
\label{def-pspider3}
  Let $P \subseteq E$ be a $6$-element, rank-$4$, corank-$4$, exactly $3$-separating set such that $P = Q \cup \{p_1,p_2\}$ and $Q$ is a quad.  If there exists a labelling $\{q_1,q_2,q_3,q_4\}$ of $Q$ such that
  \begin{enumerate}[label=\rm(\alph*)]
    \item $\{p_1,p_2,q_1,q_2\}$, $\{p_1,p_2,q_3,q_4\}$, and $Q$ are the circuits of $M$ contained in $P$, and
    \item $\{p_1,p_2,q_1,q_3\}$, $\{p_1,p_2,q_2,q_4\}$, and $Q$ are the cocircuits of $M$ contained in $P$,
  \end{enumerate}
  then $P$ is an \emph{\pspider} of $M$. 
\end{definition}

\begin{definition}
  \label{def-spider-like3}
  Let $P \subseteq E$ be a rank-$5$, corank-$5$, exactly $3$-separating set such that $P = Q_1 \cup Q_2$ where $Q_1$ and $Q_2$ are disjoint quads of $M$.
  If there exist labellings $\{p_1,p_2,p_3,p_4\}$ of $Q_1$, and $\{q_1,q_2,q_3,q_4\}$ of $Q_2$ such that
  \begin{enumerate}[label=\rm(\alph*)]
    \item $\{p_1,p_2,q_1,q_2\}, \{p_1,p_2,q_3,q_4\}, \{p_3,p_4,q_1,q_2\}$, $\{p_3,p_4,q_3,q_4\}$, $Q_1$, and $Q_2$ are the circuits of $M$ contained in $P$, and
    \item $\{p_1,p_3,q_1,q_3\}, \{p_1,p_3,q_2,q_4\}, \{p_2,p_4,q_1,q_3\}$, $\{p_2,p_4,q_2,q_4\}$, $Q_1$, and $Q_2$ are the cocircuits of $M$ contained in $P$, 
  \end{enumerate}
  then $P$ is a \emph{\spider} with {\em associated partition} $\{Q_1,Q_2\}$.
\end{definition}

\begin{figure}[bp]
  \begin{subfigure}{0.49\textwidth}
    \centering
    \begin{tikzpicture}[rotate=90,scale=0.8,line width=1pt]
      \tikzset{VertexStyle/.append style = {minimum height=5,minimum width=5}}
      \clip (-2.5,2) rectangle (3.0,-6);
      \node at (-1,-1.8) {\large$E-Q$};
      \draw (0,-1) .. controls (-2.5,1.5) and (-3.0,-2.5) .. (0,-4);
      \draw (0,-1) -- (2,-2.5) -- (0,-4);
      \draw (1,-1.75) -- (0,-4);
      \draw (1,-3.25) -- (0,-1);

      \draw (0,-1) -- (0,-4);

      \SetVertexNoLabel

      \Vertex[x=2,y=-2.5]{a2}
      \Vertex[x=0.68,y=-2.5]{a3}
      \Vertex[x=1,y=-3.25]{a4}
      \Vertex[x=1,y=-1.75]{a5}

      \tikzset{VertexStyle/.append style = {shape=rectangle,fill=white}}
    \end{tikzpicture}
    \caption{A quad $3$-separator~$Q$.}
  \end{subfigure}
  \begin{subfigure}{0.49\textwidth}
    \centering
    \begin{tikzpicture}[rotate=90,scale=0.64,line width=1pt]
      \tikzset{VertexStyle/.append style = {minimum height=5,minimum width=5}}
      \clip (-2.5,-6) rectangle (4.4,2);
      \node at (-1,-1.4) {\large$E-P$};
      \draw (0,0) .. controls (-3,2) and (-3.5,-2) .. (0,-4);
      \draw (0,0) -- (4.0,0.9);
      \draw (0,-2) -- (2.5,-2.2); 
      \draw (0,-4) -- (3.8,-4.9); 

      \Vertex[x=3.0,y=0.67,LabelOut=true,Lpos=180,L=$s_1$]{a2}
      \Vertex[x=2.0,y=0.45,LabelOut=true,Lpos=180,L=$s_2$]{a3}

      \Vertex[x=2.5,y=-2.2,LabelOut=true,Lpos=90,L=$t_2$]{b1}
      \Vertex[x=0.64,y=-2.056,LabelOut=true,Lpos=-45,L=$t_1$]{b2}

      \Vertex[x=3.8,y=-4.9,LabelOut=true,L=$u_1$]{c1}
      \Vertex[x=2.8,y=-4.67,LabelOut=true,L=$u_2$]{c2}

      \draw[dashed] (3.8,-4.9) .. controls (2.0,-2) .. (4.0,0.9);
      \draw[dashed] (2.8,-4.67) .. controls (1.0,-2) .. (3.0,0.67);
      \draw[dashed] (1.8,-4.45) .. controls (0.25,-2) .. (2.0,0.45);

      \draw (0,0) -- (0,-4);

      \SetVertexNoLabel
      \tikzset{VertexStyle/.append style = {shape=rectangle,fill=white}}
      \Vertex[x=4.0,y=0.9]{a1}
      \Vertex[x=1.5,y=-2.12]{b3}
      \Vertex[x=1.8,y=-4.45]{c3}

    \end{tikzpicture}
    \caption{A \twisted.}
  \end{subfigure}
  \begin{subfigure}{0.49\textwidth}
    \centering
    \begin{tikzpicture}[rotate=90,scale=0.8,line width=1pt]
      \tikzset{VertexStyle/.append style = {minimum height=5,minimum width=5}}
      \clip (-2.5,2) rectangle (3.0,-6);
      \node at (-1,-1.4) {\large$E-P$};
      \draw (0,0) .. controls (-3,2) and (-3.5,-2) .. (0,-4);
      \draw (0,0) -- (2,-2) -- (0,-4);
      \draw (0,0) -- (2.5,0.5) -- (2,-2);
      \draw (0,0) -- (2.25,-0.75);
      \draw (2,-2) -- (1.25,0.25);

      \Vertex[x=1.25,y=0.25,LabelOut=true,L=$q_3$,Lpos=180]{c1}
      \Vertex[x=2.25,y=-0.75,LabelOut=true,L=$q_2$,Lpos=90]{c2}
      \Vertex[x=2.5,y=0.5,LabelOut=true,L=$q_1$,Lpos=180]{c3}
      \Vertex[x=1.5,y=-0.5,LabelOut=true,L=$q_4$,Lpos=135]{c4}
      \Vertex[x=1.33,y=-2.67,LabelOut=true,L=$p_1$,Lpos=45]{c5}
      \Vertex[x=0.67,y=-3.33,LabelOut=true,L=$p_2$,Lpos=45]{c6}

      \draw (0,0) -- (0,-4);

      \SetVertexNoLabel
      \tikzset{VertexStyle/.append style = {shape=rectangle,fill=white}}
    \end{tikzpicture}
    \caption{An \pspider.}
  \end{subfigure}
  \begin{subfigure}{0.49\textwidth}
    \centering
    \begin{tikzpicture}[rotate=90,scale=0.8,line width=1pt]
      \tikzset{VertexStyle/.append style = {minimum height=5,minimum width=5}}
      \clip (-2.5,-6) rectangle (3.0,2);
      \node at (-1,-1.4) {\large$E-P$};
      \draw (0,0) .. controls (-3,2) and (-3.5,-2) .. (0,-4);

      \draw (0,0) -- (2,-2) -- (0,-4);

      \draw (0,0) -- (2.5,0.5) -- (2,-2);
      \draw (0,0) -- (2.25,-0.75);
      \draw (2,-2) -- (1.25,0.25);

      \draw (0,-4) -- (2.5,-4.5) -- (2,-2);
      \draw (0,-4) -- (2.25,-3.25);
      \draw (2,-2) -- (1.25,-4.25);

      \Vertex[x=1.25,y=0.25,LabelOut=true,L=$q_3$,Lpos=180]{c1}
      \Vertex[x=2.25,y=-0.75,LabelOut=true,L=$q_2$,Lpos=90]{c2}
      \Vertex[x=2.5,y=0.5,LabelOut=true,L=$q_1$,Lpos=180]{c3}
      \Vertex[x=1.5,y=-0.5,LabelOut=true,L=$q_4$,Lpos=135]{c4}

      \Vertex[x=1.25,y=-4.25,LabelOut=true,L=$p_3$]{c1}
      \Vertex[x=2.25,y=-3.25,LabelOut=true,L=$p_2$,Lpos=90]{c2}
      \Vertex[x=2.5,y=-4.5,LabelOut=true,L=$p_1$]{c3}
      \Vertex[x=1.5,y=-3.5,LabelOut=true,L=$p_4$,Lpos=45]{c4}

      \draw (0,0) -- (0,-4);

      \SetVertexNoLabel
      \tikzset{VertexStyle/.append style = {shape=rectangle,fill=white}}
    \end{tikzpicture}
    \caption{A \spider.}
  \end{subfigure}
  \begin{subfigure}{0.49\textwidth}
    \centering
    \begin{tikzpicture}[rotate=90,xscale=1.1,yscale=0.55,line width=1pt]
      \tikzset{VertexStyle/.append style = {minimum height=5,minimum width=5}}
      \clip (-1.5,2) rectangle (2.7,-6);
      \node at (-0.6,-1.4) {$E-P$};
      \draw (0,0) .. controls (-1.6,2) and (-2,-2) .. (0,-4);

      \draw (1.2,1) -- (0.8,-1);
      \draw (0,-4) -- (1.2,-3);

      \draw (1.2,1) -- (2.2,-1) -- (1.2,-3);
      \draw (2.2,-1) -- (2.0,-3);

      \draw (1.2,1) -- (1.2,-3);


      \draw[white,line width=5pt] (0.8,-1) -- (2.0,-3) -- (0.8,-5);
      \draw[white,line width=5pt] (0.8,-1) -- (0.8,-5);
      \draw[white,line width=5pt] (0.8,-1) -- (0,0);
      \draw (0.8,-1) -- (2.0,-3) -- (0.8,-5);
      \draw (0.8,-1) -- (0.8,-5);
      \draw (1.2,-3) -- (0.8,-5) -- (0,-4);
      \draw (0.8,-1) -- (0,0);
      \draw (0,0) -- (0,-4);
      \draw (0,0) -- (1.2,1);


      \Vertex[x=1.2,y=1,LabelOut=true,L=$q_2$,Lpos=180]{c1}
      \Vertex[x=1.2,y=-3,LabelOut=true,L=$p_2$,Lpos=225]{c4}
      \Vertex[x=0.8,y=-1,LabelOut=true,L=$q_1$,Lpos=-45]{c5}
      \Vertex[x=0.8,y=-5,LabelOut=true,L=$p_1$,Lpos=0]{c6}
      \Vertex[x=2.0,y=-3,LabelOut=true,L=$s_1$,Lpos=45]{c5}
      \Vertex[x=2.2,y=-1,LabelOut=true,L=$s_2$,Lpos=45]{c6}

    \end{tikzpicture}
    \caption{A \tvamoslike\ of $M$.}
    \label{tvamossubfig}
  \end{subfigure}
  \begin{subfigure}{0.50\textwidth}
    \centering
    \begin{tikzpicture}[rotate=90,scale=0.59,line width=1pt]
      \tikzset{VertexStyle/.append style = {minimum height=5,minimum width=5}}
      \clip (-2.78,-6) rectangle (5.0,2);
      \node at (-1,-1.4) {$E-P$};
      \draw (0,0) .. controls (-3,2) and (-3.5,-2) .. (0,-4);
      \draw (0,0) -- (4.0,0.9);
      \draw (0,-2) -- (2.5,-2.2); 
      \draw (0,-4) -- (3.8,-4.9); 

      \Vertex[x=2.0,y=0.45,LabelOut=true,Lpos=180,L=$p_2$]{a3}

      \Vertex[x=2.5,y=-2.2,LabelOut=true,Lpos=90,L=$q_1$]{b1}
      \Vertex[x=0.61,y=-2.05,LabelOut=true,Lpos=35,L=$q_2$]{b2}

      \Vertex[x=3.00,y=-4.7,LabelOut=true,L=$s_1$]{c1}
      \Vertex[x=2.15,y=-4.5,LabelOut=true,L=$s_2$]{c2}

      \Vertex[x=4.0,y=0.9,LabelOut=true,Lpos=180,L=$p_1$]{a1}

      \draw[dashed] (3.8,-4.9) .. controls (2.0,-2) .. (4.0,0.9);
      \draw[dashed] (1.3,-4.3) .. controls (0.25,-2) .. (2.0,0.45);

      \draw (0,0) -- (0,-4);

      \SetVertexNoLabel
      \tikzset{VertexStyle/.append style = {shape=rectangle,fill=white}}
      \Vertex[x=3.8,y=-4.9]{c1}
      \Vertex[x=1.3,y=-4.3]{c3}

    \end{tikzpicture}
    \caption{A \tvamoslike\ of $M^*$.}
  \end{subfigure}
  \caption{A quad $3$-separator~$Q$, and each remaining \psep~$P$ that can appear in a matroid $M$ with no $N$-detachable pairs, where $|E(M)|-|E(N)| \ge 5$} 
  \label{figobstrs}
\end{figure}

  \begin{definition}
  Let $P \subseteq E$ be a rank-$4$, corank-$4$, exactly $3$-separating set with $P=\{p_1,p_2,q_1,q_2,s_1,s_2\}$.
  Suppose that
  \begin{enumerate}[label=\rm(\alph*)]
    \item $\{p_1,p_2,s_1,s_2\}$, $\{q_1,q_2,s_1,s_2\}$, and $\{p_1,p_2,q_1,q_2\}$ are the circuits of $M$ contained in $P$, and
    \item $\{p_1,q_1,s_1,s_2\}$, $\{p_2,q_2,s_1,s_2\}$, $\{p_1,p_2,q_1,q_2,s_1\}$, and $\{p_1,p_2,q_1,q_2,s_2\}$ are the cocircuits of $M$ contained in $P$.
  \end{enumerate}
  Then $P$ is a \emph{\tvamoslike} of $M$. 
  \end{definition}

  Let $M$ be a $3$-connected matroid with a $3$-connected minor~$N$ such that $|E(M)|-|E(N)| \ge 5$, and $M$ has no $N$-detachable pairs.
  For any \psep~$P$ as just defined, $P$ can appear in $M$, and have the property that at most one element of $M$ is not in $E(N) \cup P$.
Examples of such matroids with $E(M) = E(N) \cup P$ were given in \cite{paper1,paper2}.
There also exist such $M$, $N$, and $P$ where 
a single element of $E(M)$ is not in $E(N) \cup P$.
Before giving an example, we require the following definitions.


Let $P$ be a \psep. If there exists some $z \in E(M)-P$ such that $P \cup z$ is exactly $3$-separating, then we say that $P \cup z$ is an \emph{\augsep}.
In this case, either $z \in \cl(P)-P$ or $z \in \cocl(P)-P$ (due to the soon-to-appear \cref{gutsnew}).
We also say that an \augsep~$P \cup z$ is an \emph{\auging\ of $P$}.

When $M$ has no $N$-detachable pairs and contains an \augsep~$P \cup z$, it is possible that $z$, when removed in the correct way, preserves the $N$-minor.
A careful reader might have wondered why we have included the quad as a \psep, given that we are only interested in the case where $|E(M)|-|E(N)| \ge 5$. The reason is that an augmented quad $3$-separator $Q \cup z$ can appear in a matroid with no $N$-detachable pairs, where $E(M) = E(N) \cup (Q \cup z)$, in which case it is possible that $|E(M)|-|E(N)| = 5$.
We now give an example of such a matroid $M$, where $N$ is the Fano matroid.

Let $F_7$ be a copy of the Fano matroid with a triangle $\{a,b,c\}$, and let $F_7^-$ be a copy of the non-Fano matroid with a triangle $\{x,y,z\}$, where the ground sets of $F_7$ and $F_7^-$ are disjoint.
Let $Q = E(F_7^-)-\{x,y,z\}$.
Let $(F_7^-)'$ be the matroid obtained by freely adding the element $a$ on the line $\{x,y,z\}$.  Let $M' = (F_7^-)' \oplus_2 F_7$, where $a$ is the basepoint of this $2$-sum.
Finally, we obtain $M$ by freely adding an element $h$ on the line $\{x,c\}$, then deleting $\{x,y,z\}$.
Note that $M$ is $3$-connected, $Q \cup h$ is an \aug\ quad $3$-separator of $M$ with $h \in \cocl(Q)$, and $M$ has no $F_7$-detachable pairs.
In particular, $M \ba h$ has an $F_7$-minor, but $M/h$ does not.

In a similar manner, one can obtain \auging s of the other \psep s, except the \tvamoslike.  For a \tvamoslike~$P$, labelled as in \cref{tvamossubfig} say, the existence of an element $e \in \cl(P)-P$ for which $M /e$ has an $N$-minor ensures that a detachable pair like $\{q_1,p_2\}$ or $\{q_2,p_1\}$ is also $N$-detachable (and similarly for $e \in \cocl(P)-P$).



\section{Preliminaries}
\label{secprelims}

The notation and terminology in the paper follow Oxley~\cite{oxbook}.
For a set~$X$ and element~$e$, we write $X \cup e$ instead of $X \cup \{e\}$, and $X-e$ instead of $X-\{e\}$.
We say that $X$ meets $Y$ if $X \cap Y \neq \emptyset$.
We denote $\{1,2,\dotsc,n\}$ by $\seq{n}$.
The phrase ``by orthogonality'' refers to the fact that a circuit and a cocircuit cannot intersect in exactly one element.

Let $M$ be a matroid. 
For $P \subseteq E(M)$, we say that $P$ is a \emph{$5$-element plane} if $M|P \cong U_{3,5}$.  We also say $P$ is a \emph{$5$-element coplane} if $M^*|P \cong U_{3,5}$.
The following two lemmas are straightforward consequences of orthogonality, and are used freely without reference.

\begin{lemma}
  Suppose $M$ is a matroid with a $5$-element plane $P$.
  If $C^*$ is a cocircuit of $M$ that meets $P$, then $|P \cap C^*| \ge 3$.
\end{lemma}

\begin{lemma}
\label{swapSepSides3}
Let $(X,\{e\},Y)$ be a partition of the ground set of a matroid~$M$.
Then $e \in \cl(X)$ if and only if $e \notin \cl^{*}(Y)$.
\end{lemma}

\subsection*{Connectivity}

Let $M$ be a matroid with ground set $E$.  The \emph{connectivity function} of $M$, denoted by $\lambda_M$, is defined as follows, for all subsets $X$ of \nopagebreak $E$:
\begin{align*}
  \lambda_M(X) = r(X) + r(E - X) - r(M).
\end{align*}
A subset $X$ or a partition $(X, E-X)$ of $E$ is \emph{$k$-separating} if $\lambda_M(X) \leq k-1$.
A $k$-separating partition $(X,E-X)$ is a \emph{$k$-separation} if $|X| \ge k$ and $|E-X|\ge k$.
A $k$-separating set $X$, or a $k$-separating partition $(X,E-X)$, 
is \emph{exact} if $\lambda_M(X) = k-1$.
A matroid is \emph{$n$-connected} if it has no $k$-separations for all $k < n$.
When a matroid is $2$-connected, we simply say it is \emph{connected}.

The connectivity functions of a matroid and its dual are equal; that is,
$\lambda_M(X) = \lambda_{M^*}(X)$.  In fact, it is easily shown that
\begin{align*}
  \lambda_M(X) = r(X) + r^*(X) - |X|.
\end{align*}
The next lemma is a consequence of the fact that the connectivity function is submodular.  We write ``by uncrossing'' to refer to an application of this lemma.

\begin{lemma}
\label{onetrick3}
Let $M$ be a $3$-connected matroid, and let $X$ and $Y$ be $3$-separating subsets of $E$.
\begin{enumerate}
\item If $|X \cap Y| \ge 2$, then $X \cup Y$ is $3$-separating.
\item If $|E - (X \cup Y)| \ge 2$, then $X \cap Y$ is $3$-separating.
\end{enumerate}
\end{lemma}



The following is well known.

\begin{lemma}
  \label{gutsnew}
  Let $M$ be a $3$-connected matroid, and let $(X,\{e\},Y)$ be a partition of $E$ such that $X$ is exactly $3$-separating.
  Then 
  \begin{enumerate}[label=\rm(\roman*)]
    \item $X \cup e$ is $3$-separating if and only if $e \in \cl(X)$ or $e \in \cocl(X)$, and
    \item $X \cup e$ is exactly $3$-separating if and only if either $e \in \cl(X) \cap \cl(Y)$ or $e \in \cocl(X) \cap \cocl(Y)$.
  \end{enumerate}
\end{lemma}


\noindent
When \cref{gutsnew}(ii) holds,
we say $e$ is a \emph{guts element} if $e \in \cl(X) \cap \cl(Y)$, and $e$ is a \emph{coguts element} if $e \in \cocl(X) \cap \cocl(Y)$.


\begin{lemma}
  \label{gutsstayguts3}
  Let $M$ be a $3$-connected matroid, let $(X,E-X)$ be a $3$-separation of $M$, and let $e \in E-X$.
  Then $e \notin \cl(X) \cap \cocl(X)$.
\end{lemma}

A $3$-separation $(X, E-X)$ of a matroid $M$ with ground set $E$ is \emph{vertical} if $r(X) \ge 3$ and $r(E-X) \ge 3$.
We also say a partition $(X, \{z\}, Y)$ of $E$ is a \emph{vertical $3$-separation} when $(X \cup z, Y)$ and $(X, Y \cup z)$ are both vertical $3$-separations and $z \in \cl(X) \cap \cl(Y)$.
%

A vertical $3$-separation in $M^*$ is known as a cyclic $3$-separation in $M$.
More specifically, a $3$-separation $(X, E-X)$ of $M$ is \emph{cyclic} if $r^*(X) \ge 3$ and $r^*(E-X) \ge 3$; or, equivalently, if $X$ and $E-X$ contain circuits.
We also say that a partition $(X, \{z\}, Y)$ of $E$ is a \emph{cyclic $3$-separation} if $(X, \{z\}, Y)$ is a vertical $3$-separation in $M^*$.


\begin{lemma}[{\cite[Lemma~3.5]{stabilizers}}]
\label{openVertSep2}
Let $M$ be a $3$-connected matroid and let $z \in E$.  
The following are equivalent:
\begin{enumerate}
    \item $M$ has a vertical $3$-separation $(X, \{z\}, Y)$.\label{vsi2}
    \item $\si(M/z)$ is not $3$-connected.\label{vsii2}
\end{enumerate}
\end{lemma}

We say that a partition $(X_1,X_2,\dotsc,X_m)$ of $E$ is a \emph{path of $k$-separations} if $(X_1 \cup \dotsm \cup X_i, X_{i+1} \cup \dotsm \cup X_m)$ is a $k$-separation for each $i \in \seq{m-1}$.
Observe that a vertical, or cyclic, $3$-separation $(X, \{z\},Y)$ is an instance of a path of $3$-separations.

We refer to the following as Bixby's Lemma.

\begin{lemma}[Bixby's Lemma~\cite{bixby}]
\label{bixbyL2}
Let $e$ be an element of a $3$-connected matroid $M$.
Then either $\si(M/e)$ is $3$-connected, or $\co(M\ba e)$ is $3$-connected.
\end{lemma}

The next lemma~\cite{stabilizers} shows, in particular, that an element that is in a quad but not in a triangle (or a triad) can be contracted (or deleted, respectively) without destroying $3$-connectivity.

\begin{lemma}
  \label{r3cocircsi3}
  Let $C^*$ be a rank-$3$ cocircuit of a $3$-connected matroid $M$.
If $x \in C^*$ has the property that $\cl_M(C^*)-x$ contains a triangle of $M/x$, then $\si(M/x)$ is $3$-connected.
\end{lemma}


If, rather than a quad, we have a $4$-element independent cocircuit, then the following lemma applies.
The proof of this lemma is similar to \cite[Lemma~4.5]{paper1}; we omit the details.

\begin{lemma}
  \label{r4cocirc3b}
  Let $M$ be a $3$-connected matroid and let $C^* = \{c_1,c_2,d_1,d_2\}$ be a $4$-element cocircuit such that $M \ba d_1$ and $M \ba d_2$ are $3$-connected, and neither $c_1$ nor $c_2$ is in a triangle. Then $M/c_i$ is $3$-connected for some $i \in \{1,2\}$.
\end{lemma}

We also require the following result, proved in \cite{paper1}. 

\begin{lemma}[{\cite[Lemma~4.4]{paper1}}]
    \label{6pointplane2}
    Let $M$ be a $3$-connected matroid with $P \subseteq E$ such that $M|P \cong U_{3,6}$, and $X \subseteq P$ such that $|X|=4$. 
    Suppose that $\cl(P)$ contains no triangles.
    Then there are distinct elements $x_1,x_2 \in X$ such that $M\ba x_1\ba x_2$ is $3$-connected.
  \end{lemma}

\subsection*{Retaining an $N$-minor}

\begin{lemma}
  \label{m2.73}
  Let $(S,T)$ be a $2$-separation of a connected matroid $M$ and let $N$ be a $3$-connected minor of $M$.  Then $\{S, T\}$ has a member $U$ such that $|U \cap E(N)| \leq 1$.  Moreover, if $u \in U$, then
  \begin{enumerate}
    \item $M/u$ has an $N$-minor if $M/u$ is connected, and
    \item $M \backslash u$ has an $N$-minor if $M \backslash u$ is connected.
  \end{enumerate}
\end{lemma}

For a matroid $M$ with a minor $N$ and $e \in E(M)$, we say $e$ is \emph{$N$-contractible} if $M/e$ has an $N$-minor, we say $e$ is \emph{$N$-deletable} if $M \ba e$ has an $N$-minor, and we say $e$ is \emph{$N$-flexible} if $e$ is both $N$-contractible and $N$-deletable.

The dual of the following is proved in \cite{bs2014,ben}.

\begin{lemma}
  \label{doublylabel3}
  Let $N$ be a $3$-connected minor of a $3$-connected matroid $M$. Let $(X, \{z\}, Y)$ be a cyclic $3$-separation of $M$ such that $M\ba z$ has an $N$-minor with $|X \cap E(N)| \le 1$. Let $X' = X-\cocl(Y)$
  and $Y' = \cocl(Y) - z$.
  Then
  \begin{enumerate}
    \item each element of $X'$ is $N$-deletable; and
    \item at most one element of $\cocl(X)-z$ is not $N$-contractible, and if such an element~$x$ exists, then $x \in X' \cap \cl(Y')$ and $z \in \cocl(X' - x)$.
  \end{enumerate}
\end{lemma}

Suppose $C$ and $D$ are disjoint subsets of $E(M)$ such that $M/C \ba D \cong N$.
We call the ordered pair $(C,D)$ an \emph{$N$-labelling of $M$}, and say that each $c \in C$ is \emph{$N$-labelled for contraction}, and each $d \in D$ is \emph{$N$-labelled for deletion}. 
We also say a set $C' \subseteq C$ is \emph{$N$-labelled for contraction}, and $D' \subseteq D$ is \emph{$N$-labelled for deletion}.
An element $e \in C \cup D$ or a set $X \subseteq C \cup D$
is \emph{$N$-labelled for removal}.

Let $(C,D)$ be an $N$-labelling of $M$, and let $c \in C$, $d \in D$, and $e \in E(M)-(C \cup D)$.
Then, we say that the ordered pair
$((C-c) \cup d,(D-d)\cup c)$ is 
obtained from $(C,D)$ by \emph{switching the $N$-labels on $c$ and $d$}.
Similarly, 
$((C-c) \cup e,D)$ (or $(C,(D-d) \cup e)$, respectively) is 
obtained from $(C,D)$ by \emph{switching the $N$-labels on $c$} (respectively, $d$) \emph{and $e$}.

The following straightforward lemma, 
which gives a sufficient condition for retaining a valid $N$-labelling after an $N$-label switch,
will be used freely.

\begin{lemma}
  Let $M$ be a $3$-connected matroid, let $N$ be a $3$-connected minor of $M$ with $|E(N)| \ge 4$, and let $(C,D)$ be an $N$-labelling of $M$. 
  Suppose $\{d,e\}$ is a parallel pair in $M/c$, for some $c \in C$.
  Let $(C',D')$ be obtained from $(C,D)$ by switching the $N$-labels on $d$ and $e$; then $(C',D')$ is an $N$-labelling.
\end{lemma}

Suppose that $(C,D)$ and $(C',D')$ are $N$-labellings of $M$, where $(C',D')$ can be obtained from $(C,D)$ by a sequence of $N$-label switches.  We say that $(C',D')$ is \emph{switching-equivalent} to $(C,D)$.

\subsection*{Delta-wye exchange}

Let $M$ be a matroid with a triangle $T=\{a,b,c\}$.
Consider a copy of $M(K_4)$ having $T$ as a triangle with $\{a',b',c'\}$ as the complementary triad labelled such that $\{a,b',c'\}$, $\{a',b,c'\}$ and $\{a',b',c\}$ are triangles.
Let $P_{T}(M,M(K_4))$ denote the generalised parallel connection of $M$ with this copy of $M(K_4)$ along the triangle $T$.
Let $M'$ be the matroid $P_{T}(M,M(K_4))\backslash T$ where the elements $a'$, $b'$ and $c'$ are relabelled as $a$, $b$ and $c$ respectively.
This matroid $M'$ 
is said to be obtained from $M$ by a \emph{\dY} on the triangle~$T$.
Dually, a matroid $M''$ is obtained from $M$ by a \emph{\Yd} on the triad $\{a,b,c\}$
if $(M'')^*$ is obtained from $M^*$ by a \dY\ on $\{a,b,c\}$.

\subsection*{\unfortunate\ triangles and triads}

  Let $M$ be a $3$-connected matroid and let $N$ be a $3$-connected minor of $M$.
  We say that a triangle or triad~$T$ of $M$ is \emph{\unfortunate} if, for all distinct $a,b \in T$, none of $M/a/b$, $M/a\ba b$, $M\ba a/b$, and $M\ba a\ba b$ have an $N$-minor.

  \begin{lemma}[{\cite[Lemma~3.1]{paper1}}]
    \label{freegrounded}
    Let $M$ be a $3$-connected matroid with a $3$-connected minor $N$ where 
    $|E(N)| \ge 4$.
    If $T$ is an \unfortunate\ triangle of $M$ with $x \in T$, 
    then $x$ is not $N$-contractible.
  \end{lemma}

  We focus on matroids where every triangle or triad is \unfortunate, due to the following:

  \begin{lemma}[{\cite[Theorem~3.2]{paper1}}]
    \label{unfortunatetri}
    Let $M$ be a $3$-connected matroid, 
    and let $N$ be a $3$-connected minor of $M$ with $|E(N)| \ge 4$, where $|E(M)|-|E(N)| \ge 5$.
    Then either
    \begin{enumerate}
      \item $M$ has an $N$-detachable pair, or
      \item there is a matroid $M'$ obtained by performing a single $\Delta$-$Y$ or $Y$-$\Delta$ exchange on $M$ such that $M'$ has 
        an $N$-detachable pair, or
      \item each triangle or triad of $M$ is \unfortunate.
    \end{enumerate}
  \end{lemma}

\section{Setup}
\label{secsetup}

We work towards a proof of \cref{maintheoremdetailed} in \cref{finalendgamesec}.
In this section, we lay the ground work.

\begin{lemma}
  \label{catalyst}
  Let $M$ be a $3$-connected matroid, and let $N$ be a $3$-connected minor of $M$, where every triangle or triad of $M$ is \unfortunate, and $|E(N)| \ge 4$.
  If $|E(M)| - |E(N)| \ge 3$, then, for some $(M_0,N_0) \in \{(M,N), (M^*,N^*)\}$, there exist elements $d,d' \in E(M_0)$ such that $M_0 \ba d$ is $3$-connected and $M_0 \ba d \ba d'$ has an $N_0$-minor.
\end{lemma}
\begin{proof}
  Suppose $M$ is a wheel or a whirl. Then every element is in a triangle and a triad.  So, by duality, we may assume that $M/x$ has an $N$-minor for some $x \in E(M)$, where $x$ is in a triangle $\{x,y,z\}$.  Now $\{y,z\}$ is a parallel pair in $M/x$; thus, as $|E(N)| \ge 4$, the matroid $M/x \ba y$ has an $N$-minor, up to an $N$-label switch.  But then the triangle $\{x,y,z\}$ is not \unfortunate; a contradiction.

  Now, by Seymour's Splitter theorem~\cite{pds}, there exists an element $d$ such that $M_0 \ba d$ is $3$-connected and has an $N_0$-minor, for some $(M_0,N_0) \in \{(M,N), (M^*,N^*)\}$.  For notational convenience, we will assume that $(M_0,N_0) = (M,N)$.
  So $M \ba d$ is $3$-connected and has an $N$-minor.  If $M \ba d \ba d'$ has an $N$-minor for any $d' \in E(M \ba d)$, then the \lcnamecref{catalyst} holds; so assume otherwise.  In particular, as $|E(M)| - |E(N)| \ge 3$, there exist distinct elements $c_1,c_2 \in E(M \ba d)$ such that $M \ba d/c_1/c_2$ has an $N$-minor.

  Suppose $M \ba d$ is a wheel or a whirl.  Then $c_1$ is in a triangle $\{c_1,t,t'\}$ of $M \ba d$.  Since $M \ba d/c_1$ has an $N$-minor and $|E(N)| \ge 4$, the matroid $M \ba d \ba t$ has an $N$-minor; a contradiction.
  So $M \ba d$ is not a wheel or a whirl.
  Thus, by Seymour's Splitter Theorem, there exists an element $c \in E(M \ba d)$ such that $M \ba d / c$ is $3$-connected and has an $N$-minor.
  As $|E(M)|-|E(N)| \ge 3$ and $M \ba d$ has no $N$-deletable elements, there exists an element $c'$ such that $M \ba d /c /c'$ has an $N$-minor.
  If $M/c$ is $3$-connected, then the \lcnamecref{catalyst} holds.
  So assume that $M /c$ is not $3$-connected.
  Since $M/c \ba d$ is $3$-connected, we deduce that $d$ is in a parallel pair in $M/c$, so $\{c,d\}$ is in a triangle of $M$.
  But $M\ba d/c$ has an $N$-minor, so this triangle is not \unfortunate; a contradiction.
\end{proof}

Suppose $M$ and $M \ba d$ are $3$-connected matroids, for some $d \in E(M)$, and let $X \subseteq E(M\ba d)$ be an exactly $3$-separating set in $M \ba d$.
We say that $d$ \emph{blocks} $X$ if $X$ is not $3$-separating in $M$. 
If $d$ blocks $X$, then $d \in \cocl(X)$. 

\begin{lemma}
  \label{smallYcase}
  Let $M$ be a $3$-connected matroid, and let $N$ be a $3$-connected minor of $M$, where every triangle or triad of $M$ is \unfortunate, and $|E(N)| \ge 4$.
  Suppose that there exist elements $d,d' \in E(M)$ such that $M \ba d$ is $3$-connected and $M \ba d \ba d'$ has an $N$-minor.
  Then either:
  \begin{enumerate}
    \item $M$ has an $N$-detachable pair,\label{endwin1}
    \item $M$ has a $4$-element cocircuit containing $\{d,d'\}$, or\label{endwin34}
    \item $M \ba d$ has a cyclic $3$-separation $(Y,\{d'\},Z)$ such that $|Y \cap E(N)| \le 1$ and $|Y| \ge 4$.\label{endwin2}
  \end{enumerate}
\end{lemma}
\begin{proof}
  If $M \ba d \ba d'$ is $3$-connected, then \cref{endwin1} holds, so assume otherwise.
  Let $(S,T)$ be a $2$-separation of $M \ba d \ba d'$.
  Suppose $\co(M \ba d \ba d')$ is $3$-connected.
  Then, without loss of generality, $S$ is contained in a series class in $M \ba d \ba d'$, and, by \cref{m2.73}, $|S \cap E(N)| \le 1$.
  Now $M \ba d \ba d' /(S-s)$ has an $N$-minor for any $s \in S$, since $|E(N)| \ge 4$.
%
  Let $S'$ be a $2$-element subset of $S$.
  Since $M \ba d$ is $3$-connected, $S' \cup d'$ is a triad of $M \ba d$.
  As $M \ba d \ba d'$ has an $N$-minor, and every triad of $M$ is \unfortunate, $d$ blocks $S' \cup d'$, so $S' \cup \{d,d'\}$ is a $4$-element cocircuit,
  in which case \cref{endwin34} holds.

  Now we may assume that $\co(M \ba d \ba d')$ is not $3$-connected.
  Then $M \ba d$ has a cyclic $3$-separation $(Y, \{d'\},Z)$ with $|Y \cap E(N)| \le 1$, by \cref{openVertSep2,m2.73}.
  Suppose that $|Y|=3$.
  Then $Y$ is a triangle, since $Y$ 
  contains a circuit. 
  But there is at most one element in $Y$ that is not $N$-contractible, by \cref{doublylabel3}, so $Y$ is not an \unfortunate\ triangle.
  Hence $|Y| \ge 4$, and thus \cref{endwin2} holds.
\end{proof}

\begin{lemma}
  \label{step}
  Let $M$ be a $3$-connected matroid, and let $N$ be a $3$-connected minor of $M$ such that $|E(N)| \ge 4$, every triangle or triad of $M$ is \unfortunate, and $|E(M)|-|E(N)| \ge 5$.
  Then either
  \begin{enumerate}
    \item $M$ has an $N$-detachable pair;\label{step1}
    \item up to replacing $(M,N)$ by $(M^*,N^*)$, there exists $d \in E(M)$ such that $M \ba d$ is $3$-connected and has a cyclic $3$-separation $(Y, \{d'\}, Z)$ with $|Y| \ge 4$, where $M \ba d \ba d'$ has an $N$-minor with $|Y \cap E(N)| \le 1$;\label{step2}
    \item up to replacing $(M,N)$ by $(M^*,N^*)$, $M$ has an \planespider~$Q \cup z$ such that $z \in \cocl(Q)$ 
      and $M \ba z$ 
      has an $N$-minor with $|Q \cap E(N)| \le 1$; or\label{step3}
    \item there is an $N$-labelling $(C,D)$ of $M$ such that, for every switching-equivalent $N$-labelling $(C',D')$,\label{step4}
      \begin{enumerate}[label=\rm(\alph*)]
        \item $M / c$ and $M \ba d$ are $3$-connected for every $c \in C'$ and $d \in D'$,
        \item each pair $\{c_1,c_2\} \subseteq C'$ is contained in a $4$-element circuit, and
        \item each pair $\{d_1,d_2\} \subseteq D'$ is contained in a $4$-element cocircuit.
      \end{enumerate}
  \end{enumerate}
\end{lemma}
\begin{proof}
  Assume that neither \cref{step1}, \cref{step2} nor \cref{step3} holds; we will show that \cref{step4} holds.
  By \cref{catalyst}, we may assume, up to replacing $M$ by $M^*$ and $N$ by $N^*$, that there is an element $d_0 \in E(M)$ such that $M \ba d_0$ is $3$-connected and $M \ba d_0$ has at least one $N$-deletable element, $d_0'$ say.
  Now, since neither \cref{step1} nor \cref{step2} holds, \cref{smallYcase} implies
  that $\{d_0,d_0'\}$ is contained in a $4$-element cocircuit~$C_0^*$.  More generally, we have
  the following:

  \begin{sublemma}
    \label{prop}
    If $\{d,d'\}$ is a pair of elements such that $M \ba d$ is $3$-connected and $M \ba d \ba d'$ has an $N$-minor,
    then $\{d,d'\}$ is contained in a $4$-element cocircuit. 
  \end{sublemma}

  We next show the following: 
  \begin{sublemma}
    \label{delisquasi}
    Let $d$ and $d'$ be distinct elements of $M$ such that $M \ba d$ is $3$-connected, and $M \ba d \ba d'$ has an $N$-minor.
    Then $M \ba d'$ is $3$-connected. 
  \end{sublemma}
  \begin{slproof}
    Suppose $M \ba d'$ is not $3$-connected.
    Since $d'$ is $N$-deletable, it is not in a triad, so $\co(M \ba d')$ is not $3$-connected.  Let $(Y,\{d'\},Z)$ be a cyclic $3$-separation of $M$.
    We may assume that $|Y \cap E(N)| \le 1$, by \cref{m2.73}.
    If $Y \cup d'$ is not coclosed, then there is some $z \in Z$ such that
    $(Y \cup z, \{d'\},Z-z)$ is a cyclic $3$-separation, and $|(Y \cup z) \cap E(N)| \le 1$ since $|E(N)| \ge 4$.
    So we may assume that $Y \cup d'$ is coclosed.
    If $Y$ contains a triangle $T$, then $t \in T$ is $N$-contractible in $M \ba d'$, by \cref{doublylabel3}, so $T$ is not an \unfortunate\ triangle; a contradiction.
    So $Y$ does not contain a triangle, which implies that $|Y| \ge 4$.
    Since $M \ba d$ is $3$-connected, $(Y,\{d'\},Z-d)$ or $(Y-d,\{d'\},Z)$ is a path of $3$-separations in $M \ba d$, where $d \in Z$ or $d \in Y$ respectively.
    As $|Z-d| \ge 2$ and $|Y-d| \ge 2$, the element $d'$ is a coguts element in either case.
    Thus, if $r^*_{M \ba d}(Z-d) \ge 3$ and $r^*_{M \ba d}(Y-d) \ge 3$, then this path of $3$-separations is a cyclic $3$-separation.

    Suppose $d \in Z$.
    As $d' \in \cocl_{M \ba d}(Z-d)$, it follows that $r^*_{M \ba d \ba d'}(Z-d) = r^*_{M \ba d}(Z-d) -1$.
    Since $M \ba d \ba d'$ has an $N$-minor with $|Y \cap E(N)| \le 1$, and $|E(N)| \ge 4$, we have $r^*_{M \ba d \ba d'}(Z-d) \ge 2$, thus
    $r^*_{M \ba d}(Z-d) \ge 3$.
    We deduce that $(Y,\{d'\},Z-d)$ is a cyclic $3$-separation.  Since $|Y| \ge 4$, \cref{step2} holds; a contradiction. 

    Now suppose $d \in Y$.
    If $r^*_{M \ba d}(Y-d) \ge 3$, then $(Y-d,\{d'\},Z)$ is a cyclic $3$-separation, and, since $Y-d$ does not contain a triangle, $|Y-d| \ge 4$, so \cref{step2} also holds in this case. 
    Suppose $r^*_{M \ba d}(Y-d) = 2$.
    Now $Y \cup d'$ is a corank-$3$ set in $M$ consisting of at least five elements.
    Let $Y' = \cocl(Y \cup d')-\{d,d'\}$.
    Since $r^*_{M \ba d \ba d'}(Y') = 1$, for any $y \in Y'$ the matroid $M\ba d \ba d'/(Y'-y)$ has an $N$-minor.
    Therefore,
    any triad contained in $\cocl(Y \cup d')$ would not be \unfortunate, so this set contains no triads.
    In particular, we observe that $Y \cup d'$ is a coplane in $M$ consisting of at least five elements.

    Suppose $|Y| \ge 5$.
    Then there exists $P \subseteq Y \cup d'$ such that $M^*|P = U_{3,6}$.
    Let $X = P-\{d,d'\}$. 
    Recall that $M / p_1 / p_2$ has an $N$-minor for all distinct $p_1,p_2 \in X$, and $\cocl(P)$ does not contain any triads.
    It follows, by \cref{6pointplane2}, that $M$ has an $N$-detachable pair.
    So we may assume that $|Y| = 4$.
    Since $Y$ does not contain any triangles or triads, $Y$ is a quad in $M$.
    So $Y \cup d'$ is an \planespider, thus \cref{step3} holds; a contradiction.
  \end{slproof}

  Let $C_0^*-\{d_0,d_0'\} = \{c_0,e\}$.
  Since $\{c_0,e\}$ is a series pair in $M \ba d_0 \ba d_0'$, both $M \ba d_0 \ba d_0'/c_0$ and $M \ba d_0 \ba d_0' / e$ have $N$-minors.
  So neither $c_0$ nor $e$ is contained in a triangle.
  Both $M \ba d_0$ and $M \ba d_0'$ are $3$-connected, by \cref{delisquasi}.
  Thus, \cref{r4cocirc3b} implies that either $M/c_0$ or $M/e$ is $3$-connected.
  Without loss of generality, $M/c_0$ is $3$-connected.

  Let $(C,D)$ be an $N$-labelling of $M$ with $d_0 \in D$ and $c_0 \in C$.
  For any $d_1 \in D-d_0$, the pair $\{d_0,d_1\}$ is contained in a $4$-element cocircuit by \cref{prop}, and $M \ba d_1$ is $3$-connected by \cref{delisquasi}.
  Now, for any $d_2 \in D-d_1$, the matroid $M \ba d_1 \ba d_2$ has an $N$-minor, so $\{d_1,d_2\}$ is contained in a $4$-element cocircuit, by \cref{prop}.
  Thus every pair of elements in $D$ is contained in a $4$-element cocircuit.
  A dual argument shows that
  $M / c$ is $3$-connected for every $c \in C$ and
  every pair $\{c_1,c_2\} \subseteq C$ is contained in a $4$-element circuit.

  Let $(C',D')$ be an $N$-labelling of $M$ that is switching-equivalent to $(C,D)$.
  It remains to show that 
  these properties also hold for $(C',D')$.
  It is sufficient to show they hold when $(C',D')$ is obtained from $(C,D)$ by a single $N$-label switch.
  Suppose the switch is on an element $d \in D$; that is, $D' = (D-d) \cup e$ for some $e \in E(M)-D$.
  Let $d' \in D-d$.
  Then $M \ba d' \ba e$ has an $N$-minor, and $M \ba d'$ is $3$-connected.
  By \cref{delisquasi}, $M \ba e$ is $3$-connected.
  So $\{d',e\}$ is contained in a $4$-element cocircuit by \cref{prop}, for every $d' \in D-d$.
  The same argument applies for an $N$-label switch on an element in $C$.
  So \cref{step4} holds, thus completing the proof.
\end{proof}

\section{When every $N$-deletable pair is in a $4$-element cocircuit} 
\label{sechard}

Next, we focus on the case where \cref{step}\cref{step4} holds. 
In this section, we prove \cref{weaktheoremdetailed}.
Before stating this \lcnamecref{weaktheoremdetailed}, we define some $12$-element matroids that can have no $N$-detachable pairs, for an appropriately chosen minor $N$.

\begin{definition}
  A matroid~$M$ is a \emph{\quadflower} if $|E(M)|=12$ and $r(M) = r^*(M) = 6$, and there is a partition $(Q_1,Q_2,Q_3)$ of $E(M)$ such that $Q_1 \cup Q_2$, $Q_1 \cup Q_3$, and $Q_2 \cup Q_3$ are \spider s.
\end{definition}

We note that a \quadflower~$M$ is $3$-connected, but neither $M \ba x \ba y$ nor $M / x / y$ is $3$-connected, for every pair $\{x,y\} \subseteq E(M)$.
Moreover, for a $3$-connected minor $N$ of $M$, it is possible that $E(N)$ meets each of the three quads that partition $E(M)$.

\begin{definition}
  A matroid~$M$ is a \emph{\tcn} if $|E(M)|=12$ and $r(M) = r^*(M) = 6$, and there exists a labelling $\bigcup_{i \in \seq{6}}\{e_i,e_i'\}$ of $E(M)$ such that 
  $\{e_1, e_2, e_3', e_4'\}$, $\{e_3, e_4, e_5', e_6'\}$, $\{e_5, e_6, e_1', e_2'\}$, \\
  $\{e_4, e_5, e_1', e_3'\}$, $\{e_1, e_3, e_6', e_2'\}$, $\{e_6, e_2, e_4', e_5'\}$, \\
  $\{e_2, e_5, e_3', e_6'\}$, $\{e_1, e_4, e_2', e_5'\}$, $\{e_3, e_6, e_1', e_4'\}$, \\
  $\{e_5, e_1, e_4', e_6'\}$, $\{e_4, e_6, e_2', e_3'\}$, $\{e_2, e_3, e_5', e_1'\}$, \\
  $\{e_2, e_4, e_6', e_1'\}$, $\{e_6, e_1, e_3', e_5'\}$, and $\{e_3, e_5, e_2', e_4'\}$
  are circuits; and \\
  $\{e_1, e_2, e_5', e_6'\}$, $\{e_3, e_4, e_1', e_2'\}$, $\{e_5, e_6, e_3', e_4'\}$, \\
  $\{e_4, e_5, e_6', e_2'\}$, $\{e_1, e_3, e_4', e_5'\}$, $\{e_6, e_2, e_1', e_3'\}$, \\
  $\{e_2, e_5, e_1', e_4'\}$, $\{e_1, e_4, e_3', e_6'\}$, $\{e_3, e_6, e_2', e_5'\}$, \\
  $\{e_5, e_1, e_2', e_3'\}$, $\{e_4, e_6, e_5', e_1'\}$, $\{e_2, e_3, e_4', e_6'\}$, \\
  $\{e_2, e_4, e_3', e_5'\}$, $\{e_6, e_1, e_2', e_4'\}$, and $\{e_3, e_5, e_6', e_1'\}$
  are cocircuits.
\end{definition}

A \tcn\ is a $12$-element matroid that, for an appropriately chosen $3$-connected minor $N$, has no $N$-detachable pairs.
Indeed, let $M$ be a \tcn, with labelling as defined above.
Then it can be readily checked that $M$ is self-dual under the isomorphism that maps $e_i$ to $e_i'$ and $e_i'$ to $e_i$, for each $i \in \seq{6}$, and $M$ has the property that $M \ba x \ba y$ is $3$-connected for distinct $x,y \in E(M)$ if and only if $\{x,y\}=\{e_i,e_i'\}$ for some $i \in \seq{6}$.
Moreover, $M$ is $4$-connected.
Let $N=U_{2,4}$.
Then $M$ has an $N$-minor, but neither $M/e_i/e_i'$ nor $M \ba e_i \ba e_i'$ has an $N$-minor for any $i \in \seq{6}$.

\begin{proposition}
  \label{weaktheoremdetailed}
  Let $M$ be a $3$-connected matroid, and let $N$ be a $3$-connected minor of $M$ such that
  $|E(N)| \ge 4$,
  every triangle or triad of $M$ is \unfortunate, and
  $|E(M)|-|E(N)| \ge 5$.
  Let $(C,D)$ be an $N$-labelling of $M$ such that, for every switching-equivalent $N$-labelling $(C',D')$,
  \begin{enumerate}[label=\rm(\alph*)]
    \item $M / c$ and $M \ba d$ are $3$-connected for every $c \in C'$ and $d \in D'$,
    \item each pair $\{c_1,c_2\} \subseteq C'$ is contained in a $4$-element circuit, and
    \item each pair $\{d_1,d_2\} \subseteq D'$ is contained in a $4$-element cocircuit.
  \end{enumerate}
  Then, up to replacing $(M,N)$ by $(M^*,N^*)$, one of the following holds:
  \begin{enumerate}
    \item there is some $P \subseteq E(M)$ such that $E(M)-E(N) \subseteq P$ and $P$ is a \spikelike, a \spider, an \pspider, or a \twisted; 
    \item there exists $d \in E(M)$ such that $M \ba d$ is $3$-connected and has a cyclic $3$-separation $(Y, \{d'\}, Z)$ with $|Y| \ge 4$, where $M \ba d \ba d'$ has an $N$-minor with $|Y \cap E(N)| \le 1$;\label{wtd2}
    \item $M$ has an \planespider~$Q \cup z$ such that $M$ has an $N$-minor with $|(Q \cup z) \cap E(N)| \le 1$; 
    \item $|E(M)| =12$, and $M$ is either a \quadflower\ or a \tcn; or
    \item $|E(M)| \leq 10$.
  \end{enumerate}
\end{proposition}

  For the entirety of this section we work under the hypotheses of \cref{weaktheoremdetailed}.
  We start with a lemma that provides sufficient conditions for (ii) to hold. 
  We will use this frequently in what follows.

\begin{lemma}
  \label{3sepwin2}
  If there exists a pair $\{d,d'\} \subseteq D'$, for some $N$-labelling $(C',D')$ that is switching-equivalent to $(C,D)$, and a set $X \subseteq E(M)-\{d,d'\}$ such that $|X| \ge 4$, $X$ is $3$-separating in $M \ba d$, $X$ contains a circuit, $|X-(C' \cup D')| \le 2$, and $d' \in \cocl_{M \ba d}(X)$; then \cref{wtd2} of \cref{weaktheoremdetailed} holds.
\end{lemma}
\begin{proof}
  Let $Y = X \cup \{d,d'\}$ and $Z = E(M)-Y$.
  We claim that $(X,\{d'\},Z)$ is a cyclic $3$-separation of $M \ba d$.
  Indeed, $X \cup d'$ is $3$-separating in $M \ba d$ since $d' \in \cocl_{M \ba d}(X)$ and $M \ba d$ is $3$-connected.
  So $(X,\{d'\},Z)$ is a path of $3$-separations in $M \ba d$ where $d'$ is a coguts element.
  Since $|E(N)| \ge 4$ and $|Y-(C' \cup D')| \le 2$, \cref{m2.73} implies that $|X \cap E(N)| \le 1$.
  If $r^*_{M \ba d}(Z) =2$, then $r^*_{M \ba d \ba d'}(Z)=1$, so $|Z \cap E(N)| \le 1$.
  But this contradicts that $|E(N)| \ge 4$.
  As $X$ contains a circuit, $(X,\{d'\},Z)$ is a cyclic $3$-separation of $M \ba d$ as claimed, so
  \cref{weaktheoremdetailed}(ii) holds. 
\end{proof}

Towards a proof of \cref{weaktheoremdetailed}, we first handle the case where an $N$-deletable or $N$-contractible pair is contained in a quad.

\begin{lemma}
  \label{wtdp1}
  Suppose $(C',D')$ is an $N$-labelling of $M$ that is switching-equivalent to $(C,D)$, and there is a pair $\{d,d'\} \subseteq D'$ or $\{c,c'\} \subseteq C'$ that is contained in a quad.
  Then, up to replacing $(M,N)$ by $(M^*,N^*)$, either
  \begin{enumerate}
    \item there exists $d_0 \in E(M)$ such that $M \ba d_0$ is $3$-connected and has a cyclic $3$-separation $(Y, \{d_0'\}, Z)$ with $|Y| \ge 4$, where $M \ba d_0 \ba d_0'$ has an $N$-minor with $|Y \cap E(N)| \le 1$;\label{wtdp1c2}
    \item $M$ has an \planespider~$Q \cup z$ such that $M$ has an $N$-minor with $|(Q \cup z) \cap E(N)| \le 1$;\label{wtdp1c3}
    \item there is some $P \subseteq E(M)$ such that $E(M)-E(N) \subseteq P$ and $P$ is a \spikelike, an \pspider, or a \spider\ of $M$; or\label{wtdp1c1}
    \item $|E(M)| \in \{10,12\}$, and there is a partition $(Q_1,Q_2,Z)$ of $E(M)$ such that $Q_1 \cup Q_2$ is a \spider, and $Q_1 \cup Z$ and $Q_2 \cup Z$ are both \spider s or \pspider s.\label{wtdp1c4}
  \end{enumerate}
\end{lemma}

\begin{proof}
  Suppose that \cref{wtdp1c2} does not hold; we will show that \cref{wtdp1c3}, \cref{wtdp1c1}, or \cref{wtdp1c4} holds.
  For notational convenience, in what follows we use $(C,D)$, rather than $(C',D')$, to refer to an $N$-labelling that is switching-equivalent to $(C,D)$.

  \begin{sublemma}
    \label{nomeet3con}
    Let $Q$ be a quad with $|Q \cap (C \cup D)| \ge 2$.
    If $e \in \cl(Q)-Q$ and $e \in C$, 
    or $e \in \cocl(Q)-Q$ and $e \in D$,
    then $Q \cup e$ is an \planespider\ such that \cref{wtdp1c3} holds.
  \end{sublemma}
  \begin{slproof}
    Let $e \in \cl(Q)-Q$ where $e$ is $N$-labelled for contraction.
    Clearly $Q \cup e$ is an \planespider.
    Since $|Q \cap (C \cup D)| \ge 2$ and $|E(N)| \ge 4$, \cref{m2.73} implies that $M/e$ has an $N$-minor such that $|Q \cap E(N)| \le 1$, so \cref{wtdp1c3} holds.
    Similarly, if some $e \in \cocl(Q)-Q$ is $N$-labelled for deletion, then $Q \cup e$ is an \planespider\ such that  \cref{wtdp1c3} holds.
  \end{slproof}

  We may assume, by taking the dual if necessary, that a pair of elements in $D$ is contained in a quad~$Q_1$.
  Suppose that $|D| = 2$.
  Let $\{d,d'\} \subseteq D \cap Q_1$.
  Since $|E(M)|-|E(N)| \ge 5$, we have $|C| \ge 3$.
  Let $c_1,c_2,c_3$ be distinct elements in $C$.
  For distinct $i,j \in [3]$, the pair $\{c_i,c_j\}$ is in a $4$-element circuit $C_{i,j}$. 
  Consider the set $X=(C_{1,2} \cup C_{1,3} \cup C_{2,3}) - \{c_1,c_2,c_3\}$ in $M/c_1/c_2/c_3$.
  If $r_{M/c_1/c_2/c_3}(X)=3$, then $X$ consists of three disjoint parallel pairs, so $|D| \ge 3$; a contradiction.
  So $r_{M/c_1/c_2/c_3}(X) \le 2$.
  Then, for some $\{i,j,k\} = [3]$, the set $X' = (C_{i,j} \cup C_{i,k}) - \{c_1,c_2,c_3\}$ satisfies $r_{M/c_1/c_2/c_3}(X')=1$.
  If $|X'| \ge 3$, then $|X' \cap D| \ge 2$ and $|D| \ge 3$; a contradiction.
  So $|X'| = 2$; that is, $C_{i,j} = \{c_i,c_j,d_0,e\}$ and $C_{i,k} = \{c_i,c_k,d_0,e\}$ for some $d_0 \in \{d,d'\}$.
  Now $C_{i,j} \cup C_{i,k}$ is a $5$-element plane that meets $Q_1$.
  By orthogonality, $|Q_1 \cap (C_{i,j} \cup C_{i,k})| \ge 3$, so $Q_1 \subseteq \cl_M(C_{i,j} \cup C_{i,k})$.
  It now follows from \cref{nomeet3con} that $Q_1 \cup c$ is an \planespider\ for some $c \in \{c_1,c_2,c_3\}$, so  \cref{wtdp1c3} holds.

  We may now assume that $|D| \ge 3$.

  \begin{sublemma}
    \label{nomeet3del}
    Let $d$, $d'$, and $d''$ be distinct elements in $D$, and suppose $\{d,d'\}$ is contained in a quad~$Q$, $d'' \in E(M)-Q$, and $\{d,d''\}$ is contained in a $4$-element cocircuit~$C^*$.
    Then either $|Q \cap C^*|=2$, or \cref{wtdp1c3} holds.
  \end{sublemma}
  \begin{slproof}
    By orthogonality, $|Q \cap C^*| \neq 1$, so $|Q \cap C^*| \in \{2,3\}$.
    Suppose $|Q \cap C^*| = 3$, in which case $C^* - Q = \{d''\}$.
    Now $d'' \in \cocl(Q) - Q$ and $d''$ is $N$-deletable, so \cref{wtdp1c3} holds, by \cref{nomeet3con}.
  \end{slproof}

  Recall that there is a pair of elements of $D$ contained in a quad~$Q_1$.
  We claim that, up to an $N$-label switch, there are distinct elements $d_0,d_1,d_2 \in D$ with $\{d_0,d_1\} \subseteq Q_1$ and $d_2 \in E(M)-Q_1$.
  Suppose $\{d_0,d_1\}$ is a pair of elements contained in $D \cap Q_1$, and $D - Q_1 = \emptyset$.
  Since $|D| \ge 3$, there exists an element $d' \in D-\{d,d_1\}$, so $d' \in Q_1$.
  Let $Q_1-\{d_0,d_1,d'\} = \{c\}$.
  As $c$ is a coloop in $M \ba d_0 \ba d_1 \ba d'$, we may assume that $c \in C$.
  Then, as $|C \cup D| \ge 5$, there exists an element $c' \in C-Q_1$.
  Now $\{c,c'\}$ is contained in a $4$-element circuit~$C_1$.
  By orthogonality and \cref{nomeet3con}, we may assume that $|C_1 \cap Q_1| = 2$, otherwise \cref{wtdp1c3} holds.
  Let $C_1 - Q_1 = \{c',d_2\}$.
  Since $M\ba d_0 \ba d' /c/c'$ has an $N$-minor, and $d_2$ is in a parallel pair in this matroid, we may assume, up to a possible $N$-label switch, that $d_2 \in D$, and $|\{d_0,d',d_1\} \cap D| \ge 2$.  Up to relabelling $d_0,d',d_1$, we may assume that $d_0,d_1,d_2 \in D$ (and $\{d_0,d_1\} \subseteq Q_1$).  This proves the claim.

  Now let $d_0$, $d_1$, and $d_2$ be distinct elements of $D$ with $\{d_0,d_1\} \subseteq Q_1$ and $d_2 \notin Q_1$.
  In what follows we also assume that \cref{wtdp1c3} does not hold.
  If $C_1^*$ is the $4$-element cocircuit containing $\{d_0,d_2\}$, then $|Q_1 \cap C_1^*| = 2$ by \cref{nomeet3del}.
  First we rule out the case where $d_1 \in C_1^*$.
  More generally, we prove the following:

  \begin{sublemma}
    \label{qq0}
    Let $d,d',d''$ be distinct elements in $D$.
    If $\{d,d'\}$ is contained in a quad~$Q$, and $\{d,d''\}$ is contained in a $4$-element cocircuit~$C^*$ where $d'' \notin Q$, then $d' \notin C^*$.
  \end{sublemma}
  \begin{slproof}
    Suppose $d' \in C^*$.
    By \cref{nomeet3del}, $Q \cap C^*=\{d,d'\}$.
    Pick $y$ so that $C^*-Q = \{y,d''\}$.
    Since $M \ba d \ba d'\ba d''$ has an $N$-minor, and $y$ is a coloop in this matroid, $y$ is $N$-flexible and, in particular, we may assume that $y \in D$, so $C^* \subseteq D$.
    Observe that $Q$ is $3$-separating in $M \ba d''$, and $y \in \cocl_{M \ba d''}(Q)$.
    It follows, by \cref{3sepwin2}, that \cref{wtdp1c2} holds; a contradiction.
  \end{slproof}

  We next handle the case where the $4$-element cocircuit containing $\{d_0,d_2\}$ is a quad.
  More generally, we prove the following:

  \begin{sublemma}
    \label{qq}
    If there is an element $d'' \in D - Q_1$ such that $\{d,d''\}$ is contained in a quad $Q_2$, and $d_1 \notin Q_2$,
    then $Q_1 \cup Q_2$ is a \spikelike\ where $d_0$, $d_1$, and $d''$ are in different legs.
  \end{sublemma}
  \begin{slproof}
    By \cref{qq0}, $d_1 \notin Q_2$.
    Observe that $d_0 \in Q_1 \cap Q_2$, $d_1 \in Q_1-Q_2$ and $d'' \in Q_2-Q_1$.
    By \cref{nomeet3del}, $|Q_1 \cap Q_2|=2$.
    Now $\{d_1,d''\}$ is contained in a $4$-element cocircuit~$C_1^*$ and, by \cref{nomeet3del} again, $|Q_1 \cap C_1^*| = |Q_2 \cap C_1^*| = 2$.
    By \cref{qq0} again, $d_0 \notin C_1^*$.
    Hence, either $C_1^* = Q_1 \triangle Q_2$, or $C_1^* = \{d_1,d'',x,q'\}$ where $(Q_1 \cap Q_2)-d_0 = \{x\}$ and $q' \in E(M) - (Q_1 \cup Q_2)$.

    We first consider the latter case.
    Pick $q_1$ and $q_2$ so that $Q_1= \{d_0,d_1,x,q_1\}$ and $Q_2=\{d_0,x,d'',q_2\}$. 
    As $Q_1 \cup Q_2$ is $3$-separating, $q' \in \cocl(Q_1 \cup Q_2)$, and, as $|E(M)| \ge 9$, it follows from \cref{gutsstayguts3} that $q' \notin \cl(Q_1 \cup Q_2)$. 
    Since $\{q_1,q_2,q',x\}$ is a series class in $M \ba d_0 \ba d_1 \ba d''$,
    we have $|E(M)|-|E(N)| \ge 6$, so $|E(M)| \ge 10$.
    Moreover, we may assume that $\{x,q'\} \subseteq C$, so $\{x,q'\}$ is contained in a $4$-element circuit~$C_1$. 
    By orthogonality, $C_1$ meets $Q_1-x$ and $Q_2-x$.
    As $q' \notin \cl(Q_1 \cup Q_2)$, we deduce that $C_1 = \{x,q',d_0,f\}$ where $f \in E(M)- (Q_1 \cup Q_2 \cup q')$.

    Since $M \ba d_1 \ba d''/x/q'$ has an $N$-minor, and $\{d_0,f\}$ is a parallel pair in this matroid, we may assume, up to an $N$-label switch, that $\{d_1,d'',f\} \subseteq D$.
    So $\{f,d''\}$ is contained in a $4$-element cocircuit~$C_2^*$. 
    The cocircuit~$C_2^*$ meets $\{x,d_0,q_2\}$, by orthogonality with $Q_2$.
    If $q_2 \in C_2^*$, then $C_2^*$ also meets $\{x,d_0,q'\}$, by orthogonality with $C_1$.
    On the other hand, if $q_2 \notin C_2^*$, then $C_2^*$ meets $\{x,d_0\}$, and hence intersects $Q_1$ in two elements, by orthogonality.  In either case, $C_2^* \subseteq Q_1 \cup Q_2 \cup \{q',f\}$.  Thus $f$ is in the closure and coclosure of the $3$-separating set $Q_1 \cup Q_2 \cup q'$; since $|E(M)| \ge 10$, this is contradictory.

    Now we consider the case where $C_1^* = Q_1 \triangle Q_2$.
    If $C_1^*$ is dependent, then $C_1^*$ is a quad, and $Q_1 \cup Q_2$ is a \spikelike, as illustrated in \cref{spikelikecase}.
    So suppose $C_1^*$ is independent.
    We let $Q_1-Q_2 = \{q_1,d_1\}$ and $Q_2-Q_1=\{q_2,d''\}$.
    The matroid $M \ba d_0 \ba d_1 \ba d'' / q_1 /q_2$ has an $N$-minor,
    so $\{q_1,q_2\}$ is contained in a $4$-element circuit~$C_1$. 
    Since $C_1^*$ is independent, $C_1 \neq C_1^*$, so $C_1$ meets $\{d_0,x\}$, by orthogonality.  But $\{d_0,x\} \nsubseteq C_1$, by \cref{nomeet3con}.
    So $C_1 = \{q_1,q_2,x',f\}$, where $x' \in \{x,d_0\}$ and $f \in E(M)-(Q \cup q_2)$.

    If $x' \neq d_0$, then it follows that $M \ba d_0 \ba d_1 \ba d'' / x' / q_1$ has an $N$-minor, where $\{q_2,f\}$ is a parallel pair in this matroid.
    If $f = d''$, then $q_1 \in \cl(Q_2) \cap C$, in which case \cref{wtdp1c3} holds by \cref{nomeet3con}; so we may assume $f \neq d''$. Now, up to switching the $N$-labels on $q_2$ and $f$, we have $\{d'',q_2\} \subseteq D$, in which case \cref{wtdp1c2} holds by \cref{3sepwin2}; a contradiction.

    So let $C_1 = \{q_1,q_2,d_0,f\}$.
    Now, as $M \ba d_1 \ba d'' / q_1 / q_2$ has an $N$-minor, and $\{d_0,f\}$ is a parallel pair in this matroid, $M \ba d_1 \ba d'' \ba f/ q_1 / q_2$ has an $N$-minor.
    So $\{f,d''\}$ is contained in a cocircuit~$C_2^*$. 
    Since $Q_1 \cup Q_2$ is $3$-separating, and $f \in \cl(Q_1 \cup Q_2)$, we have $f \notin \cocl(Q_1 \cup Q_2)$, since $|E(M)| \ge 9$.
    In particular, $C_2^* \nsubseteq Q_1 \cup Q_2 \cup f$.
    It follows, by orthogonality between $C_2^*$ and either $Q_2$ or $Q_1$, that $q_2 \in C_2^*$, and $C_2^* = \{f,q_2,d'',h\}$ for some $h \in E(M)-(Q_1 \cup Q_2 \cup f)$.
    As $M \ba d_1 \ba d'' \ba f/q_1$ has an $N$-minor, and $\{q_2,h\}$ is a series pair in this matroid, $M \ba f/q_1/h$ has an $N$-minor.
    So $\{q_1,h\}$ is contained in a circuit~$C_2$. 
    By orthogonality with $Q_1$, $Q_2$, and $C_2^*$, we have $C_2 \subseteq Q_1 \cup Q_2 \cup \{f,h\}$, so $h \in \cl(Q_1 \cup Q_2 \cup f) \cap \cocl(Q_1 \cup Q_2 \cup f)$.
    By \cref{gutsstayguts3}, we deduce that $|E(M)|=9$.  Since $Q_1 \cup Q_2$ and $Q_1 \cup Q_2 \cup f$ are exactly $3$-separating, $f$ is a guts element, so $E(M)-(Q_1 \cup Q_2)$ is a triangle containing $\{f,h\}$.  But since $M \ba f/h$ has an $N$-minor, this triangle is not \unfortunate; a contradiction.
\end{slproof}

    \begin{figure}[bt]
      \centering
      \begin{tikzpicture}[rotate=90,xscale=0.8,yscale=0.9,line width=1pt]
        \tikzset{VertexStyle/.append style = {minimum height=5,minimum width=5}}
        \clip (-2.5,-6) rectangle (3.0,2);
        \node at (-1,-1.4) {$E(M)-(Q_1 \cup Q_2)$};
        \draw (0,0) .. controls (-3,2) and (-3.5,-2) .. (0,-4);
        \draw (0,0) -- (2.25,-0.75);
        \draw (0,0) -- (2,-2);
        \draw (0,0) -- (1,-3);

        \Vertex[x=2.25,y=-0.75,LabelOut=true,L=$d_0$,Lpos=90]{c2}
        \Vertex[x=2,y=-2,LabelOut=true,L=$d_1$,Lpos=90]{c6}
        \Vertex[x=1,y=-3,LabelOut=true,L=$d''$,Lpos=90]{c7}

        \SetVertexNoLabel
        \Vertex[x=1.5,y=-0.5,LabelOut=true,L=$q_4$,Lpos=135]{c4}
        \Vertex[x=1,y=-1,LabelOut=true,L=$x$,Lpos=180]{c5}
        \Vertex[x=0.67,y=-2,LabelOut=true,L=$q_4$,Lpos=135]{c8}

        \draw (0,0) -- (0,-4);

      \end{tikzpicture}
      \caption{The labelling of the \spikelike\ that arises in \cref{qq}.} 
      \label{spikelikecase}
    \end{figure}

    When \cref{qq} holds, it remains to show that $E(M)-E(N)$ is contained in a spike-like $3$-separator.

  \begin{sublemma}
    \label{qq2}
    Suppose $X'$ is a \spikelike\ where each leg has an element that is $N$-labelled for deletion.
    Then $X'$ is contained in a \spikelike~$X$ such that $E(M)-E(N) \subseteq X$.
  \end{sublemma}
  \begin{slproof}
    Let $X'$ be a \spikelike\ containing elements $d_0$, $d_1$, and $d_2$, no two of which is contained in one leg, where $\{d_0,d_1,d_2\}$ is $N$-labelled for deletion.
    Let $X$ be a \spikelike\ containing $X'$ that is maximal subject to the constraint that each leg has an element that is $N$-labelled for removal.
    Let $L_1,L_2,\dotsc,L_t$ be the legs of the \spikelike~$X$, where $L_i = \{d_i,c_i\}$ for each $i \in [t-1]$, and $L_t = \{d_0,x\}$.
    We may assume, up to switching $N$-labels, that $c_1$ and $c_2$ are $N$-labelled for contraction.
    Towards a contradiction, suppose there is either some $c' \in E(M) - X$ that is $N$-labelled for contraction, or some $d' \in E(M)-X$ that is $N$-labelled for deletion.

    First, we claim that if there exists some $d' \in D-X$, then $d'$ is in a $4$-element cocircuit~$C^*$ such that $C^* \cap X = L_i$ for some $i \in [t]$.
    Let $d' \in D-X$.
    Then $\{d',d_1\}$, $\{d',d_2\}$, and $\{d',d_0\}$ are contained in $4$-element cocircuits. 
    Suppose that each of these cocircuits is contained in $X \cup d'$.
    It follows, by orthogonality with the circuits of $X$, that $t=3$. 
    Since $r^*(X) = 4$, and $d' \in \cocl(X)$, the set $\{d',d_0,d_1,d_2,c_2\}$ contains a cocircuit.
    But if $\{d',d_1,d_2,c_2\}$ is a cocircuit, then it intersects the circuit $\{x,d_0,c_1,d_1\}$ in a single element, contradicting orthogonality.
    So the cocircuit contained in $\{d',d_0,d_1,d_2,c_2\}$ contains $d_0$, and similarly we deduce that it contains $d_1$ and $d_2$.
    Thus either $\{d',d_0,d_1,d_2\}$, $\{d_0,d_1,d_2,c_2\}$, or $\{d',d_0,d_1,d_2,c_2\}$ is a cocircuit.
    In the first case, we may assume that $d_1$ is $N$-labelled for contraction, while in the latter two cases, we can swap $N$-labels on $d_1$ and $c_2$; in any case, $\{c_1,d_1\}$ is $N$-labelled for contraction.
    Observe that $d_1 \in \cl_{M/c_1}(X-L_1)$, so \cref{wtdp1c2} holds by the dual of \cref{3sepwin2}.
    We deduce that for some $i \in \{0,1,2\}$, the $4$-element cocircuit $C_i^*$ containing $\{d',d_i\}$ is not contained in $X \cup d'$.
    By orthogonality, $C_i^* \cap X = L_i$.
    This proves the claim.

    We can argue similarly in the dual.
    That is, if there is some $c' \in C-X$, then either $c'$ is in a $4$-element circuit~$C$ such that $C \cap X = L_i$ for some $i \in [t]$, or $t=3$ and $c' \in \cl(X)$.
    In the latter case, $r(X \cup c') = 4$, so $\{c',x,c_1,c_2,d_2\}$ contains a circuit.
    But if $\{c',c_1,c_2,d_2\}$ is a circuit, then it intersects the cocircuit $\{x,d_0,c_1,d_1\}$ in a single element, contradicting orthogonality.
    So the circuit contained in $\{c',x,c_1,c_2,d_2\}$ contains $x$, and, similarly, it contains $c_1$ and $c_2$.
    So either $\{c',x,c_1,c_2\}$, $\{x,c_1,c_2,d_2\}$, or $\{c',x,c_1,c_2,d_2\}$ is a circuit.
    In any case, we can perform an $N$-label switch so that $x$ is $N$-labelled for deletion, without affecting the $N$-label on $d_0$.  Since $\{x,d_0\}$ is $N$-labelled for deletion, \cref{wtdp1c2} holds by \cref{3sepwin2}.
    So we may assume that if there is some $c' \in C-X$, then $c'$ is in a $4$-element circuit~$C$ such that $C \cap X = L_i$ for some $i \in [t]$.

    Next we claim that if there exists some $d' \in D-X$, there also exists some $c' \in C-X$.
    Suppose $d' \in D-X$.
    Without loss of generality, the $4$-element cocircuit containing $\{d',d_1\}$ is not contained in $X \cup d'$.
    Let $C^* = \{d_1,c_1,d',x'\}$ be this cocircuit, for some $x' \in E(M)-(X \cup d')$.
    Now, as $\{d_0,x,d',x',d_1,c_1\}$ is a corank-$4$ set where $d_0$, $d'$ and $d_1$ are $N$-labelled for deletion, we may assume $x'$ and $c_1$ are $N$-labelled for contraction. 
    As $x' \in C-X$, this proves the claim.

    Now let $c' \in C-X$.
    Recall that $c'$ is in a $4$-element circuit~$C$ such that $C \cap X = L_i$ for some $i \in [t]$.
    Without loss of generality, $C \cap X = L_1$.
    Let $C = \{d_1,c_1,d',c'\}$ for some $d' \in E(M)-(X \cup c')$.
    Pick distinct $i,j \in \{2,3,\dotsc,t\}$.
    By circuit elimination on $C$ and $L_1 \cup L_i$, there is a circuit contained in $L_i \cup \{c_1,c',d'\}$.
    But, by orthogonality with the cocircuit $L_1 \cup L_j$, this circuit does not contain $c_1$.
    We deduce that $L_i \cup \{c',d'\}$ is a circuit for each $i \in [t]$.

    Due to the circuit $L_1 \cup \{c',d'\}$, we may assume $d'$ is $N$-labelled for deletion up to an $N$-label switch with $d_1$ (note that $c_1$ is an element in $L_1$ that is $N$-labelled for contraction).
    Thus $\{d',d_0\}$ is contained in a $4$-element cocircuit~$C^*$ such that $C^* \cap X = L_t$.
    Moreover, by orthogonality with the circuit $\{c_2,d_2,c',d'\}$, the cocircuit~$C^*$ meets $\{c_2,d_2,c'\}$, so $C^* = \{d_0,x,d',c'\}$.
    By cocircuit elimination with the cocircuits in $X$, and orthogonality with the circuits in $X$, it follows that $L_i \cup \{d',c'\}$ is a cocircuit for each $i \in [t]$.
    Now $X \cup \{c',d'\}$ is a \spikelike, where $c'$ is $N$-labelled for contraction, so $X$ is not maximal; a contradiction.
  \end{slproof}

  Now, by \cref{qq,qq2}, if the $4$-element cocircuit~$C_1^*$ containing $\{d_0,d_2\}$ is a quad, then \cref{wtdp1c1} holds.
  Next we handle the case where $C_1^*$ is independent.  We break it into two parts: first, the case where $Q_1 \cup C_1^*$ is not $3$-separating; and second, the case where $\lambda(Q_1 \cup C_1^*)=2$.

  \begin{sublemma}
    \label{qn2}
    If $\{d_0,d_2\}$ is contained in a $4$-element independent cocircuit $C_1^*$, where $d_1 \notin C_1^*$, then either
    \begin{enumerate}[label=\rm(\Roman*)]
      \item $\lambda(Q_1 \cup C_1^*)=2$, or
      \item $Q_1 \cup C_1^*$ is contained in a \spider~$Y$, and either $E(M)-E(N) \subseteq Y$, or $|E(M)| \in \{10,12\}$ and \cref{wtdp1c4} holds.
    \end{enumerate}
  \end{sublemma}
  \begin{slproof}
    Suppose that (I) does not hold; we will show that (II) holds.
    Recall that $\{d_0,d_1\}$ is contained in the quad $Q_1$, and $d_2 \notin Q_1$.
    Let $X = Q_1 \cup C_1^*$, and pick $c_1$ and $c_2$ so that $X-C_1^* = \{c_1,d_1\}$ and $X-Q_1 = \{c_2,d_2\}$.
    If $r(X) = 4$, then $\lambda(X)=2$; so we may assume that $r(X) = 5$.
    Since $M \ba d_0 \ba d_1 \ba d_2$ has an $N$-minor, and $X-\{d_0,d_1,d_2\}$ has corank one in this matroid, $M\ba d_0 \ba d_1 \ba d_2/c_1/c_2$ has an $N$-minor.  
    Now $\{c_1,c_2\}$ is contained in a $4$-element circuit~$C_1$. 
    If $C_1 \subseteq X$, then $r(X)=4$; a contradiction.
    Thus, by orthogonality, $|Q_1 \cap C_1^* \cap C_1| = 1$.

    Let $C_1 = \{c_1,c_2,x',e\}$, where $x' \in Q_1 \cap C_1^*$ and $e \in E(M)-X$.
    Since $M \ba d_1 \ba d_2 / c_1 / c_2$ has an $N$-minor, and $\{e,x'\}$ is a parallel pair in this matroid, $M \ba d_1 \ba d_2 \ba e / c_1 / c_2$ has an $N$-minor.
    So $\{e,d_1\}$ is contained in a $4$-element cocircuit~$C_2^*$. 
    If $C_2^* \subseteq Q_1 \cup C_1$, then $Q_1 \cup C_1$ has rank and corank at most four, so $Q_1 \cup C_1$ is $3$-separating.
    Similarly, if $C_2^* \subseteq Q_1 \cup C_1 \cup d_2$, then $Q_1 \cup C_1 \cup d_2$ is $3$-separating in $M$, and $d_2 \notin \cl(Q_1 \cup C_1)$, so $Q_1 \cup C_1$ is also $3$-separating.
    Since $M \ba d_1$ is $3$-connected, it follows that $\lambda_{M \ba d_1}((Q_1-d_1) \cup C_1)=2$, and $d_2 \in \cocl_{M \ba d_1}((Q_1-d_1) \cup C_1)$, in which case \cref{wtdp1c2} holds by \cref{3sepwin2}.
    Thus we may assume that $C_2^* \nsubseteq Q_1 \cup C_1 \cup d_2$.
    By orthogonality, we deduce that $C_2^*$ meets $\{x',c_1\}$, and $f \in C_2^*$ for some $f \in E(M)-(Q_1 \cup C_1 \cup d_2)$.

    Let $Y = Q_1 \cup C_1 \cup \{f,d_2\}$.
    Since $M \ba d_1 \ba e / c_2$ has an $N$-minor, and $f$ is in a series pair in this matroid, $M / c_2 / f$ has an $N$-minor, so $\{c_2,f\}$ is contained in a $4$-element circuit~$P$. 
    By orthogonality, $P$ meets $C_1^*-c_2$ and $C_2^*-f$.
    Since $x'$ is the only possible element in $(C_1^*-c_2) \cap (C_2^*-f)$, we see that $P \subseteq Y$ when $x' \notin P$.
    But if $x' \in P$, then $P$ also meets $Q_1-x'$ by orthogonality, in which case $P \subseteq Y$.
    So $r(Y) =5$, and it follows that $r^*(Y)=5$ and $\lambda(Y) =2$.
    In the case that $P$ meets $Q_1$, we have $r(Y-d_2) =4$, so $d_2 \notin \cl(Y-d_2)$.
    It follows that $Y-\{d_1,d_2\}$ is $3$-separating in $M \ba d_1$, and $d_2 \in \cocl_{M \ba d_1}(Y-\{d_1,d_2\})$.  By another application of \cref{3sepwin2}, \cref{wtdp1c2} holds.

    In the remaining case, $P=\{c_2,e,f,d_2\}$.
    Recall that $x' \in Q_1 \cap C_1$, and choose $x$ so that $Q_1 \cap C_1 = \{x,x'\}$.
    If $d_0=x$, then $M \ba d_0 \ba d_1 \ba d_2/c_1/x'$ has an $N$-minor and $\{c_2,e\}$ is a parallel pair in this matroid, so $M \ba d_2 \ba c_2$ has an $N$-minor.
    Then $Q_1$ is $3$-separating in $M \ba d_2$ with $c_2 \in \cocl_{M \ba d_2}(Q_1)$, and it follows that \cref{wtdp1c2} holds by \cref{3sepwin2}.
    So we may assume that $d_0=x'$.
    Recall also that $C_2^*$ meets $\{d_0,c_1\}$.  In a similar vein, if $d_0 \in C_2^*$, then $\{e,f\}$ is a series pair in $M \ba d_0 \ba d_1 \ba d_2 / c_1$, so $M \ba d_2 / e / c_1$ has an $N$-minor, and hence $M \ba d_2 \ba c_2$ has an $N$-minor.
    As before, in this case \cref{wtdp1c2} holds by \cref{3sepwin2}.
    So we may assume that $c_1 \in C_2^*$.

    We work towards showing that $Y$ is a \spider\ of $M$, labelled as illustrated in \cref{psst}.
    The pair $\{d_1,d_2\}$ is in a $4$-element cocircuit~$C_3^*$ that meets both $Q_1-d_1$ and $P-d_2$, by orthogonality.
    As these sets are disjoint, $C_3^* \subseteq Y$.
    If $c_2 \notin C_3^*$, then $d_2 \in \cocl(Q_1 \cup C_2^*)$, and it follows that $r^*(Y) \le 4$; a contradiction.
    So $c_2 \in C_3^*$.
    Now, if $c_1 \notin C_3^*$, then $d_1 \in \cocl(C_1^*)$, and it follows that $r^*(Y) \le 4$; a contradiction.
    Note also that $x \notin C_3^*$, for otherwise $r^*(X) = 3$ and $\lambda(X) = 2$.
    So $C_3^* = \{d_1,d_2,c_1,c_2\}$.

    Recall that $\{c_2,x\}$ is contained in a series class of size at least three in $M \ba d_0 \ba d_1 \ba d_2$, so $\{c_2,x\}$ is $N$-contractible, and hence is contained in a $4$-element circuit~$C_2$. 
    By orthogonality with $Q_1$ and $C_3^*$, either $C_2$ meets $\{c_1,d_1\}$, or $C_2=C_1^*$.
    But $C_1^*$ is independent, so the former case holds. 
    By orthogonality with $C_2^*$, we see that $C_2 \subseteq C_2^* \cup \{c_2,x\}$.
    Moreover, $C_2$ meets $\{e,f\}$, otherwise $C_2 \subseteq Q_1 \cup c_2$, which contradicts \cref{nomeet3con}.
    If $f \in C_2$, then $f \in \cl(Q_1 \cup C_1)$, and it follows that $r(Y) \le 4$; a contradiction.
    So $e \in C_2$.
    Now, if $c_1 \in C_2$, then $x \in \cl(C_1)$, and it follows that $r(Q_1 \cup C_1) \le 3$; a contradiction.
    So $C_2 = \{c_2,x,e,d_1\}$.

    Now, as $M \ba d_0 / c_2/x$ has an $N$-minor, and $\{e,d_1\}$ is a parallel pair in this matroid, $M \ba d_0 \ba e$ has an $N$-minor.
    So $\{d_0,e\}$ is contained in a $4$-element cocircuit~$C_4^*$. 
    By orthogonality, this cocircuit meets $C_2-e=\{c_2,x,d_1\}$ and $Q_1-d_0=\{c_1,x,d_1\}$, so either $C_4^*=\{d_0,e,c_1,c_2\}$, or $C_4^*$ meets $\{x,d_1\}$.
    In the former case, $e \in \cocl(Q_1 \cup C_1^*)$, and it follows that $r^*(Y) \le 4$; a contradiction.
    So $C_4^*$ meets $\{x,d_1\}$.
    By orthogonality with $P$, the cocircuit $C_4^*$ also meets $\{c_2,f,d_2\}$.
    If $f \notin C_4^*$, then $e \in \cocl(Q_1 \cup C_1^*)$, and it follows that $r^*(Y) \le 4$; a contradiction.
    So $f \in C_4^*$.
    Now, if $d_1 \in C_4^*$, then $r^*(C_2^* \cup Q_1) = 3$, so $r^*(Y) \le 4$; a contradiction.
    We deduce that $C_4^* = \{d_0,e,f,x\}$.

    Recall that $M \ba d_1 \ba d_2 \ba e$ has an $N$-minor, and note that $\{c_2,f,c_1\}$ is contained in a series class in this matroid, since $r^*(C_2^* \cup C_3^*) = 4$.
    Thus $M/c_1/f$ has an $N$-minor, so $\{c_1,f\}$ is contained in a $4$-element circuit~$C_3$.
    By orthogonality, $C_3$ meets $Q_1-c_1$ and $C_4^*-f$.
    Thus, either $C_3 = \{c_1,f,d_1,e\}$, or $C_3$ meets $\{d_0,x\}$.
    In the former case, $f \in \cl(Q_1 \cup C_1)$, and it follows that $r(Y) \le 4$; a contradiction.
    So $C_3$ meets $\{d_0,x\}$, and hence, by orthogonality with $C_1^*$, the circuit~$C_3$ also meets $\{c_2,d_2\}$.
    If $c_2 \in C_3$, then $f \in \cl(Q_1\cup C_1)$, and it follows that $r(Y) \le 4$; a contradiction.  So $d_2 \in C_3$.
    Now, if $x \in C_3$, then $x \in \cl(P \cup C_1)$, and $r(Y) \le 4$; a contradiction.
    So $C_3=\{c_1,f,d_0,d_2\}$.

    Since $M \ba d_1 \ba e / c_1 / f$ has an $N$-minor, and $\{d_0,d_2\}$ is a parallel pair in this matroid, $M \ba d_0 \ba d_1 \ba e$ has an $N$-minor.
    As $r^*(Q_1 \cup C_4^*) = 4$, we see that $\{c_1,x,f\}$ is contained in a series class in $M \ba d_0 \ba d_1 \ba e$, so $\{x,f\}$ is $N$-contractible.
    Now $\{x,f\}$ is contained in a $4$-element circuit~$C_4$.
    By orthogonality with the cocircuits $Q_1$ and $C_2^*$, either $C_4=\{x,f,d_0,e\}$, or $C_4$ meets $\{c_1,d_1\}$.
    In the former case, $C_4=C_4^*$ is a quad, and it follows that \cref{wtdp1c1} holds by \cref{qq,qq2}.
    In the latter case, $C_4$ meets $\{c_2,d_2\}$, by orthogonality with the cocircuits $C_1^*$ and $C_3^*$.
    If $c_1 \in C_4$, then $x \in \cl(P \cup C_3)$, and it follows that $r(Y) \le 4$; a contradiction.  So $d_1 \in C_4$.
    Likewise, if $c_2 \in C_4$, then $f \in \cl(Q_1 \cup C_1)$, and $r(Y) \le 4$; a contradiction.
    So $C_4 = \{x,f,d_1,d_2\}$.

    Finally, we show that $P$ is a cocircuit.
    As $\{e,d_2\}$ is $N$-deletable, there is certainly a $4$-element cocircuit~$C_5^*$ containing $\{e,d_2\}$.
    By orthogonality with the circuits $C_1$ and $C_3$, either $P=C_5^*$, or $C_5^*$ meets $\{c_1,d_0\}$.
    Similarly, due to the circuits $C_2$ and $C_4$, either $P=C_5^*$, or $C_5^*$ meets $\{d_1,x\}$.
    So, if $P \neq C_5^*$, then $C_5^* \subseteq Q_1 \cup \{e,d_2\}$, in which case $e \in \cocl(Q_1 \cup C_1^*)$, implying $r^*(Y) \le 4$; a contradiction.
    We deduce that $P$ is a cocircuit, and hence $Y$ is a \spider\ of $M$.
%
%
    \begin{figure}[htb]
      \begin{tikzpicture}[rotate=90,scale=0.8,line width=1pt]
        \tikzset{VertexStyle/.append style = {minimum height=5,minimum width=5}}
        \clip (-2.5,-6) rectangle (3.0,2);
        \node at (-1,-1.4) {$E(M)-Y$};
        \draw (0,0) .. controls (-3,2) and (-3.5,-2) .. (0,-4);

        \draw (0,0) -- (2,-2) -- (0,-4);

        \draw (0,0) -- (2.5,0.5) -- (2,-2);
        \draw (0,0) -- (2.25,-0.75);
        \draw (2,-2) -- (1.25,0.25);

        \draw (0,-4) -- (2.5,-4.5) -- (2,-2);
        \draw (0,-4) -- (2.25,-3.25);
        \draw (2,-2) -- (1.25,-4.25);

        \Vertex[x=1.25,y=0.25,LabelOut=true,L=$c_1$,Lpos=180]{c1}
        \Vertex[x=2.25,y=-0.75,LabelOut=true,L=$x$,Lpos=90]{c2}
        \Vertex[x=2.5,y=0.5,LabelOut=true,L=$d_1$,Lpos=180]{c3}
        \Vertex[x=1.5,y=-0.5,LabelOut=true,L=$d_0$,Lpos=135]{c4}

        \Vertex[x=1.25,y=-4.25,LabelOut=true,L=$e$]{c1}
        \Vertex[x=2.25,y=-3.25,LabelOut=true,L=$d_2$,Lpos=90]{c2}
        \Vertex[x=2.5,y=-4.5,LabelOut=true,L=$f$]{c3}
        \Vertex[x=1.5,y=-3.5,LabelOut=true,L=$c_2$,Lpos=45]{c4}

        \draw (0,0) -- (0,-4);

      \end{tikzpicture}
      \caption{The labelling of the \spider\ in \cref{qn2}, where $\{d_0,d_1,d_2\} \subseteq D$ and $\{c_1,c_2\} \subseteq C$.}
      \label{psst}
    \end{figure}
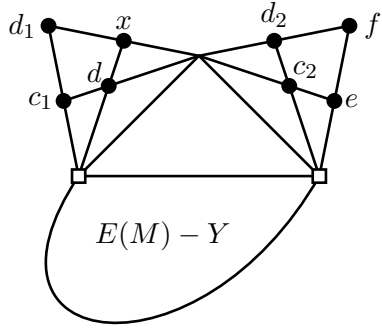

    \smallbreak

    Suppose that $E(M)-E(N) \nsubseteq Y$.
    It remains to show that $|E(M)| \in \{10,12\}$ and \cref{wtdp1c4} holds.
    Let $(C,D)$ be an $N$-labelling of $M$ with $\{c_1,c_2\} \subseteq C$ and $\{d_0,d_1,d_2\} \subseteq D$, and observe that there exists either some $c' \in C - Y$ or some $d' \in D-Y$.

    First, we show that if there is some such $c'$, then any $4$-element circuit containing $c'$ and an element of $C \cap Y$ is not contained in $Y \cup c'$.
    Consider the $4$-element circuit~$C_5$ containing $\{c',c_1\}$.  If this circuit does not contain $d_1$, then, by orthogonality, it meets $\{d_0,x\}$, $\{d_2,c_2\}$, and $\{e,f\}$; a contradiction.
    So $d_1 \in C_5$.
    Now $C_5$ meets $C_1^*$ and $P$ in at most one element, so $C_5 \cap (C_1^* \cup P) = \emptyset$, and thus $C_5 \cap Y = \{c_1,d_1\}$.
    The argument is similar when considering the $4$-element circuit containing $\{c',c\}$ for some $c \in C \cap (Y-c_1)$, where in each case the circuit intersects $Y$ in one of $\{c_1,d_1\}$, $\{d_0,x\}$, $\{c_2,d_2\}$, or $\{e,f\}$.
    A dual argument shows that if there exists some $d' \in D-Y$, then any $4$-element cocircuit containing $d'$ and an element in $D \cap Y$ intersects $Y$ in one of $\{d_1,x\}$, $\{c_1,d_0\}$, $\{d_2,f\}$, or $\{c_2,e\}$.

    Suppose there exists some $d' \in D-Y$.
    Then $d'$ is contained in a $4$-element cocircuit with $d_1$, and hence $x$. Let this cocircuit be $\{d',d_1,x,c'\}$ for $c' \in E(M)-Y$.
    Up to an $N$-label switch on $x$ and $c'$, the element $c'$ is $N$-labelled for contraction.
    So we may assume that there is some $c' \in C-Y$.

    Let $C_5$ be the $4$-element circuit containing $\{c',c_1\}$.
    Recall that $C_5 = \{c',z_1,c_1,d_1\}$ for some $z_1 \in E(M)-(Y \cup c')$.
    Let $C_6$ be the $4$-element circuit containing $\{c',c_2\}$; then $C_6=\{c',z_2,c_2,d_2\}$ for some $z_2 \in E(M)-(Y \cup c')$.
    By circuit elimination and orthogonality, $C_7 = \{c',z_1,d_0,x\}$ and $C_8 = \{c',z_2,e,f\}$ are also circuits.
    Now $r(Y \cup \{c',z_1,z_2\}) \le r(Y) + 1 = 6$.

    Due to the circuit $C_5$, we may assume $\{z_1,d_2\}$ is $N$-labelled for deletion, up to an $N$-label switch.
    Consider the $4$-element cocircuit~$C_6^*$ containing $\{z_1,d_2\}$.
    This cocircuit contains $f$.
    By orthogonality with $C_5$, the cocircuit~$C_6^*$ meets $\{c',c_1,d_1\}$,
    and by orthogonality with $C_6$, it meets $\{c',c_2,z_2\}$.
    So $C_6^* = \{z_1,c',d_2,f\}$.
    Similarly, due to the circuit $C_6$, we may assume $\{z_2,d_1\}$ is $N$-labelled for deletion.
    Now consider the $4$-element cocircuit~$C_7^*$ containing $\{z_2,d_1\}$.
    In a similar fashion, we deduce that $x \in C_7^*$ and, by orthogonality with $C_5$ and $C_6$, the final element of $C_7^*$ is $c'$.
    So $C_7^* = \{z_2,c',d_1,x\}$.
    By cocircuit elimination and orthogonality, $C_8^* = \{z_1,c',c_2,e\}$ and $C_9^* = \{z_2,c',c_1,d_0\}$ are also cocircuits.

    Now $z_1 \in \cocl(Y \cup c')$, implying $r^*(Y \cup \{c',z_1\}) \le r^*(Y) + 1=6$.
    If $z_1 = z_2$, then $\{c',z_1,c_1,d_0,c_2\}$ spans $Y \cup \{c',z_1\}$, so $r(Y \cup \{c',z_1\}) = 5$, implying $\lambda(Y \cup \{c',z_1\}) \le 5+6-10=1$.
    Hence $|E(M)| \in \{10,11\}$.
    If $|E(M)| = 10$, then $Q \cup \{c',z_1\}$ and $P \cup \{c',z_1\}$ are \pspider s, so \cref{wtdp1c4} holds.

    Suppose that $|E(M)| = 11$. Let $E(M)-(Y \cup \{c',z_1\}) = \{q\}$.
    If $q \notin \cl(Q)$, then $r(Q \cup q) = 4 = r(M)-1$, so $\cl(Q \cup q)$ is a hyperplane.
    But $E(M) - (Q \cup q)$ is the union of cocircuits $P$ and $C_6^*$; a contradiction.
    So $q$ is in a circuit properly contained in $Q \cup q$.
    By orthogonality with $C_1^*$ and $C_2^*$, it follows that $q$ is in a triangle with $\{c_1,d_1\}$ or $\{d_0,x\}$.  But $c_1$ and $x$ are $N$-contractible, so this triangle is not \unfortunate; a contradiction.

    We may now assume that $z_1 \neq z_2$.
    So $r^*(Y \cup \{c',z_1,z_2\}) \le r^*(Y) + 1 = 6$.
    Now $\lambda(Y \cup \{c',z_1,z_2\}) \le 6+6-11 = 1$, so $|E(M)| \in \{11,12\}$.
    If $|E(M)| = 11$, then $\{c',z_1,z_2\}$ is a triad, but $z_1$ is $N$-deletable so this triad is not \unfortunate; a contradiction.
    So $|E(M)| = 12$, and it follows that $Z=E(M)-Y$ is a quad.
    Pick $q$ so that $Z = \{q,c',z_1,z_2\}$.
    By (co)circuit elimination and orthogonality, it is easily checked that $\{q,z_2,c_1,d_1\}$, $\{q,z_2,d_0,x\}$, $\{q,z_1,c_2,d_2\}$, and $\{q,z_1,e,f\}$ are circuits, and $\{q,z_2,d_2,f\}$, $\{q,z_2,c_2,e\}$, $\{q,z_1,c_1,d_0\}$, and $\{q,z_1,d_1,x\}$ are cocircuits, so $Q \cup Z$ and $P \cup Z$ are \spider s, and \cref{wtdp1c4} holds.
  \end{slproof}

  \begin{sublemma}
    \label{qn3}
    Suppose $\{d_0,d_2\}$ is contained in a $4$-element independent cocircuit~$C_1^*$, where $d_1 \notin C_1^*$.
    Let $X = Q_1 \cup C_1^*$.
    If $\lambda(X)=2$, then $X$ is an \pspider\ where either $E(M)-E(N) \subseteq X$, or $|E(M)|=10$ and \cref{wtdp1c4} holds. 
  \end{sublemma}
  \begin{slproof}
    Pick $c_1$, $c_2$ and $x$ so that $X-C_1^* = \{c_1,d_1\}$, $X-Q_1 = \{c_2,d_2\}$, and $Q_1 \cap C_1^* = \{d_0,x\}$.
    By \cref{nomeet3con,nomeet3del}, $r(X) \ge 4$ and $r^*(X) \ge 4$.
    Since $\lambda(X)=2$, we have $r(X) = r^*(X) = 4$.
    As $M \ba d_0 \ba d_1 \ba d_2$ has an $N$-minor, any pair of elements contained in $\cocl(X)-\{d_0,d_1,x\}$ is $N$-contractible in this matroid, so such a pair is contained in a $4$-element circuit. 
    In particular, there are $4$-element circuits $C_1$ and $C_2$ containing $\{c_1,c_2\}$ and $\{c_2,x\}$, respectively.
    By orthogonality with the cocircuit $Q_1$, these two circuits intersect $X$ in at least three elements.
    We will show that we may assume that $C_i \subseteq X$, for $i \in \{1,2\}$.

    First, we claim that $Q_1 \triangle C_1^*$ is a cocircuit.
    As the pair $\{d_1,d_2\}$ is $N$-deletable, it is contained in a $4$-element cocircuit~$C_2^*$. 
    Suppose $C_2^* \nsubseteq X$.  By orthogonality, $|C_2^*-X|=1$.  Let $C_2^*-X=\{q\}$.
    Since $q \in \cocl(X)$, the set $X \cup q$ is $3$-separating, and the pair $\{c_2,q\}$ is contained in a $4$-element circuit~$P$. 
    If $P \subseteq X \cup q$, then $q \in \cl(X) \cap \cocl(X)$, so $\lambda(X \cup q) = 1$ and $|E(M)| \le 8$; a contradiction.
    It follows, by orthogonality, that $P = \{c_2,d_2,q,f\}$, for some $f \in E(M)-(X \cup q)$.
    Recall that $M \ba d_0 \ba d_1 / c_2 / q$ has an $N$-minor. By swapping the $N$-labels on $d_2$ and $f$, we deduce that $\{d_1,f\}$ is $N$-deletable, so this pair is contained in a $4$-element cocircuit~$C_3^*$.  By orthogonality with the disjoint circuits $Q_1$ and $P$, the cocircuit $C_3^*$ is contained in $X \cup \{q,f\}$.  Now $f \in \cl(X \cup q) \cap \cocl(X \cup q)$, so, by \cref{gutsstayguts3}, $|E(M)|=9$.  Then, since $X$ is exactly $3$-separating and $q$ is a coguts element, $E(M)-X$ is a triad containing $f$, but $M\ba f$ has an $N$-minor, so this triad is not \unfortunate.
    From this contradiction we deduce that $C_2^* \subseteq X$.
    Since $r^*(X)=4$, it follows that $C_2^*=Q_1\triangle C_1^*$, as claimed.

    Suppose $\{c_1,c_2,x,e\}$ is a circuit, for some $e \in E(M)-X$.
    Since $M \ba d_0 \ba d_1 \ba d_2 / x / c_1$ has an $N$-minor, and $\{c_2,e\}$ is a parallel pair in this matroid, $M \ba d_2 \ba c_2$ has an $N$-minor.
    As $Q_1$ is $3$-separating in $M \ba d_2$ with $c_2 \in \cocl_{M \ba d_2}(Q_1)$, it follows that \cref{wtdp1c2} holds by \cref{3sepwin2}.

    Now, if $C_i \nsubseteq X$, for some $i \in \{1,2\}$, then either $C_1=\{c_1,c_2,d_0,e\}$, or $C_2=\{c_2,x,d_1,e\}$, for some $e \in E(M)-X$.
    In either case, we have that $M \ba d_2 \ba e/c$ has an $N$-minor, for some $c \in \{c_1,x\}$.
    Indeed, in the first case, $M \ba d_2 \ba e / c_1$ has an $N$-minor by swapping the $N$-labels on $d_0$ and $e$;
    in the second, $M \ba d_2 \ba e/x$ has an $N$-minor, by switching the $N$-labels on $d_1$ and $e$.
    So, in either case, $\{e,d_2\}$ is contained in a $4$-element cocircuit~$C_3^*$.
    If $C_3^*$ is contained in $X \cup e$, then $e \in \cl(X) \cap \cocl(X)$, so $|E(M)| \le 8$; a contradiction.
    It follows, by orthogonality, that $C_3^*=\{e,d_2,c_2,h\}$ for some $h \in E(M)-(X \cup e)$.
    Since $M \ba d_2 \ba e/c$ has an $N$-minor, and $\{c_2,h\}$ is a series pair in this matroid, $M/c/h$ has an $N$-minor.
    Therefore, $\{h,c\}$ is contained in a $4$-element circuit.
    As $\{h,c\}$ meets both $C_3^*$ and $Q_1$, and these cocircuits are disjoint, we see that $h \in \cl(X \cup e)$, by orthogonality.
    But as $X \cup e$ is $3$-separating and $h \in \cocl(X \cup e)$, this implies $|E(M)|=9$.
    Then $E(M)-X$ is a triangle containing $h$, but $M/h$ has an $N$-minor, so this triangle is not \unfortunate; a contradiction.
    Hence we may assume that $C_i \subseteq X$ for $i \in \{1,2\}$.

    We will show that $X$ is an \pspider, labelled as illustrated in \cref{psss}.
    By \cref{nomeet3con}, $c_2 \notin \cl(Q_1)$, so
    $d_2 \in C_1$ and $d_2 \in C_2$.
    As $C_1^*$ is independent, $d_0 \notin C_2$, so $C_2 = \{c_2,x,d_2,q\}$ for $q \in \{c_1,d_1\}$.
    Suppose $q = c_1$.
    Then, as $M \ba d_0 \ba d_1 / c_1 / c_2$ has an $N$-minor, and $\{x,d_2\}$ is a parallel pair in this matroid, $M \ba d_0 \ba d_1 \ba x$ has an $N$-minor.
    Since $\{c_1,c_2,d_2\}$ is a series class in $M \ba d_0 \ba d_1 \ba x$, the matroid $M/c_2/d_2$ has an $N$-minor.
    Observe that $Q_1$ is $3$-separating in $M/c_2$, and $d_2 \in \cl_{M /c_2}(Q_1)$.
    It follows that \cref{wtdp1c2} holds by the dual of \cref{3sepwin2}.
    So we may assume that $C_2 = \{c_2,x,d_1,d_2\}$.
    Recall that $\{c_1,c_2,d_2\} \subseteq C_1 \subseteq X$.
    If $d_0 \notin C_1$, then, due to the circuits $Q_1$ and $C_2$, the set $\{c_2,x,d_2\}$ spans $X$, so $r(X) \le 3$; a contradiction.
    We deduce that $C_1 = \{c_1,c_2,d_2,d_0\}$.
    Finally, as $C_2^* = \{d_1,d_2,c_1,c_2\}$ is a cocircuit, $X$ is an \pspider. 

    \begin{figure}[htb]
      \begin{tikzpicture}[rotate=90,scale=0.7,line width=1pt]
        \tikzset{VertexStyle/.append style = {minimum height=5,minimum width=5}}
        \clip (-2.5,2) rectangle (3.0,-6);
        \node at (-1,-1.4) {$E(M)-X$};
        \draw (0,0) .. controls (-3,2) and (-3.5,-2) .. (0,-4);
        \draw (0,0) -- (2,-2) -- (0,-4);
        \draw (0,0) -- (2.5,0.5) -- (2,-2);
        \draw (0,0) -- (2.25,-0.75);
        \draw (2,-2) -- (1.25,0.25);

        \Vertex[x=1.25,y=0.25,LabelOut=true,L=$c_1$,Lpos=180]{c1}
        \Vertex[x=2.25,y=-0.75,LabelOut=true,L=$x$,Lpos=90]{c2}
        \Vertex[x=2.5,y=0.5,LabelOut=true,L=$d_1$,Lpos=180]{c3}
        \Vertex[x=1.5,y=-0.5,LabelOut=true,L=$d_0$,Lpos=135]{c4}
        \Vertex[x=1.33,y=-2.67,LabelOut=true,L=$d_2$,Lpos=45]{c5}
        \Vertex[x=0.67,y=-3.33,LabelOut=true,L=$c_2$,Lpos=45]{c6}

        \draw (0,0) -- (0,-4);

      \end{tikzpicture}
      \caption{The labelling of the \pspider\ in \cref{qn3}, where $\{d_0,d_1,d_2\} \subseteq D$ and $\{c_1,c_2\} \subseteq C$.}
      \label{psss}
    \end{figure}
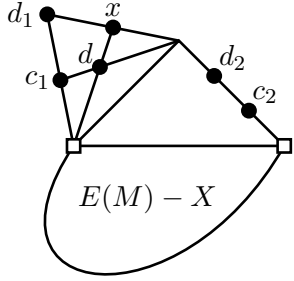

    \smallbreak

    Suppose that $E(M)-E(N) \nsubseteq X$.
    It remains to show that $|E(M)| = 10$ and \cref{wtdp1c4} holds.
    Let $(C,D)$ be an $N$-labelling of $M$ with $\{c_1,c_2\} \subseteq C$ and $\{d_0,d_1,d_2\} \subseteq D$, and observe that there exists either some $c' \in C - X$ or some $d' \in D-X$.

    First, we claim that if there is some such $c'$, and $C_3$ is a $4$-element circuit containing $c'$ and meeting $X$, then $C_3 \cap X \in \{\{x,d_0\}, \{c_1,d_1\}, \{c_2,d_2\}, \{d_0,d_1,d_2\}, \{d_0,d_1,c_2\}\}$.
    Let $C_3$ be such a circuit.
    By orthogonality with $Q_1$, $C_1^*$ and $C_2^*$, either $C_3 \cap X$ is one of $\{x,d_0\}$, $\{c_1,d_1\}$, or $\{c_2,d_2\}$; or $C_3$ has one element from each of these three sets.
    Suppose the latter. 
    If $|C_3 \cap C| \ge 3$, then, after an $N$-label switch, one of $x$, $c_1$, or $c_2$ is $N$-labelled for deletion, while retaining that $d_0$, $d_1$, or $d_2$ is $N$-labelled for deletion, respectively.
    That is, one of $\{x,d_0\}$, $\{c_1,d_1\}$, or $\{c_2,d_2\}$ is $N$-labelled for deletion, and it follows that \cref{wtdp1c2} holds by \cref{3sepwin2}.
    So $|C_3 \cap C| \le 2$.
    Then $C_3$ contains at least two of $\{d_0,d_1,d_2\}$.
    If $\{d_0,d_1\} \nsubseteq C_3$, then either $C_3=\{x,d_1,d_2,c'\}$ or $C_3=\{d_0,c_1,d_2,c'\}$; so either $C_2 \cup c'$ or $C_1 \cup c'$ is a $5$-element plane intersecting $Q_1$ in two elements, respectively, which contradicts orthogonality.
    So $C_3 \cap X \in \{\{d_0,d_1,c_2\},\{d_0,d_1,d_2\}\}$, thus proving the claim.

    By a dual argument if there exists some $d' \in D-X$, and $C_3^*$ is a $4$-element cocircuit containing $d'$ and meeting $X$, then $C_3^* \cap X \in \{\{x,d_1\}, \{c_1,d_0\}, \{c_2,d_2\}, \{x,c_1,c_2\}, \{x,c_1,d_2\} \}$.

    Suppose there exists some $d' \in D-X$.
    Then $d'$ is contained in a $4$-element cocircuit with $d_1$, and hence $x$. Let this cocircuit be $\{d',d_1,x,c'\}$ for some $c' \in E(M)-(X-d')$.
    Up to an $N$-label switch on $x$ and $c'$, the element $c'$ is $N$-labelled for contraction.
    So we may assume that there is some $c' \in C-X$.

    Let $C_3$ be the $4$-element circuit containing $\{c',c_1\}$.
    Then $C_3 \cap X = \{c_1,d_1\}$.
    So let $C_3 = \{c',c_1,d_1,z_1\}$ for some $z_1 \in E(M)-(X \cup c')$.
    By an $N$-label switch on $d_1$ and $z_1$, the element $z_1$ is $N$-labelled for deletion.

    So $\{z_1,d_2\}$ is contained in a $4$-element cocircuit~$C_3^*$.
    Either $C_3^* = \{z_1,d_2,c_1,x\}$, or $C_3^* \cap X = \{c_2,d_2\}$.
    In the latter case, it follows, by orthogonality with $C_3$ and $Q_1$, that $C_3^* = \{z_1,c_2,d_2,c'\}$.
    Thus, in either case, $r^*(X \cup \{c',z_1\}) \le r^*(X) + 1 =5$.

    Now if $c' \in \cl(X)$, then $r(X \cup \{c',z_1\})=r(X)=4$, so $\lambda(X \cup \{c',z\}) \le 4+5-8=1$.
    It follows that $|E(M)|=9$, and hence $E(M)-X$ is a triangle containing the $N$-contractible element $c'$; a contradiction.
    Thus the $4$-element circuit~$C_4$ containing $\{c',c_2\}$ is not contained in $X \cup c'$.  
    Hence $C_4 = \{c',c_2,d_2,z_2\}$ for some $z_2 \in E(M)-(X \cup c')$.
    If $z_1 = z_2$, then $\{c',z_1,c_1,c_2\}$ spans $X \cup \{c',z_1\}$, so $\lambda(X \cup \{c',z_1\}) \le 4 + 5 - 8 = 1$; as before, this is contradictory.  So $z_1 \neq z_2$.

    Let $C_4^*$ be the $4$-element cocircuit containing $\{z_2,d_1\}$.
    Then $C_4^* \cap X = \{d_1,x\}$.
    By orthogonality with $C_3$ and $C_4$, we see that $C_4^* = \{z_2,d_1,x,c'\}$,
    so $z_2 \in \cocl(X \cup c')$.
    Now $r^*(X \cup \{c',z_1,z_2\}) \le 5$, so $\lambda(X \cup \{c',z_1,z_2\}) \le 5 + 5 - 9 = 1$.  Hence $|E(M)|=10$.
    Let $Q_2 = E(M)-(Q_1 \cup \{c_2,d_2\})$ and pick $q$ so that $Q_2 = \{q,c',z_1,z_2\}$.
    It is now easily checked that $C_3$, $\{d_0,x,c',z_1\}$, $\{c_1,d_1,z_2,q\}$, and $\{d_0,x,z_2,q\}$ are circuits; $C_4^*$, $\{d_0,c_1,z_2,c'\}$, $\{d_1,x,q,z_1\}$, and $\{d_0,c_1,q,z_1\}$ are cocircuits; and $Q_1$ and $\{q,c',z_1,z_2\}$ are quads. So $Q_1 \cup Q_2$ is a \spider.
    Moreover, $C_4$ and $\{c_2,d_2,z_1,q\}$ are circuits, and $C_3^* = \{c_2,d_2,z_1,c'\}$ and $\{c_2,d_2,z_2,q\}$ are cocircuits; so $Q_2 \cup \{c_2,d_2\}$ is an \pspider.  Thus \cref{wtdp1c4} holds. 
    This completes the proof of \cref{qn3}.
  \end{slproof}

  The proof now follows from \cref{qq0,qq,qq2,qn2,qn3}. 
\end{proof}

\begin{lemma}
  \label{weaktheoremdetailed2}
  Suppose that for every $N$-labelling $(C',D')$ that is switching-equivalent to $(C,D)$, and for every pair $\{x,y\} \subseteq C'$ or $\{x,y\} \subseteq D'$, the pair $\{x,y\}$ is not contained in a quad.
  Then, up to replacing $(M,N)$ by $(M^*,N^*)$, either
  \begin{enumerate}
    \item there is some $X \subseteq E(M)$
      such that $E(M)-E(N) \subseteq X$ and $X$ is a \twisted;\label{wtd2i1}
    \item there exists $d \in E(M)$ such that $M \ba d$ is $3$-connected and has a cyclic $3$-separation $(Y, \{d'\}, Z)$ with $|Y| \ge 4$, where $M \ba d \ba d'$ has an $N$-minor with $|Y \cap E(N)| \le 1$;\label{wtd2i2}
    \item $M$ is a \tcn; or\label{wtd2i4}
    \item $|E(M)| \le 10$.\label{wtd2i3}
  \end{enumerate}
\end{lemma}

\begin{proof}
  Suppose that neither \cref{wtd2i2} nor \cref{wtd2i4} holds; we will show that \cref{wtd2i1} or \cref{wtd2i3} holds.
  Again, we will abuse notation by using $(C,D)$ to refer to an $N$-labelling that is switching-equivalent to $(C,D)$.

  \begin{sublemma}
    \label{qqcc0}
    Let $C_0$ be a $4$-element circuit containing at least two elements in $C \cup D$, and suppose there are distinct elements $d',d'' \in D \cap (\cocl(C_0)-C_0)$.  Then \cref{wtd2i2} holds.
  \end{sublemma}
  \begin{slproof}
    Since $C_0$ is not a quad, it is coindependent. Hence $C_0$ is $3$-separating in the $3$-connected matroid $M \ba d'$, and $d'' \in \cocl_{M \ba d'}(C_0)$.
    It now follows from \cref{3sepwin2} that \cref{wtd2i2} holds.
  \end{slproof}

  Since $|E(M)| - |E(N)| \ge 5$, up to duality we may assume that $|D| \ge 3$.
  Let $\{d_0,d_1,d_2\} \subseteq D$.
  Then there is a $4$-element cocircuit~$C_1^*$ containing $\{d_0,d_1\}$, and a $4$-element cocircuit~$C_2^*$ containing $\{d_0,d_2\}$.
  We start by showing that if $C_1^* \neq C_2^*$, then $r^*(C_1^* \cup C_2^*) \ge 4$.
  More generally, we prove the following:

  \begin{sublemma}
    \label{qqcc1}
    Let $d$, $d'$, and $d''$ be distinct elements in $D$, and suppose $\{d,d'\}$ is contained in a $4$-element cocircuit~$C^*$, and $\{d,d''\}$ is contained in a $4$-element cocircuit~$C_0^*$, where $C^* \neq C_0^*$.
    Then $r^*(C^* \cup C_0^*) \ge 4$.
  \end{sublemma}
  \begin{slproof}
    Let $X=C^* \cup C_0^*$, and suppose $r^*(X)=3$.
    Note that $|X| \ge 5$.
    Since each $x \in X-\{d,d',d''\}$ is a coloop in $M \ba d \ba d' \ba d''$, and this matroid has an $N$-minor, every pair contained in $X$ is $N$-deletable.
    As each $x \in X$ is not contained in a triad, $X$ is a coplane.
    Now, for distinct $d_0,d_0' \in X$, the set $\cocl(X)-\{d_0,d_0'\}$ is a series class in $M \ba d_0 \ba d_0'$, so each pair $\{c_1,c_2\} \subseteq X$ is $N$-contractible.
    Note that since $M$ is $3$-connected, $r(X) \ge |X|-1$.
    Thus $X$ contains at most one circuit.
    Let $\{c,c'\}$ be a pair of elements of $X$ not contained in such a circuit.
    Up to switching $N$-labels, we may assume that the pair $\{c,c'\}$ is $N$-labelled for contraction, and the elements in $X-\{c,c'\}$ are $N$-labelled for deletion.
    Thus $\{c,c'\}$ is contained in a circuit $C_0$, and this circuit is not contained in $X$.
    By orthogonality, $|C_0 \cap X| = 3$, and hence $C_0$ cospans $X$.
    Now $X-C_0$ contains a pair that is $N$-labelled for deletion, and is contained in $\cocl(C_0)-C_0$.
    So \cref{wtd2i2} holds by \cref{qqcc0}; a contradiction.
    We deduce that $r^*(X) \ge 4$.
  \end{slproof}

  Next we show that we may assume that $d_2 \notin C_1^*$.
  Suppose $D \subseteq C_1^*$.
  Then $C_1^* = \{c,d_0,d_1,d_2\}$ for some $c \in C$.
  Since $|E(M)| - |E(N)| \ge 5$, there exists some $c' \in C - C_1^*$, and $\{c,c'\}$ is contained in a $4$-element circuit~$C_1$.
  By orthogonality, $|C_1 \cap C_1^*| \ge 2$.
  Suppose $|C_1 \cap C_1^*| = 3$.
  Without loss of generality, let $C_1 = \{c,c',d_1,d_2\}$.
  Then we can swap the $N$-labels on $c$ and $d_0$ to deduce that there is a $4$-element circuit~$C_2$ containing $\{c',d_0\}$.
  If $C_2$ is contained in $C_1 \cup d_0$, then $r(C_1^*) = 3$, so $C_1^*$ is a quad; a contradiction.
  We deduce, by orthogonality, that $|C_2 \cap C_1^*| = 2$.
  We may now assume, up to relabelling and an $N$-label switch, that $|C_1 \cap C_1^*| = 2$.
  Without loss of generality, let $C_1 = \{x,c,c',d_2\}$, where $x \in E(M)-(C_1^* \cup c')$.
  Now, by switching the $N$-labels on $x$ and $d_2$, we obtain an element $x \in D - C_1^*$.

  We may now assume that $d_2 \notin C_1^*$.
  Let $X = C_1^* \cup C_2^*$.
  By \cref{qqcc1}, $r^*(X) \ge 4$, so $|C_1^* \cap C_2^*| \le 2$.
  We work towards \cref{qqcc4}, which handles the case where $|C_1^* \cap C_2^*| = 2$.
  First, in \cref{qqcc2}, we consider the case where $|C_1^* \cap C_2^*| = 2$ and $\lambda(X)=2$.

  \begin{sublemma}
    \label{qqcc2}
    Suppose $|C_1^* \cap C_2^*| = 2$. 
    If $X$ is $3$-separating, then $X$ is a \twisted\ of $M$ with $E(M)-E(N) \subseteq X$.
  \end{sublemma}
  \begin{slproof}
    Observe that $r^*(X)=4$, by \cref{qqcc1} and since $|C_1^* \cap C_2^*| = 2$.
    As $\lambda(X)=2$ and $|X|=6$, it follows that $r(X)=4$.
    We claim that $d_1 \notin C_2^*$.
    Suppose $d_1 \in C_2^*$.
    Then $C_1^* \cap C_2^* = \{d_0,d_1\}$, so let $C_1^* - C_2^* = \{c_1,c_2\}$.
    Since $M \ba d_0 \ba d_1 \ba d_2$ has an $N$-minor, $\{c_1,c_2\}$ is $N$-contractible.
    Now $\{c_1,c_2\} \subseteq \cl(C_2^*)-C_2^*$, since $C_2^*$ is independent and $r(X) = 4$. So \cref{wtd2i2} holds by the dual of \cref{qqcc0}; a contradiction.
    This proves the claim.

    So let $C_1^* = \{d_1,c_1,d_0,x\}$ and $C_2^* = \{d_2,c_2,d_0,x\}$ for distinct $c_1,c_2,x \in E(M)-\{d_0,d_1,d_2\}$.
    Since the pair $\{c_1,c_2\}$ is $N$-contractible, it is contained in a $4$-element circuit~$C_0$. 

    To begin with, we work under the assumption that no $4$-element circuit containing $\{c_1,c_2\}$ is contained in $X$.
    Then, by orthogonality, $C_0 = \{c_1,c_2,x',e\}$, where $x' \in \{d_0,x\}$, and $e \in E(M)-X$.
    Observe that $(X, E(M)-X)$ is a $3$-separation, and $e \in \cl(X)-X$, so $e \notin \cocl(X)$.
    Since $\{x',e\}$ is a parallel pair in $M/c_1/c_2$, by possibly swapping the $N$-labels on $x'$ and $e$ we deduce that the pairs $\{e,d_1\}$ and $\{e,d_2\}$ are $N$-deletable.
    These pairs are contained in $4$-element cocircuits $C_3^*$ and $C_4^*$ respectively.
    Neither of these cocircuits is contained in $X \cup e$, otherwise $e \in \cocl(X)$; a contradiction.
    Moreover, each of these cocircuits meets $C_0$, so they do so in at least two elements, by orthogonality.

    Since $C_1^*$ and $C_2^*$ are independent and $r(X)=4$, the element $c_2$ is in a circuit~$C_2$ contained in $C_1^* \cup c_2$, and $c_1$ is in a circuit~$C_1$ contained in $C_2^* \cup c_1$.
    The circuit~$C_1$ (or $C_2$) can only intersect $C_3^*$ (or $C_4^*$, respectively) in at most one element, so, by orthogonality, $C_3^* \cap C_1 = \emptyset$ and $C_4^* \cap C_2 = \emptyset$.
    In particular, $c_1 \notin C_3^*$ and $c_2 \notin C_4^*$.
    Since $X$ does not contain any triangles, $C_1=X-C_3^*$ and $C_2=X-C_4^*$.
    If $x' \in C_3^* \cap C_4^*$, then $C_3^* \cap C_1=\{d_1\}$; a contradiction to orthogonality.
    Without loss of generality, we may now assume that $x' \notin C_4^*$, in which case $c_1 \in C_4^*$.

    Suppose $x' \in C_3^*$.
    Let $C_3^*-\{e,d_1,x'\} = \{f\}$.
    Since $M \ba d_1 \ba d_2 / c_1 / c_2$ has an $N$-minor, and $\{e,x'\}$ is a parallel pair in this matroid, $M \ba d_1 \ba e / c_2$ has an $N$-minor, implying $M / c_2 /f$ has an $N$-minor.
    So $\{c_2,f\}$ is contained in a $4$-element circuit~$C_2'$.
    Observe that $C_2'$ is not contained in $X \cup \{e,f\}$, for otherwise $f \in \cl(X \cup e) \cap \cocl(X \cup e)$, so $\lambda(X \cup \{e,f\}) \le 1$ and $|E(M)| \le 9$; a contradiction. 
    Now, by orthogonality, $C_2'$ meets $C_2^*-c_2$ and $C_3^*-f$.
    As the intersection of these last two sets is $\{x'\}$, we deduce $x' \in C_2'$.
    But then $C_2'$ also has an element in $C_1^*-x'$, by orthogonality; a contradiction.
    So $x' \notin C_3^*$, hence $c_2 \in C_3^*$.
    Now $C_3^* = \{e,d_1,c_2,f\}$ and $C_4^* = \{e,d_2,c_1,f'\}$ for some $f,f' \in E(M)-(X \cup e)$, and $C_1 = \{x,d_0,c_1,d_2\}$ and $C_2 = \{x,d_0,c_2,d_1\}$.

    Recall that $C_0 = \{c_1,c_2,x',e\}$, where $x' \in \{d_0,x\}$.
    If $x' = x$, then $M \ba d_0 \ba d_1 \ba d_2 /c_1/c_2$ has an $N$-minor and $\{x,e\}$ is a parallel pair in this matroid, so $M \ba d_0 \ba d_1 \ba d_2 \ba x$ has an $N$-minor.  But $c_2$ and $c_1$ are coloops in this matroid, so we may assume that $X \subseteq D$, in which case there is a pair of $N$-deletable elements in $\cocl(C_2)-C_2$, 
    so \cref{wtd2i2} holds by \cref{qqcc0}.  Hence $x'=d_0$.

    The pair $\{d_1,d_2\}$ is also contained in a $4$-element cocircuit~$C_5^*$.  
    By orthogonality with the circuits $C_1$ and $C_2$, either $C_5^* = C_1^* \triangle C_2^*$, or $C_5^*$ meets $C_1^* \cap C_2^*=\{x,d_0\}$.

    We start with the former case.
    Recall that $C_3^* = \{e,d_1,c_2,f\}$ and $C_4^* = \{e,d_2,c_1,f'\}$.
    Recall also that $M \ba d_1 \ba d_2 \ba e$ has an $N$-minor, and observe that $\{c_1,c_2\}$ and $\{c_1,f'\}$ are both series pairs in this matroid, due to the cocircuits $C_5^*$ and $C_4^*$.
    Now $\{c_1,f'\}$ is $N$-contractible, so this pair is contained in a $4$-element circuit~$C_1'$.
    This circuit is not contained in $X \cup \{e,f'\}$, otherwise $f' \in \cl(X \cup e)$ in which case $|E(M)|=9$. 
    By orthogonality, $C_1'$ meets $C_1^*-c_1 = \{d_1,x,d_0\}$, and $C_5^*-c_1=\{d_1,d_2,c_2\}$.  So $d_1 \in C_1'$.
    Now, by orthogonality with $C_3^*$, we see that $C_1'$ contains an element in $\{e,c_2,f\}$.  If $f \notin C_1'$, then $f' \in \cl(X \cup e)$; a contradiction.
    Suppose $f \neq f'$.
    Then $C_1' = \{c_1,f',d_1,f\}$.
    Recall that the pair $\{c_1,f'\}$ is $N$-contractible in $M \ba e$.
    As $\{d_1,f\}$ is a parallel pair in $M \ba e/c_1 /f'$, the matroid $M \ba e \ba f/c_1$ has an $N$-minor.
    In turn, we see that $\{d_1,c_2\}$ is a series pair in this matroid, due to the cocircuit $C_3^*$, so the pair $\{c_1,d_1\}$ is $N$-contractible.
    As $\{c_1,d_1\} \subseteq \cl(C_2^*)-C_2^*$, it follows that \cref{wtd2i2} holds, by \cref{qqcc0}.
    Now $f=f'$, and $C_1' = \{c_1,d_1,f,g\}$ for some $g \in E(M)-(X \cup \{e,f\})$.
    Since $M \ba e/c_1/f$ has an $N$-minor, and $\{d_1,g\}$ is a parallel pair in this matroid, $\{e,g\}$ is contained in a $4$-element cocircuit.
    By orthogonality with $C_0$, $C_1$, and $C_2$, this cocircuit is contained in $X \cup \{e,g\}$.
    Hence $g \in \cl(X \cup \{e,f\}) \cap \cocl(X \cup \{e,f\})$, implying $\lambda(X \cup \{e,f,g\}) \le 1$, so $|E(M)| \in \{9,10\}$; a contradiction.

    So we may assume that $C_5^*$ meets $\{d_0,x\}$.
    Suppose $d_0 \in C_5^*$.
    Then, by orthogonality with $C_0$, the cocircuit $C_5^*$ meets $\{c_1,c_2,e\}$.
    If $C_5^*$ meets $\{c_1,c_2\}$, it follows that $r^*(X)=3$; on the other hand, if $e \in C_5^*$, then $e \in \cocl(X)$; either case is contradictory.
    We deduce that $x \in C_5^*$.
    Let $C_5^* = \{d_1,d_2,x,f\}$.
    Since $M \ba d_1 \ba d_2 / c_1 / c_2$ has an $N$-minor, and in this matroid $\{x,f\}$ is a series pair and $\{e,d_0\}$ is a parallel pair, we see that $M / c_1 / x \ba e$ has an $N$-minor.
    Now $\{d_0,d_2\}$ is a parallel pair in this matroid, due to the circuit $C_1$, so $M \ba e \ba d_0$ has an $N$-minor.
    Hence $\{e,d_0\}$ is contained in a $4$-element cocircuit~$C_6^*$.
    This cocircuit cannot be contained in $X \cup e$, as $e \notin \cocl(X)$.
    By orthogonality with the circuits $C_1$ and $C_2$, we deduce that $x \in C_6^*$.
    However, by circuit elimination on $C_1$ and $C_2$, the set $X-x$ contains a circuit.
    By orthogonality with $C_6^*$, we deduce that $C_1^* \triangle C_2^*$ is a circuit.
    But this contradicts our assumption that no $4$-element circuit containing $\{c_1,c_2\}$ is contained in $X$.

    \medbreak
    Now we may assume that $\{c_1,c_2\}$ is contained in a $4$-element circuit $C_0 \subseteq X$.
    Since $C_0$ is not a quad, $C_0$ is coindependent.
    If $x \in C_0$, then $X-C_0$ is a pair that is $N$-labelled for deletion, and in the coclosure of the circuit $C_0$, in which case \cref{wtd2i2} holds by \cref{qqcc0}.
    So we may assume that $x \notin C_0$.
    Now, either $C_0 = C_1^* \triangle C_2^*$, or we may assume, by symmetry (that is, up to swapping $d_1$ and $d_2$), that $C_0 = \{c_1,d_0,c_2,d_2\}$.

    Suppose we are in the former case, where $C_0 = C_1^* \triangle C_2^*$ is a circuit.
    Suppose $\{c_1,c_2,x,d_0\}$ is a circuit.
    Then $\{d_1,d_2\} \subseteq \cocl(\{c_1,c_2,x,d_0\})$, and the pair $\{d_1,d_2\}$ is $N$-labelled for deletion.  Thus \cref{wtd2i2} holds by \cref{qqcc0}.
    So $\{c_1,c_2,x,d_0\}$ is not a circuit.

    Next, we claim that $X-x$ and $X-d_0$ are cocircuits.
    Certainly, as $C_0$ is coindependent and cospans $X$, each of the sets $X-x$ and $X-d_0$ contains a cocircuit, and these cocircuits contain $d_0$ and $x$, respectively.  If either of these sets properly contains a cocircuit, it follows that $r^*(X)=3$; a contradiction.  So $X-x$ and $X-d_0$ are cocircuits, as claimed.

    Let $C_0^*$ be the $4$-element cocircuit that contains $\{d_1,d_2\}$.
    As $C_2^* \cup c_1$ contains a circuit, and this circuit is not $\{c_1,c_2,x,d_0\}$, the circuit must contain $d_2$.
    So $C_0^*$ meets $X-\{d_1,d_2\}$, by orthogonality.
    Now, if $C_0^* \subseteq X$, then $|C_0^*-C_1^*|>1$ and $|C_0^*-C_2^*| > 1$, otherwise $r^*(X)=3$.
    Moreover, $C_0^* \neq C_0$, since $C_0$ is coindependent.
    So $C_0^* \nsubseteq X$; hence
    $|C_0^* \cap X| = 3$.
    Let $C_0^*-X = \{f\}$.

    Since $f \in \cocl(X)$ and $|E(M)-(X \cup f)| \ge 2$, the set $X \cup f$ is exactly $3$-separating.  So $f \notin \cl(X)$.
    In $M \ba d_0 \ba d_1 \ba d_2$, the set $\{c_1,c_2,x,f\}$ is contained in a series class.  So $\{f,x\}$ is $N$-contractible, and thus is contained in a $4$-element circuit~$C_0'$. 
    By orthogonality with the cocircuits $C_1^*$ and $C_2^*$, 
    either $C_0' \subseteq X \cup f$ or $d_0 \in C_0'$.
    Since $f \notin \cl(X)$, we deduce that $d_0 \in C_0'$.
    But now $\{f,x,d_0\}$ meets the cocircuit $X-x$, so, by orthogonality, $C_0'$ also meets $X-\{x,d_0\}$, implying $f \in \cl(X)$; a contradiction.

    \smallbreak
    Now, suppose that $C_0 = \{c_1,d_0,c_2,d_2\}$.
    Since $\{d_1,d_2\}$ is $N$-labelled for deletion, $\{d_1,d_2\}$ is contained in a $4$-element cocircuit~$C_0^*$ that, by orthogonality, intersects $C_0$ in at least two elements.
    Suppose $C_0^*$ is not contained in $X$.
    If $d_0 \in C_0^*$, then $\{d_1,c_1\}$ is an $N$-contractible pair in the closure of the independent cocircuit~$C_2^*$, so \cref{wtd2i2} holds by the dual of \cref{qqcc0}.
    So $C_0^*=\{d_1,d_2,c',f\}$ for some $c' \in \{c_1,c_2\}$ and $f \in E(M)-X$.
    Since $\{c_1,c_2,x,f\}$ is contained in a series class in $M \ba d_0 \ba d_1 \ba d_2$, the pair $\{f,x\}$ is contained in a circuit~$C_0'$. 
    Since $X$ and $X \cup f$ are exactly $3$-separating, and $f \in \cocl(X)$, we have $f \notin \cl(X)$.
    Hence, the circuit~$C_0'$ is not contained in $X \cup f$.
    It follows, by orthogonality, that $d_0 \in C_0'$, and $C_0' \cap X = \{x,d_0\}$. But then we get a contradiction to orthogonality using the cocircuit $C_0^*=\{d_1,d_2,c',f\}$.
    So $C_0^* \subseteq X$.  Since $r^*(X)=4$, it follows that $C_0^* = C_1^* \triangle C_2^*$.

    Let $C_1$ and $C_2$ be the $4$-element circuits containing $\{c_1,x\}$, and $\{c_2,x\}$ respectively.
    By orthogonality with $C_0^*$ and $C_1^*$, and since $C_2 \neq C_2^*$ as $C_2^*$ is independent, $C_2$ meets $\{c_1,d_1\}$.
    Note also that $C_2 \neq \{c_1,c_2,x,d_1\}$, for otherwise the pair $\{d_0,d_2\}$ is $N$-labelled for deletion and contained in $\cocl(\{c_1,c_2,x,d_1\})$, so that \cref{wtd2i2} holds by \cref{qqcc0}.
    Now, if $c_1 \in C_2$, then, as $M \ba d_1 / c_1 / c_2$ has an $N$-minor, $\{d_1,x\}$ is $N$-deletable.
    As $\{d_1,x\} \subseteq \cocl(C_0)-C_0$, it follows that \cref{wtd2i2} holds, by \cref{qqcc0}; a contradiction.
    So $d_1 \in C_2$.
    By a similar argument with $C_1$, we deduce that $C_1$ meets $\{c_2,d_2\}$, and $c_2 \notin C_1$, so $d_2 \in C_1$.

    Suppose $C_2 \nsubseteq X$.  Let $C_2 = \{c_2,x,d_1,e\}$, where $e \in E(M)-X$.
    Since $M \ba d_0 \ba d_2 / c_2 / x$ has an $N$-minor, and $\{d_1,e\}$ is a parallel pair in this matroid, $\{d_0,d_2,e\}$ is $N$-deletable.
    Now $\{d_0,e\}$ and $\{d_2,e\}$ are contained in $4$-element cocircuits $C_3^*$ and $C_4^*$ respectively. 
    If $C_3^*$ or $C_4^*$ is contained in $X \cup e$, then $e \in \cocl(X) \cap \cl(X)$; a contradiction.
    So $C_3^*$ and $C_4^*$ each contain some element in $E(M)-(X \cup e)$.
    By orthogonality with $C_0$ and $C_2$, the cocircuit $C_3^*$ meets $\{c_1,c_2,d_2\}$ and $\{c_2,x,d_1\}$; so $c_2 \in C_3^*$.
    Now $X \cap C_3^* = \{d_0,c_2\}$.  Since $C_1^* \cup d_2$ contains a circuit consisting of at least four elements, this circuit is $\{c_1,d_1,x,d_2\}$, by orthogonality.
    But now $C_4^*$ meets $\{c_1,d_1,x\}$ and $C_0-d_2=\{c_1,c_2,d_0\}$ and $C_2-e=\{c_2,x,d_1\}$, by orthogonality.
    As no element is in the intersection of these three sets, $C_4^* \subseteq X \cup e$; a contradiction.

    We deduce that $C_2 \subseteq X$.
    Recall that $c_1 \notin C_2$.  So let $C_2 = \{c_2,x,d_1,e\}$ where $e \in \{d_2,d_0\}$.
    We will show that 
    if $e = d_0$, then $X$ is a \twisted\ (see \cref{tv2main}).  But first we consider the case where $e=d_2$.

    Let $C_2 = \{c_2,x,d_1,d_2\}$.
    Since $r(X)=4$ and $C_1^*$ is independent, $c_2$ is in a circuit that is contained in $C_1^* \cup c_2$.
    If this circuit is properly contained in $C_1^* \cup c_2$, then it consists of four elements, and hence it contains either $\{d_1,x\}$ and one of $c_1$ and $d_0$, or $\{c_1,d_0\}$ and one of $d_1$ and $x$.
    In any such case, due to the circuits $C_0$ and $C_2$, it follows that $r(X)=3$; a contradiction.
    So $C_1^* \cup c_2$ is a circuit.
    Similarly, we deduce that $C_1^* \cup d_2$ is a circuit.
    Now $X$ is in fact a \tvamoslike\ of $M^*$, but there is still the circuit~$C_1$ containing $\{c_1,x,d_2\}$ to consider.
    If $C_1 \subseteq X$, then $r(X)=3$; a contradiction.
    So let $C_1 = \{c_1,x,d_2,g\}$, where $g \in E(M)-X$.

    As $g \in \cl(X)$, and $X$ and $X \cup g$ are exactly $3$-separating, we have $g \notin \cocl(X)$.
    Since $M \ba d_0 \ba d_1 / c_1/x$ has an $N$-minor, it follows that $\{g,d_0\}$ is $N$-deletable in $M \ba d_1$ and $M$.
    So there is a $4$-element cocircuit~$C_3^*$ containing $\{g,d_0\}$
    that is not contained in $X \cup g$.
    By orthogonality, $C_3^*$ meets $C_0-d_0=\{c_1,c_2,d_2\}$ and $C_1-g = \{c_1,x,d_2\}$.
    As $C_3^* \nsubseteq X \cup g$, we deduce that $C_3^*$ meets $\{c_1,d_2\}$.
    But if $d_2 \in C_3^*$, then $C_3^*$ intersects $C_2$ in two elements; a contradiction.
    So $C_3^* = \{g,d_0,c_1,h\}$ for some $h \in E(M)-(X \cup g)$.


    Recall that $M \ba d_0 \ba d_1 \ba g$ has an $N$-minor.
    As $\{c_1,x\}$ and $\{c_1,h\}$ are series pairs in this matroid, we see that $M/x/h$ has an $N$-minor.
    So the pair $\{x,h\}$ is contained in a $4$-element circuit~$C_3$ that, by orthogonality, meets $C_1^*-x$ and $C_2^*-x$. 
    Since $h \in \cocl(X \cup g)$, the set $X \cup \{g,h\}$ is $3$-separating.
    Moreover, either $X \cup \{g,h\}$ is exactly $3$-separating, 
    for otherwise $|E(M)| = 9$.
    So we may assume that $h \notin \cl(X \cup g)$; in particular, $C_3 \nsubseteq X \cup \{g,h\}$.
    Since $d_0$ is the only element in the intersection of $C_1^*-x$ and $C_2^*-x$, we see that the $d_0 \in C_3$.
    Let $C_3=\{x,h,d_0,q\}$ for $q \in E(M)-(X \cup \{g,h\})$.

    Now, $q \in \cl(X \cup \{g,h\})$, so $X \cup \{g,h,q\}$ is $3$-separating.
    If this set is not exactly $3$-separating, then $|E(M)| = 10$.
    So we may assume that $q \notin \cocl(X \cup \{g,h\})$.
    Recall that $M \ba d_1 / x / h$ has an $N$-minor.
    It follows that $\{q,d_1\}$ is $N$-deletable, implying that this pair is contained in a $4$-element cocircuit.
    By orthogonality, this cocircuit meets $C_2-d_1 = \{x,c_2,d_2\}$.  But if it meets $\{c_2,d_2\}$, then it intersects $C_0=\{c_1,d_0,c_2,d_2\}$ in at least two elements.
    Then $q$ is in a cocircuit contained in $X \cup q$, so $q \in \cocl(X \cup \{g,h\})$; a contradiction.
    So the cocircuit contains $x$.
    Hence it meets $C_1-x = \{c_1,d_2,g\}$.
    Again, we deduce the contradiction that $q \in \cocl(X \cup \{g,h\})$.


    \smallbreak
    Finally, we let $C_2 = \{c_2,x,d_1,d_0\}$.
    Recall the $4$-element circuit $C_1$ containing $\{c_1,x,d_2\}$.  Let $C_1 = \{c_1,x,d_2,g\}$.
    Suppose that $g \notin X$.
    Since $M \ba d_0 / c_1/x$ has an $N$-minor, $\{g,d_0\}$ is $N$-deletable.
    So $\{g,d_0\}$ is contained in a $4$-element cocircuit that, by orthogonality with $C_0$, $C_1$, and $C_2$, meets $\{c_1,c_2,d_2\}$, $\{c_1,x,d_2\}$ and $\{c_2,x,d_1\}$.
    Since no element is contained in all three of these sets, $g \in \cocl(X)-X$.  But $X$ is $3$-separating and $g \in \cl(X)$, so this is contradictory.
    Hence $g \in X$.

    Recall that $c_2 \notin C_1$, so $g \in \{d_0,d_1\}$.  But if $g = d_0$, then $x \in \cl(\{c_1,d_0,d_2\})$, and, due to the circuits $C_0$ and $C_2$, it follows that $\{c_1,d_0,d_2\}$ spans $X$, and $r(X) < 4$; a contradiction.
    Thus $g = d_1$.
    Now $X$ is a \twisted\ of $M$, labelled as in \cref{tv2main}.

    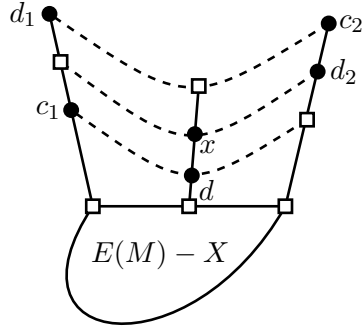
\begin{figure}[htb]
      \centering
      \begin{tikzpicture}[rotate=90,scale=0.64,line width=1pt]
        \tikzset{VertexStyle/.append style = {minimum height=5,minimum width=5}}
        \clip (-2.5,-6) rectangle (4.4,2);
        \node at (-1,-1.4) {$E(M)-X$};
        \draw (0,0) .. controls (-3,2) and (-3.5,-2) .. (0,-4);
        \draw (0,0) -- (4.0,0.9);
        \draw (0,-2) -- (2.5,-2.2); 
        \draw (0,-4) -- (3.8,-4.9); 

        \Vertex[x=4.0,y=0.9,LabelOut=true,Lpos=180,L=$d_1$]{a2}
        \Vertex[x=2.0,y=0.45,LabelOut=true,Lpos=180,L=$c_1$]{a3}

        \Vertex[x=1.5,y=-2.12,LabelOut=true,Lpos=-45,L=$x$]{b1}
        \Vertex[x=0.64,y=-2.056,LabelOut=true,Lpos=-45,L=$d_0$]{b2}

        \Vertex[x=3.8,y=-4.9,LabelOut=true,L=$c_2$]{c1}
        \Vertex[x=2.8,y=-4.67,LabelOut=true,L=$d_2$]{c2}

        \draw[dashed] (3.8,-4.9) .. controls (2.0,-2) .. (4.0,0.9);
        \draw[dashed] (2.8,-4.67) .. controls (1.0,-2) .. (3.0,0.67);
        \draw[dashed] (1.8,-4.45) .. controls (0.25,-2) .. (2.0,0.45);

        \draw (0,0) -- (0,-4);

        \SetVertexNoLabel
        \tikzset{VertexStyle/.append style = {shape=rectangle,fill=white}}
        \Vertex[x=3.0,y=0.67]{a1}
        \Vertex[x=2.5,y=-2.2]{b3}
        \Vertex[x=1.8,y=-4.45]{c3}

      \end{tikzpicture}
      \caption{The labelling of the \twisted\ of $M$ that arises in \cref{qqcc2}, where $\{d_0,d_1,d_2\} \subseteq D$ and $\{c_1,c_2\} \subseteq C$.}
      \label{tv2main}
    \end{figure}

    \smallbreak

    It remains to show that $E(M)-E(N) \subseteq X$.
    Towards a contradiction,
    suppose the $N$-labelling $(C,D)$ of $M$, with $\{c_1,c_2\} \subseteq C$ and $\{d_0,d_1,d_2\} \subseteq D$, also has either some element $c' \in C - X$ or $d' \in D-X$.

    First, we show that if there is some such $c'$, then any $4$-element circuit containing $c'$ and an element of $C \cap X$ is not contained in $X \cup c'$.
    Let $c'$ be such an element, and consider the $4$-element circuit~$C_3$ containing $\{c',c_1\}$.
    Suppose $C_3 \subseteq X \cup c'$.
    By orthogonality with $C_0^*$, $C_1^*$, and $C_2^*$, the circuit~$C_3$ is $\{c',c_1,a,b\}$ for some $a \in \{x,d_0\}$ and $b \in \{c_2,d_2\}$.
    But if $a=x$, then, up to an $N$-label switch on $x$ and $b$, the pair $\{x,d_1\}$ is $N$-labelled for deletion, and $\{x,d_1\} \subseteq \cocl(C_0)-C_0$, so \cref{wtd2i2} holds by \cref{qqcc0}.
    So $a=d_0$, and, similarly, $b=d_2$.
    Now $C_3 = \{c_1,c',d_0,d_2\}$, so $C_0 \cup C_3$ is a $5$-element plane that intersects the cocircuit $C_1^*$ in two elements; a contradiction.
    So $c'$ is not in a $4$-element circuit contained in $X \cup c'$, as claimed.
    Dually, if there exists some $d' \in D-X$, then any $4$-element cocircuit containing $d'$ and an element in $D \cap X$ is not contained in $X \cup d'$.

    Next we show that if there is some $d' \in D-X$, then there is also some $c' \in C-X$.
    Suppose $d' \in D-X$. Then $d'$ is contained in a $4$-element cocircuit with $d_1$, and an element $c' \in E(M)-X$. 
    By orthogonality, the final element in this cocircuit is $x$.  Up to an $N$-label switch on $x$ and $c'$, the element $c'$ is $N$-labelled for contraction.  

    So we may assume there is an element $c' \in C-X$.
    Since the $4$-element circuit~$C_3$ containing $\{c',c_1\}$ is not contained in $X \cup c'$, it follows from orthogonality that $C_3=
    \{c',c_1,d_1,z_1\}$ for some $z_1 \in E(M)-(X \cup c')$.
    Similarly, the $4$-element circuits containing $\{c',c_2\}$ and $\{c',x\}$ are $C_4 = \{c',c_2,d_2,z_2\}$ and $C_5 = \{c',x,d_0,z\}$ for some $z,z_2 \in E(M)-(X \cup c')$.

    By an $N$-label switch on $d_1$ and $z_1$, we may assume $z_1 \in D$.
    So there is a $4$-element cocircuit $C_3^*$ containing $\{z_1,d_0\}$.
    As this cocircuit is not contained in $X \cup z_1$, it follows from orthogonality that $c_2 \in C_3^*$.
    Now, by orthogonality with $C_3$, the final element is in $\{c_1,d_1,c'\}$.
    Again using that $C_3^* \nsubseteq X \cup z_1$, we deduce that $C_3^* = \{z_1,d_0,c_2,c'\}$.
    In particular, $r^*(X \cup \{c',z_1\}) \le 5$.
    By symmetry, we deduce that $r^*(X \cup \{c',z_2\}) \le 5$ and $r^*(X \cup \{c',z\}) \le 5$.

    Suppose that $z_1=z_2$.  Then $\{z_1,c',c_1,c_2\}$ spans $X \cup \{c',z_1\}$, so $r(X \cup \{c',z_1\}) = 4$. 
    But $r^*(X \cup \{c',z_1\}) \le 5$, so $\lambda(X \cup \{c',z_1\})\le 1$.
    It follows that $|E(M)|=9$; 
    a contradiction.
    So $z_1\neq z_2$.
    By symmetry, $z$, $z_1$, and $z_2$ are pairwise distinct.

    Now consider the cocircuit $C_4^*$ containing $\{z_1,d_2\}$.
    Using orthogonality and the fact that $C_4^* \nsubseteq X \cup z_1$, we have $c_1 \in C_4^*$.
    Then, by orthogonality with $C_4$, and since $C_4^* \nsubseteq X \cup z_1$ again,
    the cocircuit $C_4^*$ meets $\{c',z_2\}$.
    But then $C_4^*$ intersects $C_5$ in at most one element, $c'$, so $C_4^* = \{z_1,d_2,c_1,z_2\}$.
    Now $r^*(X \cup \{c',z_1,z_2\}) \le 5$.
    As $r(X \cup \{c',z_1,z_2\}) \le 5$, we have $\lambda(X \cup \{c',z_1,z_2\}) \le 1$.
    It follows that $|E(M)|=10$; 
    a contradiction.
    This completes the proof of \cref{qqcc2}.
  \end{slproof}

  \begin{sublemma}
    \label{qqcc25}
    Suppose $|C_1^* \cap C_2^*| = 2$.
    If $C_1^* \triangle C_2^*$ is a circuit, then the lemma holds.
  \end{sublemma}
  \begin{slproof}
    Recall that $X = C_1^* \cup C_2^*$, and observe that $r^*(X)=4$, by \cref{qqcc1} and since $|C_1^* \cap C_2^*| = 2$.
    Suppose $C_0 = C_1^* \triangle C_2^*$ is a circuit.

    We claim that $d_1 \notin C_2^*$.
    Towards a contradiction, suppose $d_1 \in C_2^*$.
    Recall that $d_2 \notin C_1^*$.
    So, without loss of generality, let $C_1^* = \{d_0,d_1,c_1,c'\}$ and $C_2^* = \{d_0,d_1,d_2,c_2\}$ where $c_1$ and $c_2$ are $N$-labelled for contraction.
    Since $C_0$ contains $\{c_1,c_2\}$, it is a coindependent circuit, so it cospans $X$.
    Hence $\{d_0,d_1\} \subseteq \cocl(C_0)-C_0$, and it follows that \cref{wtd2i2} holds by \cref{qqcc0}; a contradiction.
    So $d_1 \notin C_2^*$, as claimed.

    Now let $C_1^* = \{d_1,c_1,d_0,x\}$ and $C_2^* = \{d_2,c_2,d_0,x\}$ for distinct $c_1,c_2,x \in E(M)-\{d_0,d_1,d_2\}$.
    Since $M \ba d_0 \ba d_1 \ba d_2$ has an $N$-minor, and $\{x,c_1,c_2\}$ is contained in a series class in this matroid, the pairs $\{x,c_1\}$ and $\{x,c_2\}$ are $N$-contractible.
    Let $C_1$ and $C_2$ be the $4$-element circuits containing $\{x,c_1\}$ and $\{x,c_2\}$ respectively.  Observe that $|C_1 \cap X| \ge 3$ and $|C_2 \cap X| \ge 3$, by orthogonality.

    If $X$ is $3$-separating, then \cref{wtd2i1} holds by \cref{qqcc2}.
    So $\lambda(X) \ge 3$.
    Since $r^*(X)=4$ and $X$ contains a circuit, $r(X) =5$.
    Hence, $C_1 \nsubseteq X$ and $C_2 \nsubseteq X$, so $|C_1 \cap X| = |C_2 \cap X| = 3$.
    So let $C_1 = \{c_1,x,x_1,g_1\}$ and $C_2 = \{c_2,x,x_2,g_2\}$ where $x_1,x_2 \in X$ and $g_1,g_2 \in E(M)-X$.
    If $x_1 = c_2$, then $\{d_0,x\} \subseteq D$ up to an $N$-label switch, and $\{d_0,x\} \subseteq \cocl(C_0)-C_0$ since $C_0$ is coindependent, in which case \cref{wtd2i2} holds by \cref{qqcc0}.
    So we may assume that $c_2 \notin C_1$ and, by symmetry, $c_1 \notin C_2$.
    Hence, by orthogonality, $x_1 \in \{d_0,d_2\}$ and $x_2 \in \{d_0,d_1\}$.
    If $g_1 \in \cocl(X)$, then the coindependent circuit~$C_1$ cospans $X \cup g_1$.
    In this case, the pair $\{d_0,d_1,d_2\} - x_1$ is $N$-labelled for deletion, and contained in $\cocl(C_1)-C_1$.  Thus \cref{wtd2i2} holds by \cref{qqcc0}.
    So we may assume that $g_1 \notin \cocl(X)$.
    Similarly, $g_2 \notin \cocl(X)$.

    Suppose that $x_1 = d_2$, so $C_1 = \{c_1,x,d_2,g_1\}$.
    Let $C_0^*$ be the $4$-element cocircuit containing $\{d_1,d_2\}$.
    By orthogonality, $C_0^*$ meets $\{c_1,x,g_1\}$.

    We claim that $g_1 \notin C_0^*$.  Towards a contradiction, suppose $g_1 \in C_0^*$.  Then, as $g_1 \notin \cocl(X)$, we have $C_0^* = \{d_1,d_2,g_1,f\}$ for some $f \in E(M)-(X \cup g_1)$.
    As $M \ba d_0 \ba d_1 /c_1/x$ has an $N$-minor, and $\{d_2,g_1\}$ is a parallel pair in this matroid, $M \ba d_0 \ba d_1 \ba g_1 / c_1$ has an $N$-minor.
    Now $\{d_2,f\}$ is a series pair in this matroid, so $M \ba d_0 / c_1 / d_2$ has an $N$-minor.
    Finally, $\{x,g_1\}$ is a parallel pair in this matroid, so $\{d_0,x\}$ is $N$-deletable.
    But $\{d_0,x\} \subseteq \cocl(C_0)-C_0$, so \cref{wtd2i2} holds in this case, by \cref{qqcc0}.
    So we may assume that $g_1 \notin C_0^*$.

    Recall that $x_2 \in \{d_0,d_1\}$;
    we now distinguish two subcases, depending on whether $x_2=d_0$ or $x_2=d_1$.
    We first handle the case that $x_2 = d_1$, so $C_2 = \{c_2,x,d_1,g_2\}$.
    Then, by orthogonality, $C_0^*$ meets $\{c_2,x,g_2\}$.
    By the argument in the previous paragraph, $g_2 \notin C_0^*$.
    Hence $C_0^*$ meets $\{c_1,x\}$ and $\{c_2,x\}$.
    Since $r^*(X)=4$ and $C_0$ is coindependent, $C_0^* = \{d_1,d_2,x,f\}$ for some $f \in E(M)-(X \cup \{g_1,g_2\})$.
    As any pair in $\cocl(X)-\{d_0,d_1,d_2\}$ is $N$-contractible, and 
    $\{f,x\}$ is such a pair, $\{f,x\}$ is contained in a $4$-element circuit~$C_3$.
    By orthogonality, this circuit meets $C_1^* - x$ and $C_2^* - x$.
    Thus, if $C_3 \nsubseteq X \cup f$, then $d_0 \in C_3$.
    However, $X-x$ also contains a cocircuit, by cocircuit elimination on $C_1^*$ and $C_2^*$, and this cocircuit contains $d_0$, since $C_0$ is coindependent.  So $C_3 \subseteq X \cup f$.
    Now $C_3$ cospans $X \cup f$, so if there is a pair of $N$-deletable elements in $(X \cup f) - C_3$, then \cref{wtd2i2} holds by \cref{qqcc0}.
    So $C_3 = \{x,f,d',d''\}$ for some distinct $d',d'' \in \{d_0,d_1,d_2\}$.
    Moreover, as $C_0^*$ is independent, $C_3 \neq C_0^*$.  So, without loss of generality, $C_3 = \{x,f,d_0,d_1\}$.

    As $\{c_1,c_2,x,f\}$ is contained in a series class in $M \ba d_0 \ba d_1 \ba d_2$, the matroid $M / c_1/c_2/x$ has an $N$-minor.
    As $\{g_1,d_2\}$ and $\{d_1,g_2\}$ are parallel pairs in this matroid, it follows that $M \ba g_1 \ba d_1$ has an $N$-minor.
    So the pair $\{g_1,d_1\}$ is contained in a $4$-element cocircuit~$C_3^*$.
    By orthogonality with $C_0$ and $C_1$, this cocircuit meets $\{c_1,c_2,d_2\}$ and $\{c_1,x,d_2\}$.
    Thus, if $C_3^*$ does not meet $\{c_1,d_2\}$, then $C_3^* \subseteq X \cup g_1$, implying $g_1 \in \cocl(X)$; a contradiction.
    So $C_3^*$ meets $\{c_1,d_2\}$.
    But, by orthogonality with $C_2$ and $C_3$, the cocircuit $C_3^*$ also meets $\{c_2,x,g_2\}$ and $\{x,f,d_0\}$.
    Thus $x \in C_3^*$, and again we arrive at the contradiction that $g_1 \in \cocl(X)$.
    This completes the case where $x_1 = d_2$ and $x_2 = d_1$.

    We now consider the subcase where $x_2=d_0$, so $C_2=\{c_2,x,d_0,g_2\}$.
    Recall that $C_1 = \{c_1,x,d_2,g_1\}$, and the $4$-element cocircuit $C_0^*$ that contains $\{d_1,d_2\}$ meets $\{c_1,x\}$.
    As $r^*(X) = 4$ and $C_0^* \neq C_0$ (as $C_0$ is coindependent), $C_0^* \nsubseteq X$.
    If $x \in C_0^*$, then $C_0^*$ also meets $C_2-x = \{c_2,d_0,g_2\}$, by orthogonality, so $C_0^* = \{d_1,d_2,x,g_2\}$, implying $g_2 \in \cocl(X)$; a contradiction.
    So $C_0^* = \{d_1,d_2,c_1,f\}$ where $f \in E(M)-X$.

    Since $M \ba d_0 \ba d_1 \ba d_2 / x$ has an $N$-minor, and $\{c_1,f\}$ is a series pair in this matroid, $\{f,x\}$ is $N$-contractible in $M \ba d_0$ and $M$.
    So $\{f,x\}$ is contained in a $4$-element circuit~$C_4$.
    As $C_4$ cospans $X \cup f$, if $|\{d_0,d_1,d_2\}-C_4| \ge 2$, then there is a pair of $N$-deletable elements in $\cocl(C_4)-C_4$, so \cref{wtd2i2} holds by \cref{qqcc0}.
    So $C_4 = \{f,x,d',d''\}$ for some distinct $d',d'' \in \{d_0,d_1,d_2\}$.
    We may assume that $d' \in \{d_1,d_2\}$.
    Observe that $\{g_2,d'\}$ is $N$-deletable for any such $d'$.
    Thus there is a $4$-element cocircuit~$C_4^*$ containing $\{g_2,d'\}$.
    This cocircuit meets both $C_2-g_2=\{c_2,x,d_0\}$ and $C_0-d'$, by orthogonality.
    Since $g_2 \notin \cocl(X)$, we have $c_2 \in C_4^*$.
    By orthogonality with $C_4$, we see that $C_4^*$ also meets $\{f,x,d''\}$.
    But then $g_2 \in \cocl(X \cup f) = \cocl(X)$; a contradiction.
    This completes the case where $x_1 = d_2$.

    Now we may assume that $x_1 \neq d_2$, and, by symmetry, $x_2 \neq d_1$.
    Hence $x_1 = x_2 = d_0$, so $C_1 = \{c_1,x,d_0,g_1\}$ and $C_2 = \{c_2,x,d_0,g_2\}$.
    If $g_1 = g_2$, then $\{c_1,c_2,x,d_0\}$ contains a circuit, by circuit elimination on $C_1$ and $C_2$, but then $r(X) \le 4$; a contradiction. So $g_1 \neq g_2$.
    Since $M \ba d_1 \ba d_2 / c_2 / x$ has an $N$-minor, and $\{d_0,g_2\}$ is a parallel pair in this matroid, we see that any pair contained in $\{g_2,d_1,d_2\}$ is $N$-deletable in $M/x$ and $M$.
    In particular, $\{g_2,d_1\}$ is contained in a $4$-element cocircuit~$C_3^*$.
    By orthogonality, this cocircuit meets $C_2-g_2 = \{c_2,x,d_0\}$ and $C_0-d_1 = \{c_1,c_2,d_2\}$.
    Since $g_2 \notin \cocl(X)$, we deduce that $c_2 \in C_3^*$.
    So let $C_3^* = \{g_2,d_1,c_2,h_1\}$, where $h_1 \in E(M)-(X \cup g_2)$.
    Also, we let $C_4^*$ be the $4$-element cocircuit containing $\{g_2,d_2\}$, and observe, similarly, that $c_2 \in C_4^*$.  So $C_4^* = \{g_2,d_2,c_2,h_2\}$ for some $h_2 \in E(M)-(X \cup g_2)$.

    Recall that $M \ba d_1 \ba g_2 / x$ has an $N$-minor; it follows that $\{x,h_1\}$ is $N$-contractible.
    So let $C_3$ be the $4$-element circuit containing $\{x,h_1\}$.
    By orthogonality, $C_3$ meets $C_1^*-x=\{c_1,d_1,d_0\}$ and $C_2^*-x=\{d_0,c_2,d_2\}$, as well as $C_3^*-h_1 = \{g_2,d_1,c_2\}$.

    Suppose that $g_2 \in C_3$.
    Then $C_3 = \{x,h_1,g_2,d_0\}$.  But $X-x$ contains a cocircuit, by cocircuit elimination of $C_1^*$ and $C_2^*$, and this cocircuit must contain $d_0$, as $C_0$ is coindependent.  So $X-x$ is a cocircuit that intersects $C_3$ in a single element; a contradiction.
    So $C_3$ meets $\{d_1,c_2\}$.

    We claim that $c_2 \notin C_3$.
    Suppose $c_2 \in C_3$.  Then $C_3$ also meets $\{c_1,d_0,d_1\}$.
    Let $C_3 = \{x,h_1,c_2,c_1\}$.  Then, as $M \ba d_0 /c_1/c_2$ has an $N$-minor, $M \ba d_0 \ba x$ has an $N$-minor.  But $\{d_0,x\} \subseteq \cocl(C_0)-C_0$, so \cref{wtd2i2} holds by \cref{qqcc0}.
    Now let $C_3 = \{x,h_1,c_2,d_0\}$.  Then, by circuit elimination with $C_2$, there is a circuit contained in $\{x,h_1,c_2,g_2\}$.  But this set intersects $C_1^*$ in a single element $x$, so $\{h_1,c_2,g_2\}$ is a triangle; a contradiction.
    So if $c_2 \in C_3$, then $C_3 = \{x,h_1,c_2,d_1\}$.

    Let $C_3 = \{x,h_1,c_2,d_1\}$.  Then, as $M \ba d_2 / x /c_2$ has an $N$-minor, where $\{d_0,g_2\}$ and $\{h_1,d_1\}$ are parallel pairs in this matroid, $M \ba d_2 \ba g_2 \ba h_1 /x$ has an $N$-minor.
    Now, $\{d_1,c_2\}$ is a series pair in this matroid, due to $C_3^*$, so $M \ba d_2 \ba g_2 /x/ d_1$ has an $N$-minor.
    As $\{c_2, h_2\}$ is a series pair in this matroid, due to $C_4^*$, the matroid $M / x /d_1/c_2$ has an $N$-minor.
    In turn, $\{d_0,g_2\}$ is a parallel pair in this matroid, due to $C_2$, so $M \ba d_0 /d_1 / c_2$ has an $N$-minor.
    Finally, $\{h_1,x\}$ is a parallel pair in this matroid, due to $C_3$, so $M \ba d_0 \ba x$ has an $N$-minor.
    But $\{d_0,x\} \subseteq \cocl(C_0)-C_0$.  So \cref{wtd2i2} holds, by \cref{qqcc0}.
    This proves the claim that $c_2 \notin C_3$.
    It now follows that $C_3 = \{x,h_1,d_1,d'\}$ for $d' \in \{d_0,d_2\}$.

    Suppose that $h_1 = h_2$. Then, by cocircuit elimination on $C_3^*$ and $C_4^*$, there is a cocircuit $\{g_2,d_1,d_2,c_2\}$.  But then $g_2 \in \cocl(X)$; a contradiction.
    So $h_1 \neq h_2$.

    Now, if $d' = d_2$, so $C_3 = \{x,h_1,d_1,d_2\}$,
    then, as $h_1 \notin \{h_2,g_2\}$, the circuit $C_3$ intersects $C_4^*$ in a single element $d_2$; a contradiction.
    So $d' = d_0$; that is, $C_3 = \{x,h_1,d_1,d_0\}$.
    Similarly, by symmetry we obtain $C_4 = \{x,h_2,d_2,d_0\}$.

    Consider the $4$-element cocircuit~$C_0^*$ containing $\{d_1,d_2\}$.
    By orthogonality with $C_3$, the cocircuit $C_0^*$ meets $\{x,h_1,d_0\}$.
    But if $C_0^*$ meets $\{x,d_0\}$, then, by orthogonality with $C_1$ and $C_2$, and since $g_1 \neq g_2$, we have $C_0^* = \{d_1,d_2,x,d_0\}$.
    Then $r^*(X)=3$, contradicting \cref{qqcc1}.
    So $h_1 \in C_0^*$.
    Similarly, using $C_4$ in place of $C_3$, we see that $C_0^* = \{d_1,d_2,h_1,h_2\}$.

    Using a similar approach taken to reveal the cocircuits $C_3^*$ and $C_4^*$, we can observe that the pairs $\{g_1,d_1\}$ and $\{g_1,d_2\}$ are $N$-deletable, so are contained in $4$-element cocircuits $C_5^*$ and $C_6^*$ respectively, each containing $c_1$.
    These cocircuits meet $C_3$ and $C_4$, respectively, but neither cocircuit meets $C_2$.  
    It follows, by orthogonality, that $C_5^*=\{g_1,d_1,c_1,h_1\}$ and $C_6^*=\{g_1,d_2,c_1,h_2\}$.

    As $M \ba d_1 \ba d_2 / c_1$ has an $N$-minor, and $\{h_1,h_2\}$ is a series pair in this matroid due to the cocircuit $C_0^*$, the pair $\{c_1,h_1\}$ is $N$-contractible.
    Consider the $4$-element circuit~$C_5$ containing $\{c_1,h_1\}$.  By orthogonality with $C_1^*$, $C_3^*$, and $C_6^*$, it meets $\{d_1,x,d_0\}$, $\{g_2,d_1,c_2\}$, and $\{g_1,d_2,h_2\}$.
    Since the last two sets are disjoint, $C_5$ does not meet $\{x,d_0\}$, so $d_1 \in C_5$.
    The final element in $C_5$ is in $\{g_1,d_2,h_2\}$.
    Now $C_5$ intersects $C_2^*$ and $C_4^*$ in at most one element.
    By orthogonality, $C_5$ and $C_2^*$ are disjoint, so $d_2 \notin C_5$; similarly, $C_5 \cap C_4^* = \emptyset$, so $h_2 \notin C_5$.
    Now $C_5 = \{c_1,h_1,d_1,g_1\} = C_5^*$ is a quad containing the $N$-deletable pair $\{d_1,g_1\}$; a contradiction.
  \end{slproof}

  \begin{sublemma}
    \label{qqcc29}
    Suppose $C_1^* = \{d_0,x,c_1,d_1\}$ and $C_2^* = \{d_0,x,c_2,d_2\}$. 
    If there is a $4$-element circuit~$C_0$ such that $\{c_1,c_2\} \subseteq C_0 \subseteq C_1^* \cup C_2^*$, then the lemma holds.
  \end{sublemma}
  \begin{slproof}
    Let $X = C_1^* \cup C_2^*$. 
    If $X$ is $3$-separating, then \cref{wtd2i1} holds by \cref{qqcc2}.
    So we may assume that $X$ is not $3$-separating.
    Let $C_0$ be a $4$-element circuit containing $\{c_1,c_2\}$ and contained in $X$.
    As $r^*(X)=4$, by \cref{qqcc1}, it now follows that $r(X) =5$.
    If $x \in C_0$, then $X-C_0$ is a pair that is $N$-deletable.
    Since $C_0$ is coindependent, and $r^*(X)=4$, \cref{wtd2i2} holds by \cref{qqcc0} in this case.
    So we may assume that $x \notin C_0$.
    Now either $C_0 = C_1^* \triangle C_2^*$ or, up to symmetry, $C_0 = \{c_1,c_2,d_0,d_2\}$.
    But in the former case, the lemma holds by \cref{qqcc25}.
    So let $C_0 = \{c_1,c_2,d_0,d_2\}$. 

    Let $C_0^*$ be the $4$-element cocircuit containing $\{d_1,d_2\}$.
    By orthogonality, $C_0^*$ meets $C_0-d_2$.  
    So let $C_0^* = \{d_1,d_2,y,f\}$, where $y \in \{d_0,c_1,c_2\}$ and $f \in E(M)-\{d_1,d_2,y\}$.
    If $C_0^* \subseteq X$, then, as $r^*(X) = 4$, we have $C_0^* = C_1^* \triangle C_2^*$.
    That is, either $\{y,f\} = \{c_1,c_2\}$ or $f \in E(M)-X$.

    Let $C_1$ and $C_2$ be the $4$-element circuits containing $\{c_1,x\}$ and $\{c_2,x\}$ respectively, and observe that $|C_1 \cap X| \ge 3$ and $|C_2 \cap X| \ge 3$, by orthogonality.
    Recall that $r(X) = 5$.
    Hence, $C_1 \nsubseteq X$ and $C_2 \nsubseteq X$, so $|C_1 \cap X| = |C_2 \cap X| = 3$.
    So let $C_1 = \{c_1,x,x_1,g_1\}$ and $C_2 = \{c_2,x,x_2,g_2\}$ where $x_1,x_2 \in X$ and $g_1,g_2 \in E(M)-X$.
    If $x_1 = c_2$, then $\{d_1,x\}$ is $N$-deletable, and $\{d_1,x\} \subseteq \cocl(C_0)-C_0$, so \cref{wtd2i2} holds by \cref{qqcc0}.
    So we may assume that $c_2 \notin C_1$ and, by symmetry, $c_1 \notin C_2$.
    By orthogonality, we now have $x_1 \in \{d_0,d_2\}$ and $x_2 \in \{d_0,d_1\}$.
    If $g_1 \in \cocl(X)$, then the coindependent circuit~$C_1$ cospans $X \cup g_1$.
    In this case, $\{d_0,d_1,d_2\} - x_1$ is a pair of $N$-deletable elements contained in $\cocl(C_1)-C_1$.  Thus \cref{wtd2i2} holds by \cref{qqcc0}.
    So we may assume that $g_1 \notin \cocl(X)$.
    Similarly, $g_2 \notin \cocl(X)$.

    Suppose $x_1 = d_0$.
    If $C_0^*$ meets $C_1=\{c_1,x,d_0,g_1\}$, then it does so in at least two elements, by orthogonality.  So in this case $C_0^* = \{d_1,d_2,y,g_1\}$ where $y \in \{d_0,c_1\}$.
    But then $g_1 \in \cocl(X)$; a contradiction.
    So we may assume that $C_0^*$ and $C_1$ are disjoint.
    Then $C_0^* = \{d_1,d_2,c_2,f\}$ for $f \in E(M)-(X \cup g_1)$.
    Note also that $f \neq g_2$, since $f \in \cocl(X)$ but $g_2 \notin \cocl(X)$.
    Now $C_2$ and $C_0^*$ both contain the element $c_2$; hence we deduce, by orthogonality, that $x_2 = d_1$.  So $C_2 = \{c_2,x,d_1,g_2\}$.
    
    By cocircuit elimination, $X-x$ contains a cocircuit.
    By orthogonality with $C_1$, this cocircuit is not $C_1^* \triangle C_2^*$.
    So the cocircuit contains $d_0$.  Since $r^*(X)=4$, by \cref{qqcc1}, it follows that $X-x$ is a cocircuit.

    Since $M \ba d_1 \ba d_2 /x$ has an $N$-minor, and $\{f,y\}$ is a series pair in this matroid, $\{f,x\}$ is $N$-contractible.
    Let $C_3$ be the circuit containing $\{f,x\}$.
    By orthogonality, $C_3$ meets $C_1^*-x$ and thus, as $X-x$ is a cocircuit, it intersects $X-x$ in two elements.
    Now $C_3$ is coindependent, so it cospans the corank-$4$ set $X \cup f$.  Thus if $X-C_3$ contains a pair of $N$-deletable elements, then \cref{wtd2i2} holds by \cref{qqcc0}.
    So $|C_3 \cap \{d_0,d_1,d_2\}| = 2$.

    As $M \ba d_0 \ba d_2 / c_2 / x$ has an $N$-minor and $\{d_1,g_2\}$ is a parallel pair in this matroid, $\{g_2,d_0\}$ is $N$-deletable.
    So $\{g_2,d_0\}$ is contained in a $4$-element cocircuit~$C_3^*$.
    By orthogonality with $C_0$, $C_1$, and $C_2$, the cocircuit $C_3^*$ meets $\{c_1,c_2,d_2\}$, $\{c_1,x,g_1\}$, and $\{c_2,x,d_1\}$.
    If $C_3^*$ meets $\{c_1,x\}$, then $C_3^* \subseteq X \cup g_2$, contradicting that $g_2 \notin \cocl(X)$.  So $g_1 \in C_3^*$.  It follows that $C_3^* = \{g_1,g_2,d_0,c_2\}$.
    By orthogonality, we now deduce that $C_3 = \{f,x,d_1,d_2\}$.

    We claim that $\{f,x\}$ is $N$-deletable.
    First, observe that as $\{c_1,c_2,x,f\}$ is contained in a series class in $M \ba d_0 \ba d_1 \ba d_2$, the matroid $M \ba d_0 \ba d_1 \ba d_2 / c_1 /c_2 /x$ has an $N$-minor.
    Due to the circuits $C_1$ and $C_2$, it follows that $M \ba g_1 \ba g_2 \ba d_2 / c_1$ has an $N$-minor.
    Now $\{d_0,c_2\}$ is a series pair in this matroid due to $C_3^*$, so $M \ba g_2 / c_1 / d_0$ has an $N$-minor.
    As $\{c_2,d_2\}$ and $\{x,g_1\}$ are parallel pairs in this matroid, due to $C_0$ and $C_1$ respectively, the matroid $M \ba g_2 \ba c_2 \ba x /c_1$ has an $N$-minor.
    As $\{d_0,d_2\}$ is a series pair in this matroid due to $C_2^*$, the matroid $M \ba g_2 \ba x /c_1 / d_2$ has an $N$-minor.
    Now $\{d_0,c_2\}$ is a parallel pair in this matroid due to $C_0$, so $M \ba g_2 \ba x \ba d_0 / d_2$ has an $N$-minor.
    As $\{c_1,d_1\}$ and $\{g_1,c_2\}$ are series pairs in this matroid, due to $C_1^*$ and $C_3^*$ respectively, the matroid $M / c_2 / d_1 / d_2$ has an $N$-minor.
    Finally, $\{x,f\}$ and $\{x,g_2\}$ are parallel pairs in this matroid due to $C_3$ and $C_2$, so $M \ba f \ba x$ has an $N$-minor, as required.

    Since $C_0$ cospans $X \cup f$, the $N$-deletable pair $\{f,x\}$ is in $\cocl(C_0)-C_0$, implying that \cref{wtd2i2} holds by \cref{qqcc0}.

    \smallbreak

    Now we may assume that $x_1 = d_2$ and, by symmetry, $x_2 = d_1$, so that $C_1 = \{c_1,x,d_2,g_1\}$ and $C_2 = \{c_2,x,d_1,g_2\}$.
    If $g_1=g_2$, then $\{c_1,c_2,x,d_1,d_2\}$ contains a circuit, by circuit elimination, so $r(X-d_0) \le 4$.  But $d_0 \in \cl(X-d_0)$, due to the circuit $C_0$, so this is contradictory.  So $g_1 \neq g_2$.
    Now $C_0^*$ meets $C_1$ and $C_2$, but $x \notin C_0^*$ and $C_0^*$ contains at most one of $g_1$ and $g_2$.  It follows, by orthogonality, that $C_0^*$ meets $\{c_1,c_2\}$; by symmetry, we may assume that $c_1 \in C_0^*$.
    Then $C_0^* = \{d_1,d_2,c_1,f\}$ for some $f \in \{c_2,x,g_2\}$.
    But $f \neq x$, otherwise $r^*(X) = 3$.
    If $f = g_2$, then $g_2 \in \cocl(X)$; a contradiction.
    So $f = c_2$ and $C_0^* = C_1^* \triangle C_2^*$.

    As $M \ba d_0 \ba d_2/c_2/x$ has an $N$-minor, $\{d_0,g_2,d_2\}$ is $N$-deletable.
    Let $C_3^*$ be the $4$-element cocircuit containing $\{g_2,d_2\}$.
    Recall that $g_2 \notin \cocl(X)$, so $C_3^* \nsubseteq X \cup g_2$.
    As $C_3^*$ meets $C_0-d_2=\{c_1,c_2,d_0\}$ and $C_2-g_2=\{c_2,x,d_1\}$, by orthogonality, it follows that $c_2 \in C_3^*$.
    But $C_3^*$ also meets $C_1-d_2=\{c_1,x,g_1\}$, so $g_1 \in C_3^*$.
    That is, $C_3^* = \{g_2,d_2,c_2,g_1\}$.

    Now $M \ba d_0 \ba d_2 \ba g_2$ has an $N$-minor, and $\{x,c_2\}$ and $\{c_2,g_1\}$ are series pairs in this matroid. 
    So $M \ba g_2 / x / g_1$ has an $N$-minor, where $\{c_1,d_2\}$ is a parallel pair in this matroid due to the circuit $C_1$, implying $\{g_2,c_1\}$ is $N$-deletable.
    Consider the cocircuit $C_4^*$ containing $\{g_2,c_1\}$.
    By orthogonality, $C_4^*$ meets $C_0-c_1=\{d_0,d_2,c_2\}$ and $C_2-g_2=\{d_1,x,c_2\}$.  As $g_2 \notin \cocl(X)$, we have $c_2 \in C_4^*$.
    By orthogonality with $C_1$, and again using that $g_2 \notin \cocl(X)$, we have $C_4^* = \{g_1,g_2,c_1,c_2\}$.
    Now $r^*(C_3^* \cup C_4^*) = 3$, and $d_1 \in \cocl(C_3^* \cup C_4^*)$ due to the cocircuit $C_0^*$.
    It follows that $r^*(X \cup \{g_1,g_2\}) =r^*(X)= 4$, contradicting that $g_1,g_2 \notin \cocl(X)$.
  \end{slproof}

  Next we handle one particularly awkward case that arises.
  \begin{sublemma}
    \label{awkwardmissedcase}
    Suppose $C_1^* = \{d_0,x,c_1,d_1\}$ and $C_2^* = \{d_0,x,c_2,d_2\}$ and 
let $X = C_1^* \cup C_2^*$. 
    If $C_0=\{c_1,c_2,d_0,e\}$ is a circuit, for some $e \in E(M)-X$, and 
    $C_3^*=\{d_1,e,c_2,f_1\}$ and $C_4^*=\{d_2,e,c_1,f_2\}$ are cocircuits, for $f_1,f_2 \in E(M)-(X \cup e)$, then the lemma holds.
  \end{sublemma}
  \begin{slproof}
%
    Since $M \ba d_0 \ba d_1 \ba d_2 / c_1 /c_2$ has an $N$-minor, it follows that $M \ba d_1 \ba d_2 \ba e/ c_1 /f_1$ and $M \ba d_1 \ba d_2 \ba e/ c_2 /f_2$ have $N$-minors.
    Let $C_3$ be the $4$-element circuit containing $\{c_1,f_1\}$, and let $C_4$ be the $4$-element circuit containing $\{c_2,f_2\}$.

    Suppose $f_1 \neq f_2$.
    By orthogonality with $C_3^*$ and $C_4^*$, the circuit~$C_3$ meets $\{e,c_2,d_1\}$ and $\{e,d_2,f_2\}$.
    If $e \in C_3$, then $C_2^* \cap C_3 = \emptyset$, so, by orthogonality with $C_1^*$, we see that $C_3 = \{c_1,f_1,e,d_1\}$.
    On the other hand, if $e \notin C_3$ and $C_3$ meets $C_2^*$, then $C_3$ contains both $c_2$ and $d_2$, which then contradicts orthogonality with $C_1^*$.
    So when $e \notin C_3$, the sets $C_2^*$ and $C_3$ are disjoint, in which case $\{d_1,f_2\} \subseteq C_3$.
    Now 
    $C_3= \{c_1,f_1,d_1,p_1\}$ for some $p_1 \in \{e,f_2\}$, and, by symmetry, $C_4 = \{c_2,f_2,d_2,p_2\}$ for some $p_2 \in \{e,f_1\}$.

    Let $Y = X \cup \{e,f_1,f_2\}$.
    Next, we claim that $x \notin \cl(Y-x)$.
    Observe that $e \in \cl(X-x)$, due to the circuit~$C_0$.
    If $e \in \{p_1,p_2\}$, then $r(Y-x)=r(X-x) \le 5$, due to the circuits $C_3$ and $C_4$.
    Otherwise, $C_3 = \{c_1,f_1,d_1,f_2\}$ and $C_4= \{c_2,f_2,d_2,f_1\}$, in which case $\{f_1,f_2,c_1,c_2,d_0\}$ spans $Y-x$.
    So $r(Y-x) \le 5$.
    Since $r^*(Y-x) \le 5$, we have $\lambda(Y-x) \le 5+5-8=2$.
    Now, if $x \in \cl(Y-x)$, then $\lambda(Y) \le 1$, so $|E(M)| \le 10$; a contradiction.

    Consider the $4$-element circuit~$C_1$ containing $\{c_1,x\}$.
    By orthogonality with $C_2^*$, this circuit meets $\{d_0,d_2,c_2\}$.
    If $d_0 \in C_1$, then, by orthogonality with $C_4^*$, the final element is in $\{e,d_2,f_2\}$, in which case $x \in \cl(Y-x)$; a contradiction.
    If $c_2 \in C_1$, then, by orthogonality with $C_3^*$ and $C_4^*$, the final element is $e$; again, we obtain the contradiction that $x \in \cl(Y-x)$.
    So $d_2 \in C_1$.
    Let $C_1 = \{c_1,x,d_2,g_1\}$, for some $g_1 \in E(M)-\{c_1,x,d_2\}$.  
    Note that $g_1 \notin Y$, since $x \notin \cl(Y-x)$.
    By symmetry, there is a circuit $C_2 = \{c_2,x,d_1,g_2\}$ where $g_2 \in E(M)-Y$.
    If $g_1 = g_2$, then $\{c_1,c_2,x,d_1,d_2\}$ contains a circuit, which cannot contain $x$ since $x \notin \cl(Y-x)$.  But then $\{c_1,c_2,d_1,d_2\}$ is a circuit, so the lemma holds by \cref{qqcc25}. 
    So we may assume that $g_1 \neq g_2$.

    As $M \ba d_0 / c_1 /x$ has an $N$-minor, the pair $\{d_0,g_1\}$ is $N$-deletable.
    Thus $\{d_0,g_1\}$ is contained in a $4$-element cocircuit~$C_5^*$.
    By orthogonality with $C_1$ and $C_0$, the cocircuit $C_5^*$ meets $\{c_1,x,d_2\}$ and $\{c_1,c_2,e\}$.

    Consider when $c_1 \notin C_5^*$.
    Then, by orthogonality with $C_2$, either $C_5^* = \{d_0,g_1,d_2,e\}$ or $C_5^* = \{d_0,g_1,x,c_2\}$.
    Suppose $C_5^* = \{d_0,g_1,d_2,e\}$.
    Then $M \ba d_0 \ba d_1 \ba d_2 / c_1 / c_2 / g_1$ has an $N$-minor, so there is a $4$-element circuit $C_5$ containing $\{c_2,g_1\}$.
    If $x \notin C_5$, then $C_5=\{c_2,g_1,d_0,d_1\}$ or $C_5 = \{c_2,g_1,d_2,e\}$ by orthogonality with $C_3^*$ and $C_4^*$, in which case $g_1 \in \cl(Y-x)$.
    As $\lambda(Y-x) = 2$, we have $\lambda((Y-x) \cup g_1) \le 1$ and $|E(M)| \le 10$; a contradiction.
    So $x \in C_5$, in which case $C_5 = \{c_2,g_1,x,d_1\}$, and $C_2 \cup C_5$ is a $5$-element plane that intersects $C_1^*$ in two elements; a contradiction.
    Now suppose $C_5^* = \{d_0,g_1,x,c_2\}$.
    Then $C_2^* \cup C_5^*$ is a $5$-element coplane that intersects the circuit $C_0$ in two elements; a contradiction.

    So now we may assume $\{d_0,g_1,c_1\} \subseteq C_5^*$.
    By orthogonality with $C_3$, the final element of $C_5^*$ is in $\{f_1,d_1,p_1\}$. 
    If $d_1 \in C_5^*$, then $C_1^* \cup C_5^*$ is a $5$-element coplane that intersects the circuit $C_0$ in two elements; a contradiction.
    By orthogonality with $C_3$, either $C_5^* = \{d_0,g_1,c_1,f_1\}$, or $C_5^* = \{d_0,g_1,c_1,e\}$ and $p_1 = e$.
    Consider the former case, where $C_5^* = \{d_0,g_1,c_1,f_1\}$.
    As $M \ba d_0 / c_1 / x$ has an $N$-minor, $M \ba d_0 \ba g_1 / x$ has an $N$-minor, and thus $\{f_1,x\}$ is $N$-contractible.
    Now there is a $4$-element circuit~$C_6$ containing $\{f_1,x\}$,
    which meets $\{d_0,d_1,c_1\}$ and $\{d_0,d_2,c_2\}$, by orthogonality.
    If $d_0 \notin C_6$, then $C_6 = \{f_1,x,c_1,d_2\}$ or $C_6 = \{f_1,x,c_2,d_1\}$, by orthogonality with $C_3^*$.
    But then $C_1 \cup f_1$ or $C_2 \cup f_1$, respectively, is a $5$-element plane that intersects $C_1^*$ in two elements; a contradiction.
    So $d_0 \in C_6$.  By orthogonality with $C_3^*$ and $C_4^*$, we have $C_6=\{f_1,x,d_0,d_1\}$ or $C_6=\{f_1,x,d_0,c_2\}$.  In either case, $x \in \cl(Y-x)$; a contradiction. 
    So $C_5^* = \{d_0,g_1,c_1,e\}$ and $C_3 = \{c_1,f_1,d_1,e\}$.
    By circuit elimination on $C_3$ and $C_0$, there is a circuit contained in $\{c_1,f_1,d_1,d_0,c_2\}$.  By orthogonality with $C_4^*$, this circuit is $\{f_1,d_1,d_0,c_2\}$, but then it intersects $C_5^*$ in a single element; a contradiction.

    \medskip
    Now suppose $f_1 = f_2=f$.
    Let $Y = X \cup \{e,f\}$.
    By orthogonality, $C_3$ meets $\{d_1,d_0,x\}$ and $\{d_1,e,c_2\}$.
    If $d_1 \notin C_3$, then by orthogonality with $C_2^*$, we have $c_2 \in C_3$.
    But if $C_3 = \{c_1,f,d_0,c_2\}$, then $C_0 \cup C_3$ is a $5$-element plane that intersects $C_1^*$ in two elements, contradicting orthogonality.
    So $C_3 = \{c_1,f,x,c_2\}$ when $d_1 \notin C_3$.
    On the other hand, if $d_1 \in C_3$, then by orthogonality with $C_2^*$, either $C_3=\{c_1,f,d_1,e\}$, or $C_3=\{c_1,f,d_1,g_3\}$ for some $g_3 \in E(M)-Y$.

    We claim that if $\{c_1,f,x,c_2\}$ is a circuit, then the lemma holds.
    Suppose $C_3 = \{c_1,f,x,c_2\}$ is a circuit.
    As $M \ba d_0 \ba d_1 \ba d_2 / c_1 /c_2$ has an $N$-minor and $\{f,x\}$ is a parallel pair in this matroid, $M \ba d_0 \ba d_1 \ba d_2 \ba x$ has an $N$-minor.
    This matroid has $c_1$ and $c_2$ as coloops, so $M \ba X$ has an $N$-minor.
    Now, since $M \ba (X-\{d_1,d_2\})$ has an $N$-minor, we see that $\{d_1,d_2\}$ is $N$-contractible, so this pair is contained in a $4$-element circuit $C_4$.
    If $C_4=\{d_1,d_2,c_1,c_2\}$, then the lemma holds by \cref{qqcc25}.
    So we may assume, by orthogonality with $C_1^*$ and $C_3^*$, that $C_4$ meets $\{x,d_0\}$ and $\{e,f\}$.
    It now follows that $d_1 \in \cl(Y-d_1)$, so $r(Y) \le 5$.
    Moreover, as $\{e,d_0\}$ and $\{f,x\}$ are parallel pairs in the matroid $M \ba d_1 \ba d_2 / c_1 /c_2$, the matroid $M \ba d_1 \ba d_2 \ba e \ba f$ has an $N$-minor.
    If $\{d_1,d_2,e,f\}$ is a cocircuit, then it follows that $r^*(Y) \le 4$; a contradiction.
    So $\{d_1,d_2,e,f\}$ is coindependent, and hence $r^*_{M \ba d_1 \ba d_2 \ba e \ba f}(\{c_1,c_2,x,d_0\}) = 1$, implying the pair $\{x,d_0\}$ is $N$-contractible.
    Now $r(Y-\{x,d_0\}) \le 5$ and $r^*(Y-\{x,d_0\})=4$, so $\lambda_{M/x}(Y-\{x,d_0\})=2$.
    Since $d_0 \in \cl_{M/x}(Y-\{x,d_0\})$, \cref{wtd2i2} holds by the dual of \cref{3sepwin2}.
    Henceforth, we may assume that $\{c_1,f,x,c_2\}$ is not a circuit.

    Now $C_3=\{c_1,f,d_1,e\}$ or $C_3=\{c_1,f,d_1,g_3\}$ for some $g_3 \in E(M)-Y$.
    By symmetry, $C_4=\{c_2,f,d_2,e\}$ or $C_4=\{c_2,f,d_2,g_4\}$ for some $g_4 \in E(M)-Y$.
    If $C_3=\{c_1,f,d_1,e\}$ and $C_4=\{c_2,f,d_2,e\}$, then $x \notin \cl(Y-x)$, and $\{c_1,x,d_2\} \subseteq C_1$.  In this case, by using a similar argument as in the case that $f_1 \neq f_2$, it follows that $|E(M)| \le 10$; we omit the details. 
    So we may assume, without loss of generality, that $C_3=\{c_1,f,d_1,g_3\}$ for some $g_3 \in E(M)-Y$.

    Recall also the $4$-element circuit $C_1$ containing $\{c_1,x\}$.
    We claim that either $C_1 = \{c_1,x,d_0,d_2\}$ or $C_1 = \{c_1,x,d_2,g_1\}$ for some $g_1 \in E(M)-X$.
    If $c_2 \in C_1$, then, by orthogonality with $C_3^*$ and $C_4^*$, either $C_1 = \{c_1,x,c_2,e\}$ or $C_1 = \{c_1,x,c_2,f\}$.
    We have already argued that the latter is not a circuit.
    In the former case, $C_0 \cup x$ is a $5$-element plane that intersects $C_3^*$ in two elements; a contradiction.
    By orthogonality with $C_2^*$, it now follows that either $d_0 \in C_1$ or $d_2 \in C_1$.
    If $d_0 \in C_1$, then $C_1 = \{c_1,x,d_0,d_2\}$, by orthogonality with $C_3^*$ and $C_4^*$.
    Otherwise, $d_2 \in C_1$, and $C_1$ intersects $C_4^*$ in at most one element.
    Thus, by orthogonality, either $C_1 = \{c_1,x,d_0,d_2\}$ or $C_1 = \{c_1,x,d_2,g_1\}$ for some $g_1 \in E(M)-X$.
    Similarly, either $C_2 = \{c_2,x,d_0,d_1\}$ or $C_2 = \{c_2,x,d_1,g_2\}$ for $g_2 \in E(M)-X$.

    Suppose $C_1 = \{c_1,x,d_0,d_2\}$ and $C_2 = \{c_2,x,d_0,d_1\}$.
    Then $r(X) \le 4$, and $r^*(X) \le 4$, so $X$ is $3$-separating, and \cref{wtd2i1} holds by \cref{qqcc2}.
    So we may assume that, for some $\{i,j\} = \{1,2\}$, we have $C_i=\{c_i,x,d_j,g_i\}$.

    Suppose that $C_4=\{c_2,f,d_2,e\}$.
    Recall that $M \ba d_1 \ba d_2 \ba e / c_1 / c_2$ has an $N$-minor.
    Due to the cocircuit $C_3^*$, it follows that $M \ba d_2 \ba e / c_1 / f$ has an $N$-minor.
    As $\{d_1,g_3\}$ is a parallel pair in this matroid, $M \ba g_3 \ba e$ also has an $N$-minor, so there is a $4$-element cocircuit $C_5^*$ containing $\{g_3,e\}$.
    As $C_5^*$ meets the circuits $C_0$, $C_3$, and $C_4$, the cocircuit also contains an element from each of $\{d_0,c_1,c_2\}$, $\{d_1,c_1,f\}$, and $\{d_2,c_2,f\}$, by orthogonality.
    It follows that $C_5^* \subseteq Y \cup g_3$, so
    $g_3 \in \cocl(Y)$.
    If $r(Y) \le 5$, then,
    as $r^*(Y) \le 5$ and $g_3 \in \cl(Y)$, due to the circuit $C_3$, we have $\lambda(Y \cup g_3) \le 1$, implying $|E(M)| \le 10$; a contradiction.
    So $r(Y) \ge 6$.
    In particular, 
    if, for some $\{i,j\}=\{1,2\}$, we have $C_j=\{c_j,x,d_0,d_i\}$, then we obtain the contradiction that $r(Y) \le 5$.

    So $C_1 = \{c_1,x,d_2,g_1\}$ and $C_2=\{c_2,x,d_1,g_2\}$.
    Moreover, if $g_3 \in \{g_1,g_2\}$, then, after circuit elimination on $C_3$ and either $C_1$ or $C_2$, it follows that $r(Y) \le 5$; a contradiction.
    By a similar argument, $g_1 \neq g_2$.
    So we may assume that $g_1$, $g_2$, and $g_3$ are distinct.
%
    Since $M \ba d_2 \ba e / c_1 / f$ has an $N$-minor, and $\{g_3,d_1\}$ is a parallel pair in this matroid, $\{g_3,d_2\}$ is $N$-deletable.
    So there is a $4$-element cocircuit~$C_6^*$ containing $\{g_3,d_2\}$.
    By orthogonality with $C_1$ and $C_4$, the cocircuit contains an element in $\{g_1,x,c_1\}$ and an element in $\{c_2,e,f\}$.
    If $C_6^*$ meets $\{c_2,e\}$, then, by orthogonality with $C_0$, the other element is $c_1$.  Then, by orthogonality with $C_2$, we have $C_6^*=\{g_3,d_2,c_1,e\}$, in which case $C_3^* \cup g_3$ is a $5$-element coplane intersecting $C_0$ in two elements; a contradiction.
    So $f \in C_6^*$.
    Then, the final element of $C_6^*$ is in $\{g_1,x,c_1\}$, and by orthogonality $C_6^*$ avoids $C_0$ and $C_2$; so $C_6^*=\{g_3,d_2,f,g_1\}$.

    Now $M \ba d_2 \ba e / c_1 / f$ has an $N$-minor, and $\{d_1,g_3\}$ is a parallel pair in this matroid, due to $C_3$, so $M \ba g_3 \ba d_2 \ba e / c_1$ has an $N$-minor.
    As $\{f,g_1\}$ is a series pair in this matroid, due to $C_6^*$, the matroid $M \ba e / c_1 / g_1$ has an $N$-minor.
    As $\{x,d_2\}$ is a parallel pair in this matroid, due to $C_1$, the pair $\{x,e\}$ is $N$-deletable.
    So there is a $4$-element cocircuit~$C_7^*$ containing $\{x,e\}$, and meeting $\{c_1,d_2,g_1\}$, $\{c_2,d_1,g_2\}$ by orthogonality with $C_1$ and $C_2$.
    By orthogonality with $C_3$ and $C_0$, this cocircuit is either $\{x,e,c_1,d_1\}$, $\{x,e,d_2,c_2\}$, or $\{x,e,g_1,c_2\}$.
    If $C_7^*=\{x,e,c_1,d_1\}$ or $C_7^*=\{x,e,d_2,c_2\}$, then $C_1^* \cup e$ or $C_2^* \cup e$ is a $5$-element coplane intersecting $C_1$ or $C_2$ in two elements, respectively; a contradiction.
    So $C_7^*=\{x,e,c_2,g_1\}$.
    Then, by cocircuit elimination with $C_2^*$, there is a cocircuit~$C_8^*$ contained in $\{e,g_1,x,d_0,d_2\}$.
    By orthogonality with $C_2$, we have $C_8^*=\{d_0,d_2,e,g_1\}$.
    Now, both $\{x,c_2\}$ and $\{e,g_1\}$ are series pairs in $M \ba d_0 \ba d_2$, so $M / x / e$ has an $N$-minor.
    Consider the $4$-element circuit containing $\{e,x\}$.
    By orthogonality with $C_1^*$, $C_2^*$, $C_3^*$, and $C_4^*$, the circuit is either $\{e,x,f,d_0\}$, $\{e,x,d_1,d_2\}$, or $\{e,x,c_1,c_2\}$.
    In the first two cases, the circuit intersects $C_6^*$ in a single element, while in the latter case the circuit intersects $C_8^*$ in a single element; a contradiction.

    So in what follows, we may assume that $C_3 = \{c_1,f,d_1,g_3\}$ and $C_4 = \{c_2,f,d_2,g_4\}$, for $g_3,g_4 \in E(M)-Y$.
    By symmetry, we may also assume $C_1=\{c_1,x,d_2,g_1\}$.
    Suppose $C_2= \{c_2,x,d_0,d_1\}$.
    Since $M \ba d_1 \ba e / c_2 / f$ has an $N$-minor, and $\{d_2,g_4\}$ is a parallel pair in this matroid, $\{g_4,d_1\}$ is $N$-deletable.
    %
    Consider the $4$-element cocircuit containing $\{g_4,d_1\}$.
    By orthogonality with $C_2$ and $C_4$, the cocircuit meets $\{c_2,x,d_0\}$ and $\{c_2,d_2,f\}$.
    If the cocircuit contains $c_2$, then the final element is $c_1$, by orthogonality with $C_3$ and $C_0$.
    But then the cocircuit intersects $C_1$ in a single element, contradicting orthogonality.
    So the cocircuit contains an element in $\{x,d_0\}$ and an element in $\{d_2,f\}$.
    By orthogonality with $C_1$, the cocircuit is either $\{g_4,d_1,x,d_2\}$ or $\{g_4,d_1,d_0,f\}$.
    But the former intersects $C_3$ in a single element, and the latter intersects $C_0$ in a single element.
    From this contradiction, we deduce that $C_1 = \{c_1,x,d_2,g_1\}$ and $C_2 = \{c_2,x,d_1,g_2\}$ for $g_1,g_2 \in E(M)-Y$.

    \medskip
    Now suppose $C_1 = \{c_1,x,d_2,g_1\}$, $C_2 = \{c_2,x,d_1,g_2\}$, $C_3 = \{c_1,f,d_1,g_3\}$, and $C_4 = \{c_2,f,d_2,g_4\}$, for $g_1,g_2,g_3,g_4 \in E(M)-Y$.
    Since $M \ba d_0 / c_1 /x$ has an $N$-minor and $\{g_1,d_2\}$ is a parallel pair in this matroid, $\{g_1,d_0\}$ is $N$-deletable.
    Consider the $4$-element cocircuit $C_5^*$ containing $\{g_1,d_0\}$.
    This cocircuit meets $\{c_1,x,d_2\}$ and $\{c_1,c_2,e\}$ by orthogonality with $C_1$ and $C_0$.
    Also, $|C_5^* \cap C_1^*| \neq 3$, for otherwise $C_1^* \cup g_1$ is a $5$-element coplane that intersects $C_0$ or $C_1$ in two elements; a contradiction.
    Similarly, $|C_5^* \cap C_2^*| \neq 3$.

    Suppose $g=g_1=g_2$.  Then $C_5^*$ also meets $\{c_2,x,d_1\}$, and it follows that $C_5^* = \{g,d_0,c_1,c_2\}$ or $C_5^* = \{g,d_0,x,e\}$.
    In the former case, it follows that $g=g_3=g_4$, so 
    $C_3 = \{c_1,f,d_1,g\}$ and $C_4 = \{c_2,f,d_2,g\}$.
    Now $r^*(Y \cup g) \le 5$ and $r(Y \cup g) \le 5$, so $\lambda(Y \cup g) \le 1$ and $|E(M)| \le 10$; a contradiction.
    So $C_5^* = \{g,d_0,x,e\}$.
    By circuit elimination on $C_1$ and $C_2$, the set $\{c_1,d_1,x,c_2,d_2\}$ contains a circuit.  By orthogonality with $C_5^*$, we see that $C_1^* \triangle C_2^* = \{c_1,d_1,c_2,d_2\}$ is a circuit, so the lemma holds by \cref{qqcc25}.
    So we may assume that $g_1 \neq g_2$.

    Now suppose $g'=g_3=g_4$.
    Since $M \ba d_1 \ba e / c_1$ has an $N$-minor, so does $M \ba e / c_1 / f$, and it follows that $\{e,g'\}$ is $N$-deletable.  Thus there is a $4$-element cocircuit $C^*$ containing $\{e,g'\}$.
    If $C^* = \{e,g',c_1,c_2\}$, then, by orthogonality with $C_1$ and $C_2$, we have $g' = g_1 = g_2$; a contradiction.
    Now, if $c_1 \in C^*$, then by orthogonality with $C_3$, the final element is either $d_2$ or $f$; but then $C_4^* \cup g'$ is a $5$-element coplane intersecting $C_0$ in two elements; a contradiction.
    So $c_1 \notin C^*$ and, similarly, $c_2 \notin C^*$; hence $d_0 \in C^*$.
    By orthogonality with $C_3$ and $C_4$, we have $C^* = \{d_0,e,f,g'\}$.
    Now by circuit elimination on $C_3$ and $C_4$, there is a circuit contained in $\{c_1,d_1,f,c_2,d_2\}$.
    By orthogonality with $C^*$, we see that $C_1^* \triangle C_2^* = \{c_1,d_1,c_2,d_2\}$ is a circuit, so the lemma holds by \cref{qqcc25}.
    So we may assume that $g_3 \neq g_4$.

    We return to the $4$-element cocircuit $C_5^*$ containing $\{g_1,d_0\}$ and meeting $\{c_1,x,d_2\}$ and $\{c_1,c_2,e\}$.
    Consider when $c_1 \notin C_5^*$.
    Then, by orthogonality with $C_2$, either $C_5^* = \{d_0,g_1,d_2,e\}$ or $C_5^* = \{d_0,g_1,x,c_2\}$.
    But if $C_5^* = \{d_0,g_1,x,c_2\}$, then $C_2^* \cup g_1$ is a $5$-element coplane that intersects $C_0$ in two elements; a contradiction.
    So $C_5^* = \{d_0,g_1,d_2,e\}$. 
    Similarly, there is a $4$-element cocircuit $C_6^*$ containing $\{d_0,g_2\}$ where either $c_2 \in C_6^*$ or $C_6^* = \{d_0,g_2,d_1,e\}$.

    Suppose $C_5^* = \{d_0,g_1,d_2,e\}$ and $C_6^* = \{d_0,g_2,d_1,e\}$.
    Then, by cocircuit elimination, there is a cocircuit contained in $\{g_1,g_2,d_1,d_2,e\}$.
    By orthogonality with $C_0$, this cocircuit is $\{g_1,g_2,d_1,d_2\}$.
    Now by orthogonality we have $g_2=g_3$ and $g_1=g_4$, so $C_3 = \{c_1,f,d_1,g_2\}$ and $C_4 = \{c_2,f,d_2,g_1\}$.
%
    As $M \ba d_1 \ba d_2 \ba e / c_1 / c_2$ has an $N$-minor, and $\{g_1,g_2\}$ is a series pair in this matroid, $M \ba e / c_2 / g_1$ has an $N$-minor.
    As $\{d_2,f\}$ is a parallel pair in this matroid, $M \ba e \ba f / g_1$ also has an $N$-minor.
    Then $\{d_1,c_2\}$ is a series pair in the latter matroid, so $\{g_1,d_1\}$ is $N$-contractible.
    So there is a $4$-element circuit containing $\{g_1,d_1\}$.
    By orthogonality with $C_5^*$ and $C_1^*$, this circuit meets $\{d_0,d_2,e\}$ and $\{d_0,x,c_1\}$.
    If the circuit contains $d_0$, then, by orthogonality with $C_2^*$ and $C_4^*$, it is $\{g_1,d_1,d_0,c_2\}$.
    Otherwise, when the circuit does not contain $d_0$, then, by orthogonality with $C_4^*$ and $C_3^*$, it is $\{g_1,d_1,e,c_1\}$.
    In either case, $g_1 \in \cl(C_0 \cup d_1)$, and it follows that $r(Y \cup \{g_1,g_2\}) \le 5$.  As $r^*(Y \cup \{g_1,g_2\}) \le 5$, we have $\lambda(Y \cup \{g_1,g_2\}) = 0$, implying $|E(M)| = 10$; a contradiction.
    We deduce that either $c_1 \in C_5^*$ or $c_2 \in C_6^*$.

    By symmetry, we may now assume that $\{d_0,g_1,c_1\} \subseteq C_5^*$.
    First we assume that $g_1 \neq g_3$.
    By orthogonality with $C_3$, the final element of $C_5^*$ is in $\{f,d_1,g_3\}$. 
    If $d_1 \in C_5^*$, then $C_1^* \cup C_5^*$ is a $5$-element coplane that intersects the circuit $C_0$ in two elements; a contradiction.
    So either $C_5^* = \{d_0,g_1,c_1,f\}$, or $C_5^* = \{d_0,g_1,c_1,g_3\}$.

    Suppose $C_5^* = \{d_0,g_1,c_1,f\}$.
    As $M \ba d_0 / c_1 / x$ has an $N$-minor, $M \ba d_0 \ba g_1 / x$ has an $N$-minor, and thus $\{f,x\}$ is $N$-contractible.
    So there is a $4$-element circuit~$C_6$ containing $\{f,x\}$,
    which meets $\{d_0,d_1,c_1\}$ and $\{d_0,d_2,c_2\}$, by orthogonality.
    If $d_0 \notin C_6$, then $C_6 = \{f,x,c_1,c_2\}$ or $C_6 = \{f,x,d_1,d_2\}$, by orthogonality with $C_3^*$ and $C_4^*$.
    By orthogonality with $C_5^*$, we have $C_6 = \{f,x,c_1,c_2\}$.
    But $\{c_1,f,x,c_2\}$ is not a circuit, by an earlier claim.
    So $d_0 \in C_6$.  By orthogonality with $C_3^*$ and $C_4^*$, we now have $C_6 = \{f,x,d_0,e\}$.
%
    By cocircuit elimination on $C_3^*$ and $C_5^*$, there is a cocircuit contained in $\{c_1,d_2,e,d_0,g_1\}$.
    By orthogonality with $C_3$, the cocircuit is $C_7^* = \{d_2,e,d_0,g_1\}$.
    Now, by circuit elimination on $C_0$ and $C_6$, there is a circuit contained in $\{e,c_1,c_2,f,x\}$; by orthogonality with $C_7^*$, this circuit is $\{c_1,c_2,f,x\}$.  But, as argued earlier, this is not a circuit.

    Now suppose $C_5^* = \{d_0,g_1,c_1,g_3\}$.
    Since $M \ba d_1 /c_1 /c_2$ has an $N$-minor, it follows that
    $M \ba e \ba d_1 /c_1$ has an $N$-minor, so
    $M \ba e /c_1 /f$ has an $N$-minor, and finally 
    $\{e,g_3\}$ is $N$-deletable.
    So there is a $4$-element cocircuit~$C_7^*$ containing $\{e,g_3\}$.
    By orthogonality with $C_3$ and $C_0$, this cocircuit meets $\{d_1,c_1,f\}$ and $\{c_1,c_2,d_0\}$.
    If it contains $c_1$, so that $\{e,g_3,c_1\} \subseteq C_7^*$, then by orthogonality with $C_1$, the final element is in $\{x,d_2,g_1\}$.
    By orthogonality with $C_2$ and $C_4$, it follows that $C_7^* = \{e,g_3,c_1,g_1\}$.  But then $C_5^* \cup C_7^*$ is a $5$-element coplane that intersects $C_3$ in two elements; a contradiction.
    So $c_1 \notin C_7^*$, and thus $C_7^*$ meets $\{c_2,d_0\}$ and $\{d_1,f\}$, by orthogonality with $C_0$ and $C_3$.
    But this contradicts orthogonality with $C_1$, or with $C_4$.

    \medskip
    Next we assume that $g_1=g_3$.
    Suppose also that $g_2=g_4$.
    Recall that $\{d_0,g_1,c_1\} \subseteq C_5^*$.
    By orthogonality, $C_5^*$ does not meet $C_2$ or $C_4$, so either $C_5^* = \{d_0,g_1,c_1,e\}$ or $C_5^* = \{d_0,g_1,c_1,h_1\}$ for some $h_1 \in E(M)-(Y \cup \{g_1,g_2\})$.
    Suppose $C_5^* = \{d_0,g_1,c_1,e\}$ and observe that $r^*(Y \cup g_1) \le 5$.
    We claim that $M / d_0 /f$ has an $N$-minor.
    As $M \ba d_1 \ba d_2 \ba e / c_1$ has an $N$-minor and $\{f,c_2\}$ is a series pair in this matroid, $M \ba d_2 \ba e / c_1 / f$ has an $N$-minor.
    Due to $C_3$, the latter matroid has the parallel pair $\{d_1,g_1\}$, so $M \ba g_1 \ba d_2 \ba e / f$ has an $N$-minor.
    Then $\{d_0,c_1\}$ is a series pair in this matroid, due to $C_5^*$, so $M \ba g_1 \ba d_2 / d_0 / f$ has an $N$-minor, thus showing the claim.
    Now $\{d_0,f\}$ is contained in a $4$-element circuit.
    By orthogonality with $C_1^*$, $C_2^*$, $C_3^*$, $C_4^*$, and $C_5^*$, this circuit is either $\{d_0,f,c_1,c_2\}$ or $\{d_0,f,x,e\}$.
    But in the former case, $C_0 \cup f$ is a $5$-element plane that intersects $C_1^*$ in two elements; a contradiction.
    So $\{d_0,f,x,e\}$ is a circuit.
    Hence $r(Y \cup \{g_1,g_2\}) \le 5$.
    We now claim that $M \ba g_1 \ba g_2$ has an $N$-minor.
    As $\{x,e\}$ is a parallel pair in the matroid $M \ba g_1 \ba d_2 / d_0 / f$, which has an $N$-minor, the matroid
    $M \ba x \ba g_1 \ba d_2 / f$ also has an $N$-minor, and it follows that
    $M \ba g_1 / c_2 / f$ has an $N$-minor, and finally
    $M \ba g_1 \ba g_2$ also has an $N$-minor as claimed.
    Now $\{g_1,g_2\}$ is contained in a $4$-element cocircuit $C_8^*$.
    By orthogonality, $C_8^* \subseteq Y \cup \{g_1,g_2\}$, so $r^*(Y \cup \{g_1,g_2\}) \le 5$.
    Thus $\lambda(Y \cup \{g_1,g_2\}) =0$, in which case $|E(M)| = 10$; a contradiction. 

    Now suppose $C_5^* = \{d_0,g_1,c_1,h_1\}$ for $h_1 \in E(M)-(Y \cup \{g_1,g_2\})$.
    Recall that there is a $4$-element cocircuit $C_6^*$ containing $\{d_0,g_2\}$.
    By orthogonality with $C_2$ this cocircuit meets $\{d_1,x,c_2\}$, but cannot contain $d_1$, by orthogonality with $C_0$, $C_1$, and $C_3$, and cannot contain $x$, by orthogonality with $C_0$, $C_1$, and $C_4$.
    So, by orthogonality with $C_3$, either $C_6^* = \{d_0,g_2,c_2,e\}$ or $C_6^* = \{d_0,g_2,c_2,h_2\}$ for $h_2 \in E(M)-(Y \cup \{g_1,g_2\})$.

    Next we claim that there is a cocircuit~$C_7^*$ where either $C_7^*=\{d_1,g_1,x,p\}$ for $p \in E(M)-(Y \cup \{g_1,g_2\})$, or $C_7^* = \{d_1,g_1,g_2,d_2\}$ in which case one of $\{x,f,c_1,c_2\}$, $\{x,f,d_1,d_2\}$, or $\{x,f,g_1,g_2\}$ is also a cocircuit.
    Since $M \ba d_1 / c_1 / x$ has an $N$-minor, and $\{d_2,g_1\}$ is a parallel pair in this matroid, $M \ba d_1 \ba g_1$ has an $N$-minor.
    So there is a $4$-element cocircuit~$C_7^*$ containing $\{d_1,g_1\}$.
    By orthogonality with $C_2$, the cocircuit~$C_7^*$ meets $\{x,c_2,g_2\}$.
    But $c_2 \notin C_7^*$, by orthogonality with $C_0$ and $C_4$.
    If $g_2 \in C_7^*$, then by orthogonality with $C_1$ and $C_4$, we have $C_7^* = \{d_1,g_1,g_2,d_2\}$; 
    otherwise, $x \in C_7^*$, and, by orthogonality with $C_0$ and $C_4$, we have $C_7^*=\{d_1,g_1,x,p\}$ for $p \in E(M)-(Y \cup \{g_1,g_2\})$.
    Now, in the case that $C_7^* = \{d_1,g_1,g_2,d_2\}$, the matroid $M \ba d_1 \ba d_2 / c_1 / c_2$ has an $N$-minor and $\{g_1,g_2\}$ is a series pair, so $\{c_2,g_1\}$ and $\{c_1,g_2\}$ are $N$-contractible in $M$.
    Since $M / c_1 / g_1$ has an $N$-minor and $\{x,d_2\}$ and $\{d_1,f\}$ are parallel pairs in this matroid (due to $C_1$ and $C_3$), the pair $\{x,f\}$ is $N$-deletable.
    So $\{x,f\}$ is contained in a $4$-element cocircuit~$C_8^*$.
    By orthogonality with $C_3$ and $C_4$, the cocircuit~$C_8^*$ meets the disjoint sets $\{c_1,d_1,g_1\}$ and $\{c_2,d_2,g_2\}$.
    Again by orthogonality, this time with $C_0$, $C_1$, and $C_2$, the cocircuit $C_8^*$ is one of $\{x,f,c_1,c_2\}$, $\{x,f,d_1,d_2\}$, or $\{x,f,g_1,g_2\}$.

    Now we claim that there is a $4$-element circuit~$C_5$ containing $\{h_1,x,d_0\}$ and one of $g_1$, $g_2$, or some $g' \in E(M)-(Y \cup \{g_1,g_2,h_1\})$.
    Since $M \ba d_0 / c_1 / x$ has an $N$-minor, and $\{d_2,g_1\}$ is a parallel pair in this matroid, $M \ba d_0 \ba g_1 / x$ has an $N$-minor.
    As $\{h_1,c_1\}$ is a series pair in the latter matroid, $M / h_1 / x$ has an $N$-minor.
    So $\{h_1,x\}$ is contained in a $4$-element circuit~$C_5$.
    By orthogonality with $C_1^*$, the circuit~$C_5$ meets $\{d_1,c_1,d_0\}$.
    But, by orthogonality with $C_3^*$ and $C_5^*$, we have $d_1 \notin C_5$.
    If $c_1 \in C_5$, then, by orthogonality with $C_2^*$ and $C_4^*$, we have $C_5=\{h_1,x,c_1,d_2\}$, but then $C_1 \cup h_1$ is a $5$-element plane that intersects $C_1^*$ in two elements; a contradiction.
    So $d_0 \in C_5$.
    By orthogonality, $C_5$ does not meet $C_3^*$ or $C_4^*$, so the final element of $C_5$ is either $g_1$, $g_2$, or some $g' \in E(M)-(Y \cup \{g_1,g_2,h_1\})$.

    Suppose $C_5=\{h_1,x,d_0,g_1\}$. Then, 
    by orthogonality, $C_6^* = \{d_0,g_2,c_2,h_1\}$. 
    First let $C_7^* = \{d_1,g_1,g_2,d_2\}$, so there is a circuit $C_8^*$ containing $\{x,f\}$.
    Then $C_8^* = \{x,f,g_1,g_2\}$, by orthogonality with $C_5$.
    Since $M \ba d_0 \ba g_2 / x$ has an $N$-minor and $\{c_2,h_1\}$ is a series pair in this matroid, $M \ba g_2 / x / h_1$ has an $N$-minor.
    Due to the circuit~$C_5$, it follows that $M \ba g_1 \ba g_2 / h_1$ has an $N$-minor.
    Now, due to the cocircuit~$C_8^*$, the pair $\{f,h_1\}$ is $N$-contractible.
    So there is a $4$-element circuit~$C_6$ containing $\{f,h_1\}$.
    By orthogonality with $C_3^*$, $C_4^*$, $C_5^*$, and $C_6^*$, this circuit meets each of the sets $\{c_1,e,d_2\}$, $\{c_2,e,d_1\}$, $\{c_1,d_0,g_1\}$, and $\{c_2,d_0,g_2\}$.
    Thus either $C_6 = \{f,h_1,c_1,c_2\}$ or $C_6 = \{f,h_1,d_0,e\}$.
    But in either case $C_6$ intersects the cocircuit $C_1^*$ in a single element; a contradiction.

    Now let $C_7^* = \{d_1,g_1,x,p\}$ for some $p \in E(M)-(Y \cup \{g_1,g_2\})$.
    Since $M \ba d_0 \ba d_1 \ba g_1 / c_1$ has an $N$-minor, and $\{x,p\}$ is a series pair in this matroid, $\{c_1,p\}$ is $N$-contractible.
    So there is a $4$-element circuit containing $\{c_1,p\}$.
    There are two cases to consider, depending on whether or not $p=h_1$.
    In the case that $p=h_1$, the circuit contains $\{c_1,h_1\}$ and meets $\{d_0,x,d_1\}$ and $\{d_2,e,f\}$, by orthogonality with $C_1^*$ and $C_3^*$.
    By orthogonality with $C_4^*$, the circuit contains $d_2$; then, by orthogonality with $C_2^*$ and $C_6^*$, the final element is $d_0$.
    But then the circuit is $\{c_1,h_1,d_2,d_0\}$, and intersects $C_7^*$ in a single element; a contradiction.
    Now $p \neq h_1$, and the $4$-element circuit containing $\{c_1,p\}$ meets $\{d_1,g_1,x\}$ and $\{d_2,e,f\}$ by orthogonality with $C_7^*$ and $C_3^*$.
    By orthogonality with $C_4^*$, the circuit contains $d_2$; then by orthogonality with $C_1^*$ and $C_2^*$, the circuit is $\{c_1,p,x,d_2\}$.
    But then $C_1 \cup p$ is a $5$-element plane that intersects $C_1^*$ in two elements; a contradiction.

    \smallskip
    Suppose $C_5=\{h_1,x,d_0,g_2\}$. 
    Then, by orthogonality, $C_7^* = \{d_1,g_1,x,h_1\}$. 
    Now, $M \ba d_0 \ba d_1 \ba g_1$ has an $N$-minor, and $\{c_1,x,h_1\}$ is contained in a series class in this matroid.
    So, in particular, $M \ba d_1 / x / h_1$ has an $N$-minor.
    As $\{d_0,g_2\}$ is a parallel pair in this matroid, the pair $\{d_1,g_2\}$ is $N$-deletable in $M$.
    Thus there is a $4$-element cocircuit~$C_9^*$ containing $\{d_1,g_2\}$.
    By orthogonality with $C_3$, $C_4$, and $C_5$, this cocircuit meets $\{c_1,f,g_1\}$, $\{c_2,d_2,f\}$, and $\{h_1,x,d_0\}$.
    Hence $f \in C_9^*$.
    By orthogonality with $C_0$ and $C_1$, we have $d_0 \notin C_9^*$ and $x \notin C_9^*$, so $C_9^* = \{d_1,g_2,f,h_1\}$.
    Recall that $\{h_1,c_1\}$ is $N$-contractible.
    Thus there is a $4$-element circuit~$C_7$ containing $\{h_1,c_1\}$.
    By orthogonality with $C_1^*$ and $C_9^*$, this circuit meets $\{d_1,x,d_0\}$ and $\{d_1,g_2,f\}$.
    It follows that $C_7$ does not meet $C_2^*$, so $d_1 \in C_7$.
    Then, by orthogonality with $C_3^*$, the final element in $C_7$ is either $e$ or $f$.
    But if $f \in C_7$, then $C_3 \cup h_1$ is a $5$-element plane that intersects the cocircuit $C_1^*$ in two elements; a contradiction.
    So $C_7=\{h_1,c_1,d_1,e\}$.

    Now let $Y' = C_3 \cup C_5 = C_1^* \cup C_5^* \cup C_9^*$, so $r^*(Y') \le 5$ and $r(Y') \le 6$.
    Then $d_2,c_2,e \in \cl(Y')$, due to $C_1$, $C_2$, and $C_7$ respectively, and $c_2,d_2 \in \cocl(Y' \cup e)$, due to $C_3^*$ and $C_4^*$.
    Hence $r^*(Y' \cup \{d_2,c_2,e\}) \le 6$ and $r(Y' \cup \{d_2,c_2,e\}) \le 6$, so $\lambda(Y' \cup \{d_2,c_2,e\}) \le 1$, implying $|E(M)| \le 12$.
    Recall that $C_6^* = \{d_0,g_2,c_2,e\}$ or $C_6^* = \{d_0,g_2,c_2,h_2\}$ for $h_2 \in E(M)-(Y \cup \{g_1,g_2\})$.
    By orthogonality with $C_7$, we have $C_6^* = \{d_0,g_2,c_2,h_2\}$ and $h_1 \neq h_2$.
    Now, as the elements $g_1,g_2,h_1,h_2 \in E(M)-Y$ are distinct, $|E(M)| = 12$.
    We work towards showing that $M$ is a \tcn.



    As $M \ba d_0 \ba g_2 / x$ has an $N$-minor, and $\{c_2,h_2\}$ is a series pair in this matroid, 
    $\{h_2,x\}$ is $N$-contractible.
    So there is a $4$-element circuit~$C_8$ containing $\{h_2,x\}$.
    By orthogonality with $C_6^*$, this circuit contains an element in $\{d_0,g_2,c_2\}$.
    If $g_2 \in C_8$, then, by orthogonality with $C_1^*$ and $C_2^*$, we have $C_8=\{h_2,x,g_2,d_0\}$, in which case $|C_8 \cap C_5^*|=1$; a contradiction.
    If $c_2 \in C_8$, then, by orthogonality with $C_3^*$ and $C_4^*$, we have $C_8=\{h_2,x,c_2,c_1\}$, in which case, again, $|C_8 \cap C_5^*|=1$; a contradiction.
    So $d_0 \in C_8$, in which case the final element is in $\{g_1,c_1,h_1\}$, by orthogonality with $C_5^*$.
    But the final element cannot be $c_1$, by orthogonality with $C_3^*$, and cannot be $h_1$, by orthogonality with $C_9^*$.
    So $C_8=\{h_2,x,d_0,g_1\}$.

    We focus on uncovering the remaining circuits in the \tcn.
    By circuit elimination on $C_3$ and $C_7$, there is a circuit contained in $\{c_1,e,f,g_1,h_1\}$.
    By orthogonality with $C_1^*$, we have that $C_6 = \{e,f,g_1,h_1\}$ is a circuit.
    By circuit elimination on $C_5$ and $C_8$, the set $\{x,g_1,g_2,h_1,h_2\}$ contains a circuit, which is $C_9=\{g_1,g_2,h_1,h_2\}$, by orthogonality with $C_1^*$.
    Similarly, the set $\{c_1,f,x,d_1,d_2\}$ contains a circuit, by circuit elimination on $C_1$ and $C_3$, so $C_{10}=\{f,x,d_1,d_2\}$ is a circuit by orthogonality with $C_5^*$.
    By circuit elimination on $C_6$ and $C_9$, there is a circuit contained in $\{g_1,e,f,g_2,h_2\}$, which is $C_{11}=\{e,f,g_2,h_2\}$ by orthogonality with $C_5^*$.
    Now, by circuit elimination on $C_4$ and $C_{11}$, there is a circuit contained in $\{f,c_2,d_2,e,h_2\}$; by orthogonality with $C_9^*$, the circuit is $C_{12}=\{c_2,d_2,e,h_2\}$.
    The set $\{x,c_1,d_0,d_2,h_2\}$ also contains a circuit, by circuit elimination on $C_1$ and $C_8$; by orthogonality with $C_7^*$, the circuit is $C_{13}=\{c_1,d_0,d_2,h_2\}$.
    Finally, by circuit on elimination on $C_0$ and $C_7$, and orthogonality with $C_3^*$, we have that $C_{14}=\{c_2,d_0,d_1,h_1\}$ is a circuit.

    We now turn to the remaining cocircuits.
    By cocircuit elimination on $C_3^*$ and $C_4^*$, there is a cocircuit contained in $\{e,c_1,c_2,d_1,d_2\}$, which, by orthogonality with $C_6$, is $C_0^* = \{c_1,c_2,d_1,d_2\}$.
    By cocircuit elimination on $C_7^*$ and $C_9^*$, there is a cocircuit contained in $\{d_1,x,f,g_1,g_2\}$; by orthogonality with $C_7$, this cocircuit is $C_8^* = \{x,f,g_1,g_2\}$.
    By cocircuit elimination on $C_2^*$ and $C_6^*$, there is a cocircuit contained in $\{d_0,d_2,x,g_2,h_2\}$; by orthogonality with $C_0$, this cocircuit is $C_{10}^* = \{d_2,x,g_2,h_2\}$.
    By cocircuit elimination on $C_7^*$ and $C_{10}^*$, there is a cocircuit contained in $\{x,d_2,f,g_1,h_2\}$; by orthogonality with $C_2$, this cocircuit is $C_{11}^* = \{d_2,f,g_1,h_2\}$.
    By cocircuit elimination on $C_3^*$ and $C_{11}^*$, there is a cocircuit contained in $\{f,c_1,e,g_1,h_2\}$; by orthogonality with $C_4$, this cocircuit is $C_{12}^* = \{c_1,e,g_1,h_2\}$.
    By cocircuit elimination on $C_4^*$ and $C_9^*$, there is a cocircuit contained in $\{f,c_2,e,g_2,h_1\}$; by orthogonality with $C_3$, the cocircuit is $C_{13}^* = \{c_2,e,g_2,h_1\}$.
    Finally, by cocircuit elimination on $C_6^*$ and $C_{13}^*$, the set $\{c_2,d_0,e,h_1,h_2\}$ contains a cocircuit; by orthogonality with $C_2$, the cocircuit is $C_{14}^* = \{d_0,e,h_1,h_2\}$.
    Now, using the labelling $(e_1, e_2, e_3, e_4, e_5, e_6, e'_1, e'_2, e'_3, e'_4, e'_5, e'_6) = (d_1, d_0, g_1, d_2, e, g_2, h_2, f, c_2, h_1, x, c_1)$, $M$ is a \tcn.  So \cref{wtd2i4} holds.

    \smallskip
    Now suppose $C_5=\{h_1,x,d_0,g'\}$ for $g' \in E(M)-(Y \cup \{g_1,g_2,h_1\})$.
    Note 
    that either $C_7^* = \{d_1,g_1,x,p\}$ for $p \in \{h_1,g'\}$, or $C_7^* = \{d_1,g_1,g_2,d_2\}$ in which case there is another cocircuit $C_8^*$ containing $\{x,f\}$.
    In the latter case, the cocircuit $C_8^*$ is one of $\{x,f,c_1,c_2\}$, $\{x,f,d_1,d_2\}$, or $\{x,f,g_1,g_2\}$.
    But then $C_8^*$ intersects $C_5$ in a single element~$x$; a contradiction.
    So $C_7^* = \{d_1,g_1,x,p\}$ for $p \in \{h_1,g'\}$.
    Suppose that $p=g'$, so $C_7^* = \{d_1,g_1,x,g'\}$.
    Since $M \ba d_1 / c_1 /x$ has an $N$-minor, $M \ba d_1 \ba g_1 / c_1$ also has an $N$-minor (due to $C_1$), and hence, due to $C_7^*$, the pair $\{c_1,g'\}$ is $N$-contractible in $M$.
    Consider the $4$-element circuit containing $\{c_1,g'\}$.
    By orthogonality with $C_1^*$ and $C_3^*$, this circuit meets the disjoint sets $\{x,d_0,d_1\}$ and $\{e,f,d_2\}$.
    Due to $C_5^*$, the circuit contains $d_0$, and then $d_2$, due to $C_2^*$.
    But then it meets $C_7^*$ in a single element; a contradiction.

    Now suppose $p=h_1$, so $C_7^* = \{d_1,g_1,x,h_1\}$.
    Since $M \ba d_0 \ba d_1 \ba g_1 / x$ has an $N$-minor, and $\{h_1,c_1\}$ is a series pair in this matroid, $M \ba d_1 / x / h_1$ has an $N$-minor.
    Due to $C_5$, it follows that $\{g',d_1\}$ is $N$-deletable in $M$.
    So there is a $4$-element cocircuit~$C_9^*$ containing $\{g',d_1\}$.
    By orthogonality with $C_3$ and $C_5$, this cocircuit meets the disjoint sets $\{c_1,f,g_1\}$ and $\{x,d_0,h_1\}$.
    Then, by orthogonality with $C_2$, we have $x \in C_9^*$.
    Again by orthogonality, this time with $C_1=\{g_1,c_1,x,d_2\}$, we see that either $C_9^* = \{g',d_1,x,g_1\}$ or $C_9^* = \{g',d_1,x,c_1\}$.
    But in the first case, $\{g',d_1,x,g_1,h_1\}$ is a $5$-element coplane, while in the second case, $C_1^* \cup g'$ is a $5$-element coplane; in either case, the $5$-element coplane intersects the circuit $C_1$ in two elements, contradicting orthogonality.

    \medskip
    Finally we assume $g=g_1 = g_3$ where $g$, $g_2$, and $g_4$ are distinct.
    Recall that $C_5^* = \{d_0,g,c_1,h\}$ where either $h=e$ or $h \in E(M)-(Y \cup \{g,g_2,g_4\})$.
    Also, either $\{d_0,g_2,c_2\} \subseteq C_6^*$ or $C_6^* = \{d_0,g_2,e,d_2\}$.
    By orthogonality between $C_6^*$ and $C_2$, the latter is not possible; it follows, by orthogonality with $C_1$, $C_3$, and $C_4$, that $C_6^* = \{d_0,g_2,c_2,g_4\}$.
    Now, as $M \ba d_0 \ba g_2 / x$ has an $N$-minor, and $\{c_2,g_4\}$ is a series pair in this matroid, the pair $\{x,g_4\}$ is $N$-contractible.
    Thus there is a $4$-element circuit~$C_5$ containing $\{x,g_4\}$.
    If $d_0 \notin C_5$, then by orthogonality with $C_6^*$ and $C_1^*$, the circuit $C_5$ meets $\{c_2,g_2\}$ and $\{c_1,d_1\}$.
    In this case, by orthogonality with $C_3^*$ and $C_4^*$, we have $C_5 = \{x,g_4,c_2,d_1\}$, but then $C_2 \cup g_4$ is a $5$-element plane intersecting $C_1^*$ in two elements; a contradiction.
    So $d_0 \in C_5$, and thus, by orthogonality with $C_5^*$ and $C_3^*$, either $C_5 = \{x,g_4,d_0,g\}$ or $C_5 = \{x,g_4,d_0,h\}$ and $h \neq e$.

    Consider the latter case, where $C_5 = \{x,g_4,d_0,h\}$ for $h \in E(M) - (Y \cup \{g,g_2,g_4\})$.
    Then, as $M \ba g_2 / x / g_4$ has an $N$-minor and $\{d_0,h\}$ is a parallel pair in this matroid, the pair $\{h,g_2\}$ is $N$-deletable.
    So there is a $4$-element cocircuit containing $\{h,g_2\}$.
    By orthogonality with $C_5$ this cocircuit meets $\{x,g_4,d_0\}$.
    But if the cocircuit contains $d_0$, then by orthogonality with $C_0$ it also meets $\{c_1,c_2,e\}$, in which case it intersects $C_3^*$ or $C_4^*$ in a single element; a contradiction.
    If the cocircuit contains $g_4$, then, by orthogonality with $C_2$ and $C_4$, it contains $c_2$, and thus intersects $C_0$ in a single element; a contradiction.
    So the cocircuit contains $\{h,g_2,x\}$, in which case, by orthogonality with $C_1$, the final element is in $\{g,c_1,d_2\}$, but then it intersects $C_3$ or $C_4$ in a single element; a contradiction.

    So we may now assume that $C_5 = \{x,g_4,d_0,g\}$.
    Then, as $M \ba g_2 / x / g_4$ has an $N$-minor and $\{d_0,g\}$ is a parallel pair in this matroid, the pair $\{g,g_2\}$ is $N$-deletable.
    So there is a $4$-element cocircuit containing $\{g,g_2\}$.
    By orthogonality with $C_1$, $C_2$, and $C_3$, the cocircuit meets $\{c_1,x,d_2\}$, $\{c_2,x,d_1\}$, and $\{c_1,f,d_1\}$.
    Thus, if the cocircuit contains $d_2$, then it is $\{g,g_2,d_2,d_1\}$, but then it intersects $C_4$ in a single element; a contradiction.
    If it contains $c_1$, then, by orthogonality with $C_0$, it also contains $c_2$, and we again obtain the contradiction that the cocircuit intersects $C_4$ in a single element.
    So it contains $x$ and meets $\{c_1,f,d_1\}$.
    By orthogonality with $C_0$ and $C_4$, the cocircuit is $\{g,g_2,x,d_1\}$.
    By circuit elimination on $C_1^*$ and $C_5^*$, there is also a cocircuit contained in $\{c_1,g,h,x,d_1\}$; by orthogonality with $C_0$, this cocircuit is $\{g,h,x,d_1\}$.
    But then $\{g,g_2,h,x,d_1\}$ is a $5$-element coplane that intersects $C_1$ in two elements; a contradiction.
  \end{slproof}

  \begin{sublemma}
    \label{qqcc3}
    Suppose $|C_1^* \cap C_2^*| = 2$ and $d_1 \notin C_2^*$.
    Then the lemma holds.
  \end{sublemma}
  \begin{slproof}
    Let $X = C_1^* \cup C_2^*$. 
    If $X$ is $3$-separating, then \cref{wtd2i1} holds by \cref{qqcc2}.
    So we may assume that $X$ is not $3$-separating.
    Since $r^*(X)=4$, by \cref{qqcc1}, we have $r(X) \in \{5,6\}$.
    Let $C_1^* = \{d_0,d_1,c_1,x\}$ and $C_2^* = \{d_0,d_2,c_2,x\}$ for distinct $c_1,c_2,x \in E(M)-\{d_0,d_1,d_2\}$.
    Since $M \ba d_0 \ba d_1 \ba d_2$ has an $N$-minor, and $\{x,c_1,c_2\}$ is contained in a series class in this matroid, the pair $\{c_1,c_2\}$ is $N$-contractible.
    Hence, this pair is contained in a $4$-element circuit~$C_0$. 
    By orthogonality, $|C_0 \cap X| \ge 3$.
    If $C_0 \subseteq X$, then the lemma holds by \cref{qqcc29}.

    So we may assume that $C_0 \nsubseteq X$. 
    By orthogonality, $C_0 = \{c_1,c_2,x',e\}$ for some $x' \in C_1^* \cap C_2^*$ and $e \in E(M)-X$.
    If $x' = x$, then, as $M \ba d_0 \ba d_1 \ba d_2 / c_1 / c_2$ has an $N$-minor, and $\{x,e\}$ is a parallel pair in this matroid, $M \ba x \ba d_1 \ba d_2$ has an $N$-minor.
    So we may assume, up to swapping the labels on $x$ and $d_0$, that $x' = d_0$; that is, $C_0 = \{c_1,c_2,d_0,e\}$.
    Moreover, if $e \in \cocl(X)$, then the coindependent circuit $C_0$ cospans $X \cup e$, and $\{d_1,d_2\} \subseteq \cocl(C_0)-C_0$, so \cref{wtd2i2} holds by \cref{qqcc0}.
    So $e \notin \cocl(X)$.

    As $M \ba d_1 / c_1 /c_2$ has an $N$-minor, and $\{d_0,e\}$ is a parallel pair in this matroid, $\{d_1,e\}$ is $N$-deletable.
    Let $C_3^*$ be the $4$-element cocircuit containing $\{d_1,e\}$.
    Then $C_3^*$ meets $\{c_1,d_0,c_2\}$, by orthogonality with $C_0$.
    Since $e \notin \cocl(X)$, the cocircuit $C_3^*$ also contains an element $f_1 \in E(M)-(X \cup e)$.
    Similarly, we let $C_4^*$ be the cocircuit containing $\{d_2,e\}$, and observe that $C_4^*$ contains an element $f_2 \in E(M)-(X \cup e)$.

    Suppose that $C_3^* = \{d_1,e,c_1,f_1\}$.
    Let $C_1$ be the circuit containing $\{c_1,x\}$.  By orthogonality with $C_3^*$ and $C_2^*$, the circuit $C_1$ meets $\{d_1,e,f_1\}$ and $\{c_2,d_2,d_0\}$.
    We claim that $e \notin C_1$.
    Towards a contradiction, suppose $e \in C_1$.  Then $C_1$ does not meet $\{c_2,d_0\}$, as otherwise $r(C_0 \cup x)=3$, in which case $C_0 \cup x$ is a $5$-element plane that intersects the cocircuit $C_3^*$ in two elements, contradicting orthogonality.
    So $C_1 = \{c_1,x,e,d_2\}$.
    But in this case, by circuit elimination with $C_0$, the set $\{c_1,x,d_0,c_2,d_2\}$ contains a circuit.  This circuit cannot contain $c_1$, by orthogonality with $C_3^*$, so $\{x,d_0,c_2,d_2\} = C_2^*$ is a quad; a contradiction.

    Now $C_1$ meets $\{d_1,f_1\}$ and $\{c_2,d_2,d_0\}$.  Suppose $d_1 \in C_1$.
    If $C_1 = \{c_1,x,d_1,c_2\}$, then $\{d_0,d_2\} \subseteq \cocl(C_1)-C_1$, so \cref{wtd2i2} holds by \cref{qqcc0}.
    If $C_1 = \{c_1,x,d_1,d_0\}$, then $C_1$ is a quad; a contradiction.
    So $C_1 = \{c_1,x,d_1,d_2\}$.
    Recall the cocircuit $C_4^*$ containing $\{d_2,e\}$ and an element $f_2 \in E(M)-(X \cup e)$.
    By orthogonality with $C_0$ and $C_1$, this cocircuit meets $\{c_1,d_0,c_2\}$ and $\{c_1,x,d_1\}$.  So $C_4^* = \{d_2,e,c_1,f_2\}$.

    
    By cocircuit elimination on $C_3^*$ and $C_4^*$, the set $\{d_1,d_2,e,f_1,f_2\}$ contains a cocircuit.
    But this set intersects $C_0$ in a single element, $e$, so $C_0^* = \{d_1,d_2,f_1,f_2\}$ is a cocircuit, and, in particular, $f_1 \neq f_2$.

    We work towards a contradiction by showing $\{f_1,f_2\} \subseteq \cl(X \cup e)$.
    Since $M \ba d_0 \ba d_1 \ba d_2$ has an $N$-minor, and $\{f_1,f_2\}$ and $\{c_2,x\}$ are series pairs in this matroid, $\{f_1,c_2\}$ and $\{f_2,x\}$ are $N$-contractible.
    First, consider the $4$-element circuit containing $\{f_1,c_2\}$.
    By orthogonality, 
    it
    meets $\{d_1,e,c_1\}$ and $\{x,d_0,d_2\}$.  So $f_1 \in \cl(X \cup e)$.
    Now consider the $4$-element circuit 
    containing $\{f_2,x\}$.
    By orthogonality with $C_1^*$ and $C_2^*$, this circuit meets $\{c_1,d_1,d_0\}$ and $\{c_2,d_2,d_0\}$.
    Suppose this circuit contains $d_0$.
    By orthogonality with $C_0^*$ and $C_4^*$, the circuit 
    also meets $\{d_1,d_2,f_1\}$ and $\{d_2,e,c_1\}$, so in this case the circuit is $\{f_2,x,d_0,d_2\}$.
    Thus $f_2 \in \cl(X)$ in either case,
    and $r(X \cup \{e,f_1,f_2\}) = r(X)=5$.
    But $\{c_1,d_0,d_1,d_2,f_1\}$ (for example) cospans $X \cup \{e,f_1,f_2\}$, so $r^*(X\cup \{e,f_1,f_2\}) =5$, implying $\lambda(X\cup \{e,f_1,f_2\}) = 1$.
    Thus $|E(M)| = 10$; a contradiction.

    We deduce that $d_1 \notin C_1$, so $C_1=\{c_1,x,p,f_1\}$ for some $p \in \{c_2,d_2,d_0\}$.
    Recall the cocircuit $C_4^*$ that contains $\{d_2,e\}$ and an element $f_2 \in E(M)-(X \cup e)$.
    By orthogonality with $C_0$, we see $C_4^* = \{d_2,e,q,f_2\}$ for some $q \in \{c_1,c_2,d_0\}$.
    By orthogonality between $C_1$ and $C_4^*$, either $(p,q) \in \{(d_2,c_1),(c_2,d_0),(d_0,c_2)\}$ or $f_1 = f_2$.

    Before considering these subcases, suppose that $\{c_1,x,c_2,f_1\}$ is a circuit.
    Let $Y = C_0 \cup \{x,f_1\}$, so $\{c_1,x,c_2,f_1\} \subseteq Y$.
    Observe that $r(Y) \le 4$ and $r^*_{M \ba d_1}(Y) \le 4$.
    Hence $Y$ is $3$-separating in $M \ba d_1$, and $d_2 \in \cocl_{M \ba d_1}(Y)$.
    Towards an application of \cref{3sepwin2}, it remains to show that $|Y - (C \cup D)| \le 2$ for an $N$-labelling $(C,D)$ with $\{d_1,d_2\} \subseteq D$.
    To this end, note that $M \ba d_0 \ba d_1 \ba d_2 / c_1 / x$ has an $N$-minor, and this matroid has the parallel pair $\{c_2,f_1\}$, so we may assume that $|Y-(C \cup D)| \le 2$ as required.
    Now \cref{wtd2i2} holds in this case, by \cref{3sepwin2}.
    We may now assume that $\{c_1,x,c_2,f_1\}$ is not a circuit; in particular, $p \neq c_2$.

    Suppose $q = c_1$, so $C_4^* = \{d_2,e,c_1,f_2\}$.
    By cocircuit elimination on $C_3^*$ and $C_4^*$, there is a cocircuit contained in $\{d_1,d_2,e,f_1,f_2\}$.  But this set intersects $C_0$ in a single element, $e$, so $C_0^* = \{d_1,d_2,f_1,f_2\}$ is a cocircuit.
    In particular, when $q=c_1$, we have $f_1 \neq f_2$.

    We now consider cases.  First,
    suppose $(p,q) = (d_0,c_2)$; we will show this leads to a contradiction.
    Recall that $C_1 = \{c_1,x,d_0,f_1\}$ and $C_4^* = \{d_2,e,c_2,f_2\}$, so $f_1 \neq f_2$, by orthogonality.
    Observe that $\{e,f_1\} \subseteq \cl(X-d_2)$.
    As the pair $\{c_2,x\}$ is $N$-contractible, it is contained in a $4$-element circuit~$C_2$ that, by orthogonality with $C_1^*$ and $C_4^*$, meets $\{c_1,d_1,d_0\}$ and $\{d_2,e,f_2\}$.

    We claim that either $C_2 = \{c_2,x,d_0,f_2\}$ or $C_2 = \{c_2,x,d_1,e\}$.
    First, observe that if $c_1 \in C_2$, then, by orthogonality with $C_3^*$, we have $C_2 = \{c_2,x,c_1,e\}$.
    But then $C_0 \cup C_2$ is a $5$-element plane that intersects the cocircuit $C_4^*$ in two elements; a contradiction.
    So $C_2$ meets $\{d_1,d_0\}$.
    If $d_0 \in C_2$, then $C_2$ intersects $C_3^*$ in at most one element, so $C_2 \cap C_3^* = \emptyset$.  As $C_2$ meets $\{d_2,e,f_2\}$, and $C_2 \neq C_2^*$, we deduce that $C_2 = \{c_2,x,d_0,f_2\}$ when $d_0 \in C_2$.
    Finally, if $d_1 \in C_2$, then, by orthogonality with $C_3^*$, the final element is in $\{c_1,e,f_1\}$.  As the only element of this set that is in $\{d_2,e,f_2\}$ is $e$, we have $C_2 = \{c_2,x,d_1,e\}$.
    This proves the claim.

    Next we claim that $d_2 \in \cl(X-d_2)$, so $r(X) = 5$.
    Observe that $M \ba d_1 /c_1 /c_2$ has an $N$-minor, and $\{d_0,e\}$ is a parallel pair in this matroid, so $M \ba e \ba d_1 /c_2$ also has an $N$-minor.  As $\{c_1,f_1\}$ is a series pair in the latter matroid, $\{f_1,c_2\}$ is $N$-contractible.
    So there is a $4$-element circuit~$C_3$ containing $\{f_1,c_2\}$.
    By orthogonality with $C_2^*$ and $C_3^*$, this circuit contains an element in $\{x,d_0,d_2\}$, and an element in $\{d_1,c_1,e\}$.
    If $d_2 \notin C_3$, then $C_3$ meets $C_1^*$, so by orthogonality $|C_3 \cap C_1^*|=2$.
    But then $C_3 \cap C_4^* = \{c_2\}$; a contradiction.
    So $d_2 \in C_3$, in which case $C_3$ intersects $C_1^*$ in at most one element; thus $C_3 \cap C_1^* = \emptyset$, by orthogonality.
    Hence $C_3 = \{f_1,c_2,d_2,e\}$, so $d_2 \in \cl((X-d_2) \cup \{e,f_1\}) = \cl(X-d_2)$, as required.

    Now, if $C_2 = \{c_2,x,d_0,f_2\}$, then $\{d_2,e,f_1,f_2\} \subseteq \cl(X-d_2)$, so $r(X \cup \{f_1,f_2,e\}) = 5$.  As $r^*(X \cup \{f_1,f_2,e\}) \le 5$, we have $\lambda(X \cup \{f_1,f_2,e\}) \le 1$, so $|E(M)| \le 10$; a contradiction.
    So we may assume that $C_2 = \{c_2,x,d_1,e\}$.
    By circuit elimination with $C_0$, there is a circuit contained in $\{c_1,c_2,x,d_0,d_1\}$.
    By orthogonality with $C_4^*$, this circuit does not contain $c_2$.
    Thus $C_1^*=\{c_1,x,d_0,d_1\}$ is a quad; a contradiction.

    Next, suppose $(p,q) = (d_2,c_1)$; that is, $C_1=\{c_1,x,d_2,f_1\}$ and $C_4^* = \{d_2,e,c_1,f_2\}$.
    Recall that when $q =c_1$, we have that $C_0^* = \{d_1,d_2,f_1,f_2\}$ is a cocircuit, and $f_1 \neq f_2$.
    We work towards a contradiction by showing that $r(X \cup \{e,f_1,f_2\}) = 5$.
    To start with, note that $\{e,f_1\} \subseteq \cl(X-d_1)$.
    As $\{x,c_2\}$ and $\{f_1,f_2\}$ are series pairs in $M \ba d_0 \ba d_1 \ba d_2$, the pairs $\{x,f_2\}$ and $\{c_2,f_1\}$ are $N$-contractible.
    Let $C_3$ be the $4$-element circuit containing $\{x,f_2\}$.  This circuit meets $\{d_0,c_1,d_1\}$ by orthogonality.
    This circuit can only intersect $C_3^* = \{c_1,d_1,e,f_1\}$ in at most one element, so it does not meet $C_3^*$.
    Thus $d_0 \in C_3$.  By orthogonality with $C_4^*$, the final element is in $\{d_2,e,c_1\}$.  But as $\{e,c_1\} \subseteq C_3^*$, we have $C_3=\{x,f_2,d_0,d_2\}$.
    In particular, $f_2 \in \cl(X-d_1)$.


    We claim that $d_1 \in \cl(X-d_1)$.
    Let $C_4$ be the $4$-element circuit containing $\{c_2,f_1\}$.
    By orthogonality with $C_2^*$ and $C_3^*$, the circuit~$C_4$ meets the disjoint sets $\{d_0,x,d_2\}$ and $\{d_1,e,c_1\}$.
    Moreover, due to the cocircuit $C_0^*$, either $d_1 \in C_4$ or $d_2 \in C_4$.
    Suppose $d_2 \in C_4$.
    Then $C_4$ intersects $C_1^*$ in at most one element, so $C_4 \cap C_1^* = \emptyset$, implying $C_4=\{c_2,f_1,d_2,e\}$.
    By circuit elimination on $C_1$ and $C_4$, there is a circuit contained in $\{c_1,c_2,e,x,f_1\}$.
    But this set intersects $C_0^*$ in a single element, $f_1$, so $\{c_1,c_2,e,x\}$ is a circuit.
    Now $C_0 \cup x$ is a $5$-element plane that intersects the cocircuit $C_3^*$ in two elements, contradicting orthogonality.
    So $d_1 \in C_4$.
    Thus $d_1 \in \cl(Y-d_1) = \cl(X-d_1)$, as claimed.
    Finally, as $r(X \cup \{e,f_1,f_2\})=5$, we see that $\lambda(X \cup \{e,f_1,f_2\}) \le 5 + 5 - 9 = 1$, so $|E(M)| \le 10$; a contradiction.

    Now we assume that $f_1 = f_2$.
    By orthogonality between $C_1$ and $C_4^*$, we have either $q =c_1$, $p=d_2$, or $p=q$ and $p \in \{c_2,d_0\}$.
    However, we have seen that if $q = c_1$, then $f_1 \neq f_2$; so $q \neq c_1$.

    Suppose $p=d_2$, so $C_1=\{c_1,x,d_2,f_1\}$, and recall that $C_4^*=\{d_2,e,q,f_1\}$ with $q \in \{c_1,c_2,d_0\}$.
    If $q=d_0$, then by cocircuit elimination on $C_3^*$ and $C_4^*$, there is a cocircuit contained in $\{d_0,d_1,d_2,c_1,e\}$. Then, as $e \notin \cocl(X)$, the set $\{d_0,d_1,d_2,c_1\}$ is a cocircuit, implying $r^*(X)<4$; a contradiction.
    Since $q \neq c_1$, we have $C_4^* = \{d_2,e,c_2,f_1\}$.
    As the pair $\{c_2,x\}$ is $N$-contractible, it is contained in a $4$-element circuit~$C_2$ that, by orthogonality with $C_1^*$ and $C_4^*$, meets $\{c_1,d_1,d_0\}$ and $\{d_2,e,f_1\}$.
    By orthogonality with $C_3^*$, either $\{d_0,d_2\} \subseteq C_2$ or $\{d_0,d_2\} \cap C_2 = \emptyset$.  But in the former case $C_2 = C_2^*$; a contradiction.
    Now, if $c_1 \in C_2$, then either $C_2 = \{c_2,x,c_1,e\}$ or $C_2 = \{c_2,x,c_1,f_1\}$, so $|C_2 \cap C_0|=3$ or $|C_2 \cap C_1|=3$.
    In either case, it follows that there is a $5$-element plane intersecting a cocircuit in two elements; a contradiction.
    So either $C_2 = \{c_2,x,d_1,e\}$ or $C_2 = \{c_2,x,d_1,f_1\}$.
    In the former case, by circuit elimination with $C_0$, there is a circuit contained in $\{c_1,c_2,x,d_0,d_1\}$.  By orthogonality with $C_4^*$, the set $\{c_1,x,d_0,d_1\}$ is a circuit, so $C_1^*$ is a quad; a contradiction.
    So $C_2 = \{c_2,x,d_1,f_1\}$.
    Now $M \ba d_2 /c_2 / x$ has an $N$-minor, and $\{d_1,f_1\}$ is a parallel pair in this matroid, so $M \ba d_2 \ba f_1 /x$ has an $N$-minor.
    The latter matroid has $\{c_2,e\}$ as a series pair, so $\{e,x\}$ is $N$-contractible.
    Thus $\{e,x\}$ is contained in a $4$-element circuit $C_3$.
    By orthogonality with $C_2^*$ and $C_3^*$, the circuit $C_3$ meets $\{d_0,d_2,c_2\}$ and $\{d_1,c_1,f_1\}$.
    By orthogonality with $C_1^*$ and $C_4^*$, either $\{d_0,f_1\} \subseteq C_3$, or $\{d_0,f_1\} \cap C_3 = \emptyset$.
    But if $C_3 = \{e,x,d_0,f_1\}$, then $C_3 = C_1^* \triangle C_3^*$. As $M \ba d_0 \ba d_1 \ba f_1 / c_1 / x$ has an $N$-minor, the lemma holds, by \cref{qqcc25}, in this case.
    So we may assume that $C_3$ meets $\{c_1,d_1\}$ and $\{c_2,d_2\}$.
    If $c_1 \in C_3$, then $C_0 \cup C_3$ or $C_1 \cup C_3$ is a $5$-element plane that intersects a cocircuit in two elements, contradicting orthogonality.
    Similarly if $C_3=\{e,x,d_1,c_2\}$, then $C_3 \cup C_2$ is a $5$-element plane that intersects $C_1^*$ in two elements; a contradiction.
    So $C_3 = \{e,x,d_1,d_2\}$.
    By circuit elimination with $C_2$, there is a circuit contained in $\{c_2,d_1,d_2,e,f\}$.  By orthogonality with $C_1^*$, we see that $C_3^*=\{c_2,d_2,e,f\}$ is a quad; a contradiction.
    This completes the subcase where $p=d_2$.

    Suppose $p=q$ and $p \in \{c_2,d_0\}$.
    We have seen that $p \neq c_2$, so $p=q=d_0$.
    In particular, $C_4^* = \{d_2,e,d_0,f_1\}$.
    As the pair $\{c_2,x\}$ is $N$-contractible, it is contained in a $4$-element circuit~$C_2$ that, by orthogonality with $C_1^*$, meets $\{c_1,d_1,d_0\}$.
    If $d_0 \in C_2$, then, by orthogonality with $C_4^*$, the final element of $C_2$ is in $\{d_2,e,f_1\}$.
    But $\{c_2,x,d_0,d_2\}$ is not a cocircuit, for otherwise it would be a quad; and $e,f_1 \notin \cocl(X)$.  So $d_0 \notin C_2$.
    Again using that $e,f_1 \notin \cocl(X)$, it now follows that $C_2 = \{c_2,x,c_1,d_1\}$.
    Thus $C_2$ is coindependent, and hence $\{d_0,d_2\} \subseteq \cocl(C_2)-C_2$.
    So \cref{wtd2i2} holds by \cref{qqcc0}.
    This completes the case where $C_3^* = \{d_1,e,c_1,f_1\}$.

    \medbreak

    Now we may assume that $C_3^* = \{d_1,e,s_1,f_1\}$ for $s_1 \in \{d_0,c_2\}$.
    By symmetry, $C_4^* = \{d_2,e,s_2,f_2\}$ for some $s_2 \in \{d_0,c_1\}$ and $f_2 \in E(M)-(X \cup e)$.

    Suppose that $s_1=s_2=d_0$, so $C_3^* = \{d_1,e,d_0,f_1\}$ and $C_4^* = \{d_2,e,d_0,f_2\}$.
    Then, by cocircuit elimination, $\{d_1,d_2,e,f_1,f_2\}$ contains a cocircuit.
    But this set intersects the circuit $C_0$ in a single element, $e$, so $C_0^*=\{d_1,d_2,f_1,f_2\}$ is a cocircuit.
    In particular, $f_1 \neq f_2$.

    Let $Y = X \cup \{e,f_1,f_2\}$, and note that $r^*(Y) \le 5$.  We claim that $r(Y) =r(X)$.
    As $\{d_0,d_1\} \subseteq C_3^*$ and $\{d_0,d_2\} \subseteq C_4^*$, we have $r^*(C_3^* \cup C_4^*) = 4$ by \cref{qqcc1}.
    Thus, in $M \ba d_0 \ba d_1 \ba d_2$, both $\{c_1,x,c_2\}$ and $\{f_1,e,f_2\}$ are contained in series classes.
    So a matroid obtained by contracting a pair of elements in $\{c_1,x,c_2\}$, and a pair in $\{f_1,e,f_2\}$, from $M \ba d_0 \ba d_1 \ba d_2$, has an $N$-minor.
    In particular, $\{c,f\}$ is $N$-contractible for distinct $c,f \in \{c_1,x,c_2,f_1,e,f_2\}$.
    We first consider the $4$-element circuits $C_3$ and $C_4$ containing $\{c_1,f_2\}$ and $\{c_2,f_1\}$, respectively.
    If $d_0$ is in either of these circuits, then, by orthogonality, the circuit also meets the disjoint sets $C_2^*-d_0$ and $C_3^*-d_0$; a contradiction.
    So $C_3$ meets $\{x,d_1\}$ and $\{e,d_2\}$.  By orthogonality with $C_2^*$, either $C_3 = \{c_1,f_2,x,d_2\}$ or $C_3 = \{c_1,f_2,d_1,e\}$.
    Similarly, $C_4 = \{c_2,f_1,x,d_1\}$ or $C_4 = \{c_2,f_1,d_2,e\}$.
    As $e \in \cl(X)$, we now also have $\{f_1,f_2\} \subseteq \cl(X \cup e) = \cl(X)$, so $r(Y) = r(X)$.

    Now, if $r(X)=5$, then, as $r^*(Y) \le 5$ and $|Y| = 9$, we have $\lambda(Y) \le 1$, so $|E(M)| \le 10$;
    a contradiction.
    So $r(X)=6$.

    We will show that \cref{wtd2i2} holds.
    Suppose that $C_3 = \{c_1,f_2,d_1,e\}$.
    There is also a $4$-element circuit~$C_5$ containing $\{f_2,x\}$.
    By orthogonality with $C_4^*$, this circuit meets $\{d_2,d_0,e\}$.
    If $d_2 \in C_5$, then $C_5$ meets $C_1^*-x$ but is disjoint from $C_3^*$; so $C_5 = \{f_2,x,d_2,c_1\}$.
    Now $\{e,f_2\} \subseteq \cl(X-d_1)$, so $d_1 \in \cl(X-d_1)$ due to $C_3$.
    Thus $r(X)=5$; a contradiction.
    So $d_2 \notin C_5$. Then $C_5$ meets $\{d_0,e\}$, and, by orthogonality with $C_0^*$ and $C_3^*$, the final element is in $\{d_1,f_1\}$.
    By orthogonality with $C_2^*$ it follows that $C_5 = \{f_2,x,d_0,q\}$ for some $q \in \{d_1,f_1\}$.
    Now due to $C_0$ and $C_3$, we have $\{e,f_2\} \subseteq \cl(\{c_1,c_2,d_0,d_1\})$, so if $q=d_1$ then $x \in \cl(X-x)$ and $r(X)=5$; a contradiction.
    So $C_5 = \{f_2,x,d_0,f_1\}$.
    If $C_4 = \{c_2,f_1,d_2,e\}$, then $\{f_1,f_2\} \subseteq \cl(X-x)$, so $x \in \cl(X-x)$ due to $C_5$, implying $r(X)=5$; a contradiction.
    So $C_4 = \{c_2,f_1,x,d_1\}$.
    Let $X' = C_1^* \cup C_3^*$, and observe that $r^*(X') = 4$, by \cref{qqcc1}.
    Due to the circuits $C_0$ and $C_4$ contained in $X' \cup c_2$, we have $r_{M/c_2}(X') = 4$, so $\lambda_{M/c_2}(X') =2$, with $c_2 \in \cl(X')$ and $f_2 \in \cl_{M/c_2}(X')$.
    As $\{c_2,f_2\}$ is an $N$-contractible pair, and $M \ba d_0 \ba d_1 /c_2/f_2/x/e$ has an $N$-minor, \cref{wtd2i2} holds by the dual of \cref{3sepwin2}.

    Now $C_3 = \{c_1,f_2,x,d_2\}$ and, by symmetry, $C_4 = \{c_2,f_1,x,d_1\}$.
    Let $X' = (X-d_0) \cup \{f_1,f_2\}$ and observe that $r(X') = 5$ and $r^*_{M \ba d_0}(X') = r^*(X' \cup d_0) - 1 = 4$.
    Thus $\lambda_{M \ba d_0}(X') = 2$.
    Also, $e \in \cocl_{M \ba d_0}(X')$.
    Towards an application of \cref{3sepwin2}, it remains to show that the pair $\{d_0,e\}$ is $N$-deletable, and $|X' - (C \cup D)| \le 2$ for an $N$-labelling $(C,D)$ with $\{d_0,e\} \subseteq D$.
    To this end, observe that $\{f_1,f_2\}$ is contained in a $4$-element circuit~$C_5$ that, by orthogonality with $C_3^*$ and $C_4^*$, meets $\{d_1,d_0,e\}$ and $\{d_2,d_0,e\}$.  Since $C_0^*=\{f_1,f_2,d_1,d_2\}$ is independent, $C_5$ meets $\{d_0,e\}$.
    By orthogonality with $C_1^*$ and $C_2^*$, there are two cases:
    if $d_0 \in C_5$, then $C_5 = \{f_1,f_2,d_0,x\}$; otherwise, $C_5 = \{f_1,f_2,e,g\}$ for some $g \in E(M)-Y$.
    In the first case, $M \ba d_1 \ba d_2 \ba e / c_1 /c_2/f_1/f_2$ has an $N$-minor and $\{d_0,x\}$ is a parallel pair in this matroid, so $d_0$ is $N$-deletable.
    In the second case, $M \ba d_0 \ba d_1 \ba d_2 / c_1 / c_2 /f_1 /f_2$ has an $N$-minor, and $\{e,g\}$ is a parallel pair in this matroid, so $e$ is $N$-deletable. 
    In either case, \cref{wtd2i2} holds by \cref{3sepwin2}.

    \smallbreak

    Now we may assume, by symmetry, that $s_1 = c_2$ and $s_2 \in \{d_0,c_1\}$.
    That is, $C_3^* = \{d_1,e,c_2,f_1\}$ and either $C_4^* = \{d_2,e,d_0,f_2\}$ or $C_4^* = \{d_2,e,c_1,f_2\}$.
    In the latter case, the lemma holds by \cref{awkwardmissedcase}.
    So we may assume that $C_4^* = \{d_2,e,d_0,f_2\}$.

    If $f_1 = f_2$, then, by cocircuit elimination, $\{d_1,d_2,e,c_2,d_0\}$ contains a cocircuit. Since $r^*(X) = 4$, this cocircuit must contain $e$, in which case $e \in \cocl(X)$; a contradiction.
    So $f_1 \neq f_2$.

    Now $M \ba d_0 \ba d_1 \ba d_2$ has an $N$-minor, where in this matroid $\{c_1,c_2,x\}$ is contained in a series class and $\{e,f_2\}$ is a series pair.
    Let $C_1$ be the $4$-element circuit containing $\{c_1,x\}$.
    By orthogonality with $C_2^*$, this circuit meets $\{c_2,d_2,d_0\}$.
    If $c_2 \in C_1$, then $C_1$ intersects $C_4^*$ in at most one element, so $C_1 \cap C_4^* = \emptyset$.
    By orthogonality with $C_3^*$, the final element is either $d_1$ or $f_1$.
    But if $C_1 = \{c_1,x,c_2,d_1\}$, then the lemma holds by \cref{qqcc29}.
    So $C_1 = \{c_1,x,c_2,f_1\}$.
    Let $X' = C_0 \cup C_1 = (X-\{d_1,d_2\}) \cup \{e,f_1\}$.
    Now $r(X') = 4$ and $r^*_{M \ba d_1}(X') \le r^*(X) = 4$, so $\lambda_{M \ba d_1}(X')=2$.
    As $d_2 \in \cocl_{M \ba d_1}(X')$ and $M \ba d_1 \ba d_2 \ba d_0 / c_1 /c_2 / e$ has an $N$-minor, \cref{wtd2i2} holds by \cref{3sepwin2}.

    So we may assume that $C_1$ meets $\{d_0,d_2\}$.
    Then $|C_1 \cap C_4^*|=2$, by orthogonality.
    Thus $C_1$ intersects $C_3^*$ in at most one element, so $C_1 \cap C_3^* = \emptyset$; in particular, $e \notin C_1$.
    So $C_1$ is one of $\{c_1,x,d_0,d_2\}$, $\{c_1,x,d_0,f_2\}$ or $\{c_1,x,d_2,f_2\}$.

    If $C_1 = \{c_1,x,d_0,d_2\}$, then as $C_1$ is coindependent, $\{d_1,c_2\} \subseteq \cocl(C_1)$.
    Recall that $M \ba d_1 \ba d_2 /c_1/c_2$ has an $N$-minor, where $\{d_0,e\}$ is a parallel pair in this matroid, so $M \ba d_1 \ba d_2 \ba e/ c_1$ has an $N$-minor.
    In turn, $\{d_0,f_2\}$ is a series pair, so $M \ba d_1 /c_1/d_0$ has an $N$-minor; then $\{c_2,e\}$ is a parallel pair, so $M \ba d_1 \ba c_2$ has an $N$-minor.  
    Thus \cref{wtd2i2} holds, by \cref{3sepwin2}, when $C_1 = \{c_1,x,d_0,d_2\}$.
    So $C_1 = \{c_1,x,d',f_2\}$ for some $d' \in \{d_0,d_2\}$.
    In particular, $f_2 \in \cl(X-\{c_2,d_1\})$.

    Let $Y = X \cup \{e,f_1,f_2\}$.
    Our goal is to bound $r(Y)$.
    Note that if $r(Y) \le 5$, then, as $r^*(Y) \le 5$, it follows that $\lambda(Y) \le 1$, so $|E(M)| \le 10$; a contradiction.

    Recall that $\{x,e\}$ is $N$-contractible, and consider the circuit~$C_3$ containing this pair.
    By orthogonality with $C_1^*$ and $C_2^*$, either $d_0 \in C_3$, or $C_3$ has an element in $\{c_1,d_1\}$ and an element in $\{c_2,d_2\}$.
    In the latter case, by orthogonality with $C_4^*$ we see that $d_2 \in C_3$.
    Then, by orthogonality with $C_3^*$, we see that $d_1 \in C_3$.
    On the other hand, if $d_0 \in C_3$, then the final element is in $\{d_1,c_2,f_1\}$, due to the cocircuit $C_3^*$.
    But if $C_3 = \{x,e,d_0,d_1\}$, then by circuit elimination with $C_0$, there is a circuit contained in $\{x,d_0,d_1,c_1,c_2\}$, which by orthogonality with $C_4^*$ is $\{x,d_1,c_1,c_2\}$, so the lemma holds by \cref{qqcc29}.
    If $C_3 = \{x,e,d_0,c_2\}$, then $C_0 \cup C_3$ is a $5$-element plane that intersects $C_3^*$ in two elements; a contradiction.
    Therefore we may assume that $C_3$ is $\{x,e,d_1,d_2\}$ or $\{x,e,d_0,f_1\}$.


    Suppose that $C_3 = \{x,e,d_1,d_2\}$.
    Then $d_1 \in \cl(X-d_1)$. 
%
    Recall that $M \ba d_1 \ba d_2 \ba e$ has an $N$-minor, and observe that $\{d_0,f_2\}$ and $\{f_1,c_2\}$ are series pairs in this matroid.
    Thus $\{d_0,f_1\}$ is $N$-contractible.
    The $4$-element circuit~$C_4$ containing $\{d_0,f_1\}$ meets $\{e,d_1,c_2\}$, $\{e,d_2,f_2\}$, $\{x,c_1,d_1\}$ and $\{x,c_2,d_2\}$, by orthogonality with $C_3^*$, $C_4^*$, $C_1^*$, and $C_2^*$.
    Hence $C_4$ is either $\{d_0,f_1,e,x\}$ or $\{d_0,f_1,d_1,d_2\}$.
    In either case, $f_1 \in \cl(X \cup e)$. 
    Recall that $\{d_1,e,f_2\} \subseteq \cl(X-d_1)$, so $r(Y) \le 5$. 
    Now $\lambda(Y) \le 1$, so $|E(M)| \le 10$; a contradiction.

    So we may assume that $C_3=\{x,e,d_0,f_1\}$; in particular, $f_1 \in \cl(X-\{d_1,d_2\})$.
    Now $f_2$ and $c_2$ are in distinct series pairs of $M \ba d_0 \ba d_1 \ba d_2$, so $\{f_2,c_2\}$ is $N$-contractible.
    Consider the circuit~$C_5$ containing $\{f_2,c_2\}$.
    By orthogonality with $C_3^*$, $C_2^*$ and $C_4^*$, this circuit meets $\{d_1,e,f_1\}$, $\{x,d_0,d_2\}$ and $\{e,d_0,d_2\}$.
    But if $C_5$ meets $C_1^*$, it does so in two elements; so $C_5$ is either $\{f_2,c_2,d_1,d_0\}$, $\{f_2,c_2,e,d_2\}$ or $\{f_2,c_2,f_1,d_2\}$.
%
    In the first case, $c_2 \in \cl((X \cup f_2)-c_2)= \cl(X-c_2)$, so $X-c_2$ spans $Y$, implying $\lambda(Y) \le 1$, so $|E(M)| \le 10$; a contradiction.
    In the second case, $f_2 \in \cl((X-\{x,d_1\})\cup e) = \cl(X-\{x,d_1\})$.
    Due to the circuit $C_1$, it follows that $x \in \cl(X-\{x,d_1\})$, so $X-\{x,d_1\}$ spans $Y-d_1$, implying $r(Y) \le 5$, so $\lambda(Y) \le 1$, and $|E(M)| \le 10$; a contradiction.
    So we may assume that $C_5 = \{f_2,c_2,d_2,f_1\}$.

    Recall that $C_1 = \{c_1,x,d',f_2\}$ for some $d' \in \{d_0,d_2\}$.
    Now, if $d' = d_0$,
    then $X-\{d_1,d_2\} = \{c_1,d_0,x,c_2\}$ spans $Y-d_1$, in which case $r(Y) \le 5$, so $\lambda(Y) \le 1$, and $|E(M)| \le 10$; a contradiction.
    So we may also assume that $C_1 = \{c_1,x,d_2,f_2\}$.


    Let $Y' = Y-\{d_1,d_2\}$.
    Now $
    \{c_1,c_2,d_0,x,f_2\}$ spans $Y'$, so $r(Y') \le 5$; and $r^*_{M \ba d_2}(Y')= r^*(Y' \cup d_2)-1 = 4$, so $\lambda_{M \ba d_2}(Y') = 2$.
    Moreover, $d_1 \in \cocl_{M \ba d_2}(Y')$.
    We work towards an application of \cref{3sepwin2}.
    By circuit elimination on $C_1$ and $C_5$, there is a circuit contained in $\{c_1,x,c_2,f_1,d_2\}$.  But this set intersects $C_4^*$ in a single element, $d_2$, so $C_6=\{c_1,x,c_2,f_1\}$ is a circuit.
    Recall that $M \ba d_1 \ba d_2 \ba e / c_1$ has an $N$-minor, where $\{c_2,f_1\}$ and $\{d_0,f_2\}$ are series pairs in this matroid, so $M \ba d_1 \ba d_2 \ba e / c_1 /c_2 /f_2$ has an $N$-minor.
    Due to the circuit~$C_6$, the pair $\{x,f_1\}$ is a parallel pair in this matroid.
    It follows that \cref{wtd2i2} holds by \cref{3sepwin2}.
  \end{slproof}

  \begin{sublemma}
    \label{qqcc4}
    Suppose $|C_1^* \cap C_2^*| = 2$.
    Then the lemma holds.
  \end{sublemma}
  \begin{slproof}
    If $d_1 \notin C_2^*$, then the lemma holds by \cref{qqcc3}.
    So we may assume that $d_1 \in C_2^*$.
    Let $C_1^* = \{d_0,d_1,c_1,x\}$ and $C_2^* = \{d_0,d_1,d_2,c_2\}$, and recall that $X = C_1^* \cup C_2^*$.
    As $M \ba d_0 \ba d_1 \ba d_2$ has an $N$-minor, it follows that there is a $4$-element circuit~$C_0$ containing $\{c_1,c_2\}$, and there is a $4$-element circuit~$C_1$ containing $\{x,c_2\}$.
    If $\{c_1,c_2,x,d_2\}$ is a circuit, then the lemma holds by \cref{qqcc25}.
    If $\{c_1,c_2,x,d_1\}$ is a circuit, then we can swap the $N$-labels on $x$ and $d_1$, and with $x \in D \cap (C_1^*-C_2^*)$ playing the role of $d_1$, the lemma holds by \cref{qqcc3}.
    By symmetry, the lemma also holds if $\{c_1,c_2,x,d_0\}$ is a circuit.
    Now, if $C_0$ and $C_1$ are both contained in $X$, then $C_0 \neq C_1$, so $r(X) \le 4$, implying $\lambda(X)=2$, and \cref{wtd2i1} holds by \cref{qqcc2}.
    Without loss of generality, $C_0$ is not contained in $X$.
    By orthogonality and symmetry, we may assume that $C_0=\{c_1,c_2,d_1,e\}$ for some $e \in E(M)-X$.

    Now $M \ba d_0 \ba d_2 / c_1 /c_2$ has an $N$-minor, and $\{d_1,e\}$ is a parallel pair in this matroid, so by swapping $N$-labels we may assume that $\{d_0,e,d_2\} \subseteq D$.
    In particular, there exists a $4$-element cocircuit $C_3^*$ containing $\{d_2,e\}$, and a $4$-element cocircuit $C_4^*$ containing $\{d_0,e\}$.
    By orthogonality, each of $C_3^*$ and $C_4^*$ meets $\{c_1,d_1,c_2\}$.
    We claim that $c_1 \in C_3^* \cap C_4^*$.
    First suppose $d_1 \in C_3^*$.
    We may assume that $C_3^* \nsubseteq C_2^* \cup e$, for otherwise $C_2^* \cup C_3^* = C_2^* \cup e$ is a corank-$3$ set, contradicting \cref{qqcc1}.
    So $C_3^* = \{d_2,e,d_1,f\}$ for some $f \in E(M)-(C_2^* \cup e)$.
    Now $|C_2^* \cap C_3^*| = 2$, with $d_2 \in C_2^* \cap C_3^*$, $d_0 \in C_2^* - C_3^*$, and $e \in C_3^* - C_2^*$, so we can apply \cref{qqcc3}, with $C_3^*$ in the role of $C_1^*$, in which case the lemma holds.
    Similarly, if $c_2 \in C_3^*$, then $C_3^* \nsubseteq C_2^* \cup e$, by \cref{qqcc1}, so $C_3^* = \{d_2,e,c_2,f\}$ for some $f \in E(M)-(C_2^* \cup e)$.
    Then, with $C_3^*$ in the role of $C_1^*$, the lemma holds by \cref{qqcc3}.
    So $c_1 \in C_3^*$.
    By a similar argument, $c_1 \in C_4^*$.

    Now let $C_3^* = \{d_2,e,c_1,f\}$ and $C_4^* = \{d_0,e,c_1,g\}$, for some $f \in E(M)- \{d_2,e,c_1\}$ and $g \in E(M)-\{d_0,e,c_1\}$.
    Note that $|C_3^* \cap C_4^*| \neq 3$, by \cref{qqcc1}.
    So $C_3^* \cap C_4^* = \{e,c_1\}$, with $d_2 \in C_3^* - C_4^*$ and $d_0 \in C_4^* - C_3^*$.
    Using \cref{qqcc3}, with $C_3^*$ and $C_4^*$ in the roles of $C_1^*$ and $C_2^*$, the lemma holds.
  \end{slproof}

  Now, by \cref{qqcc1,qqcc4}, we may assume that for any triple of elements $d,d',d'' \in D$, the $4$-element cocircuits containing $\{d,d'\}$, $\{d,d''\}$ and $\{d',d''\}$ pairwise intersect in a single element.
  Dually, for a triple of elements in $C$, two $4$-element circuits each containing a pair of this triple meet in a single element.

  Let $C_0^*$, $C_1^*$ and $C_2^*$ be the $4$-element cocircuits containing $\{d_1,d_2\}$, $\{d_0,d_2\}$, and $\{d_0,d_1\}$ respectively, where $\{d_0,d_1,d_2\} \subseteq D$. Then $|C_i^* \cap C_j^*| = 1$ for distinct $i,j \in \{0,1,2\}$.
  Let $\{c_i,e_i\} \subseteq C_i^*$ for each $i \in \{0,1,2\}$, where $c_0,c_1,c_2,e_0,e_1,e_2 \in E(M)-\{d_0,d_1,d_2\}$ are distinct.
  Since $M \ba d_0 \ba d_1 \ba d_2$ has an $N$-minor, $M /c_0/c_1/c_2$ has an $N$-minor, up to switching the $N$-labels on $c_i$ and $e_i$ for each $i$.
  Thus, for distinct $i,j \in \{0,1,2\}$, the pair $\{c_i,c_j\}$ is contained in a $4$-element circuit, and these circuits meet in a single element.
  Let $X = C_0^* \cup C_1^* \cup C_2^*$.

  Suppose there is such a circuit that is not contained in $X$.
  Without loss of generality, let $\{c_1,c_2,g\}$ be contained in a $4$-element circuit~$C_0$, for $g \in E(M)-X$.
  By orthogonality, $C_0 = \{c_1,c_2,d_0,g\}$.
  It now follows that $\{g,d_1\}$ and $\{g,d_2\}$ are $N$-deletable, so they are contained in $4$-element cocircuits $C_3^*$ and $C_4^*$ respectively.
  Since $|C_3^* \cap C_2^*|=1$, and by orthogonality, $c_1 \in C_3^*$.
  Similarly, $c_2 \in C_4^*$.
  Now $C_3^*$ and $C_4^*$ intersect each of $C_0^*$, $C_1^*$, and $C_2^*$ in a single element, so $C_3^* = \{g,d_1,c_1,h_1\}$ and $C_4^* = \{g,d_2,c_2,h_2\}$ for some distinct $h_1,h_2 \in E(M)-(X \cup g)$.

  Now consider the $4$-element circuit~$C_1$ containing $\{c_0,c_1\}$. 
  By orthogonality with $C_1^*$, this circuit meets $\{d_0,d_2,e_1\}$.
  If $d_0 \in C_1$, then, by orthogonality with $C_0^*$ and $C_2^*$, we have $d_1 \in C_1$.
  If $d_2 \in C_1$, then, by orthogonality with $C_3^*$ and $C_4^*$, we have $g \in C_1$.
  If $e_1 \in C_1$, then, by orthogonality with $C_0^*$ and $C_3^*$, we have $d_1 \in C_1$, but then $|C_1 \cap C_2^*| = 1$; a contradiction.
  So $C_1 = \{c_0,c_1,d_0,d_1\}$ or $C_1 = \{c_0,c_1,d_2,g\}$,
  but in either case, $|C_0 \cap C_1| = 2$; a contradiction.

  So, for distinct $i,j \in \{0,1,2\}$, the $4$-element circuit~$C_{i,j}$ containing $\{c_i,c_j\}$ is contained in $X$.
  Let $\{i,j,k\} = \{0,1,2\}$. 
  Now, $C_{i,j}$ cannot meet $\{c_k,e_k\}$, by orthogonality, as there is no element common to $C_i^*$, $C_j^*$ and $C_k^*$.
  Then, by orthogonality with $C_k^*$, either $C_{i,j}$ contains $\{d_i,d_j\}$, or it does not meet this pair, for any $\{i,j,k\} = \{0,1,2\}$.
  But now if $C_{i,j}$ meets $\{d_0,d_1,d_2\}$, then neither $C_{i,k}$ nor $C_{j,k}$ meets $\{d_0,d_1,d_2\}$, since the circuits pairwise intersect in a single element,
  in which case $C_{i,k}$ and $C_{j,k}$ intersect in the pair $\{c_k,e_k\}$; a contradiction.
  So $C_{i,j}=\{c_i,e_i,c_j,e_j\}$ for all distinct $i,j \in \{0,1,2\}$, which again contradicts that these circuits pairwise intersect in one element.
  This completes the proof.
\end{proof}

\begin{proof}[Proof of \cref{weaktheoremdetailed}]
  If, up to switching $N$-labels, there exists a pair $\{c_1,c_2\} \subseteq C$ or $\{d_1,d_2\} \subseteq D$ that is contained in a quad, then the \lcnamecref{weaktheoremdetailed} holds by \cref{wtdp1}.
  Otherwise, the \lcnamecref{weaktheoremdetailed} holds by \cref{weaktheoremdetailed2}.
\end{proof}

\section{\Psep s}
\label{secpseps}

Now we show that when $M$ has a \psep~$P$, most of the elements that are $N$-labelled for removal must be in $P$, otherwise $M$ has an $N$-detachable pair.

As a warm-up, we first consider the case where $M$ has an \planespider~$Q \cup z$. 
We will later prove a similar result, \cref{problematiclemma}, that handles the case where $M$ has one of the other \psep s.

\begin{lemma}
  \label{handlequads}
  Let $M$ be a $3$-connected matroid, and let $N$ be a $3$-connected minor of $M$ such that $|E(N)| \ge 4$, and every triangle or triad of $M$ is \unfortunate. 
  Suppose that $(Q,\{z\},S)$ is a cyclic $3$-separation of $M$ such that
  $Q \cup z$ is an \planespider, 
  $M \ba z$ has an $N$-minor with $|Q \cap E(N)| \le 1$,
  and $M \ba s$ is not $3$-connected for each $s \in S$ that is $N$-deletable in $M \ba z$.
  If $M$ has no $N$-detachable pairs, then 
      $S \subseteq E(N)$.
\end{lemma}
\begin{proof}
  Since $M \ba z$ has an $N$-minor with $|Q \cap E(N)| \le 1$, it follows from \cref{m2.73} that, for any distinct $q,q' \in Q$, the matroid $M \ba z \ba q / q'$ has an $N$-minor.
  In particular, each element of $Q$ is $N$-flexible in $M \ba z$.
  Since every triangle or triad of $M$ is \unfortunate, observe that no triangle or triad meets $Q$.
  Let $(C,D)$ be an $N$-labelling such that $z \in D$ and at most one element of $Q$ is not $N$-labelled for removal.
  Let $Q= \{q_1,q_2,q_3,q_4\}$.

  First we claim that if $e \in S$ is $N$-labelled for contraction, then $M/e$ is not $3$-connected.
  Towards a contradiction, suppose that $M/e$ is $3$-connected for some $e \in S$ that is $N$-labelled for contraction.
  If $e \in \cl(Q)$, then $\si(M/e)$ is not $3$-connected, so we may assume that $e \notin \cl(Q)$.
  As $Q$ is a quad in $M/e$, the matroid $\si(M/e/q_i)$ is $3$-connected for each $i \in [4]$, by \cref{r3cocircsi3}.
  It follows from \cref{m2.73} that any such pair $\{e,q_i\}$ is $N$-contractible in $M \ba z$.
  Thus, as $M$ has no $N$-detachable pairs, for each $i \in [4]$ there is a $4$-element circuit containing $\{e,q_i\}$. 
  By orthogonality, and since $e \notin \cl(Q)$, each of these circuits intersects $Q$ in precisely two elements.
%
  We may assume that $\{q_1,q_2,e,h\}$ is a circuit, for some $h \in E(M)-(Q \cup e)$.
  Now, for $i \in \{1,2\}$ and $j \in \{3,4\}$, the matroid $M \ba z \ba q_j / q_i / e$ has an $N$-minor.
  Note that if $h=z$, then $(Q \cup z, e, S-e)$ is a vertical $3$-separation, 
  so $\si(M/e)$ is not $3$-connected; a contradiction.
  Thus $h$ is $N$-deletable in $M \ba z$.
  By hypothesis, $M \ba h$ is not $3$-connected. 
  Since each triad of $M$ is \unfortunate, $h$ is not in a triad, so $\co(M \ba h)$ is not $3$-connected, and $M$ has a cyclic $3$-separation $(U,\{h\},V)$ 
  with
  $|U \cap Q| \ge 2$.
  Then, by uncrossing, $(U \cup Q, (V-Q) \cup h)$ is $3$-separating.
  If $e \in U$, then $h \in \cl(U \cup Q) \cap \cocl(U)$, so $\lambda(U \cup Q \cup h) \le 1$.
  Hence $|V-Q| \le 1$, so $|V|= 3$.  But then $V$ is a triangle containing an $N$-contractible element; a contradiction.
  So $e \in V$.
  Since $e$ is $N$-contractible, it follows that $|V| \ge 4$.
  Now $U \cup Q$ and $U \cup (Q \cup e)$ are exactly $3$-separating, so
  $(U \cup Q, \{e\}, (V \cup h)-(Q\cup e))$ is a path of $3$-separations where $e$ is a guts element.
  As $e$ is not in a triangle, $r((V \cup h)-(Q\cup e)) \ge 3$, so the path of $3$-separations is a vertical $3$-separation, implying $\si(M/e)$ is not $3$-connected; a contradiction.

  Now, if $e \in S$ is $N$-labelled for removal, then either $M/e$ is not $3$-connected and $e$ is $N$-labelled for contraction, or $M \ba e$ is not $3$-connected and $e$ is $N$-labelled for deletion.

  Suppose there is some $e\in S-\cl(Q)$ that is $N$-labelled for contraction.
  By the foregoing, $M/e$ is not $3$-connected.  Moreover, $e$ is not in a triangle, so $\si(M/e)$ is not $3$-connected, and hence $M$ has a vertical $3$-separation $(U,\{e\},V)$ with $|U \cap E(N)| \le 1$.  We may assume, without loss of generality, that $V \cup e$ is closed.
  By the dual of \cref{doublylabel3}, at most one element in $U$ is not $N$-flexible in $M/e$, and if such an element~$x$ exists, then $x \in U \cap \cocl(V)$ and $e \in \cl(U - x)$.
  But if $u \in S \cup z$ is $N$-flexible, then, as either $\si(M/u)$ or $\co(M\ba u)$ is $3$-connected by Bixby's Lemma, and $u$ is not in an \unfortunate\ triangle or triad, either $M/u$ or $M \ba u$ is $3$-connected; a contradiction.
%
  So every $N$-flexible element is contained in $Q$, implying $U \subseteq Q$ or $U-x \subseteq Q$ for some $x \in U$ such that $e \in \cl(U-x)$.
  In either case, $e \in \cl(Q)$; a contradiction.

  By a similar argument, if $e \in S$ is $N$-labelled for deletion, then $e \in \cocl(Q)$.
  Suppose there is some $e \in S \cap \cocl(Q)$.
  Then $(Q \cup e, z, S-e)$ is a cyclic $3$-separation, and $e$ is $N$-contractible, by \cref{doublylabel3}(ii), since $e \in \cocl(Q)$.
  But $(Q, e, (S-e) \cup z)$ is also a cyclic $3$-separation, and it follows that $M/e$ is $3$-connected; a contradiction.
  Similarly, if there is an element $s \in S \cap \cl(Q)$, then $(Q \cup s, z, S-s)$ is a cyclic $3$-separation, so $M \ba z$ is $3$-connected, and $s$ is $N$-deletable in $M \ba z$ by \cref{doublylabel3}; a contradiction. 
  We deduce that no elements of $S$ are $N$-labelled for removal, thus completing the proof.
\end{proof}

We now shift our attention to \psep s other than the quad.
The next lemma describes how triangles (or, dually, triads) can meet such a \psep.

\begin{lemma}
  \label{probcondel}
  Let $M$ be a $3$-connected matroid with $P \subseteq E(M)$.
  Suppose that $P$ contains no triangles, and $P$ has a partition $(L_1,L_2,\dotsc,L_t)$ into pairs, for some $t \ge 3$, such that $L_i \cup L_j$ is a cocircuit for all distinct $i,j \in [t]$ except perhaps $\{i,j\} = \{1,2\}$.
  If a triangle $T$ meets $L_i$, for some $i \in [t]$, then $L_i \subseteq T$.
\end{lemma}
\begin{proof}
  Suppose $T$ is a triangle such that $|T \cap L_i|=1$ for some $i \in [t]$.
  If $i \notin \{1,2\}$, then, by orthogonality, $T$ meets $L_j$ for each $j \in [t]-i$.
  On the other hand, if $i \in \{1,2\}$, then $T$ meets $L_t$.  If $T \nsubseteq L_i \cup L_t$, then $|T \cap L_t|=1$ and, in turn, $T$ meets $L_j$ for each $j \in [t-1]$.
  Since $t \ge 3$, we deduce in either case that the triangle $T$ is contained in $P$; a contradiction.
\end{proof}

The next lemma guarantees that when a single element in a \psep\ is removed, $3$-connectivity is preserved up to series or parallel classes.
We require the following in the proof of this lemma.

A set $X$ in a matroid $M$ is {\em fully closed} if it is closed and coclosed; that is, $\cl(X)=X=\cl^*(X)$.
The {\em full closure} of a set $X$, denoted $\fcl(X)$, is the intersection of all fully closed sets that contain $X$.
We say that a $2$-separation $(U,V)$ is \emph{trivial} if $U$ or $V$ is contained in a series or parallel class.
It is easily seen that if $(U,V)$ is a non-trivial $2$-separation of a connected matroid $M$, then $(\fcl(U),V-\fcl(U))$ is also a $2$-separation of $M$.

\begin{lemma}
  \label{probconnmain}
  Let $M$ be a $3$-connected matroid, and 
  let $P$ be either 
  \begin{enumerate}[label=\rm(\alph*)]
    \item a \spikelike\ or a \spider; or 
    \item an \pspider, a \twisted, or a \tvamoslike\ of $M$ or $M^*$, with $|E(M)-P| \ge 3$.
  \end{enumerate}
  Then $\si(M/p)$ and $\co(M \ba p)$ are $3$-connected for each $p \in P$.
\end{lemma}
\begin{proof}
  In case (a), every element of $P$ is in a quad, and the result follows from \cref{r3cocircsi3} and its dual.
  For (b), we consider each of the \psep s in turn.

  \begin{sublemma}
    \label{connpspider}
    The lemma holds when $P$ is a \twisted\ or an \pspider.
  \end{sublemma}
  \begin{slproof}
%
%
%
  Suppose that $P$ is a \twisted.
    By symmetry and duality, it suffices to show that $\co(M\ba p)$ is $3$-connected for some $p \in P$.
    Let $P = \{u_1,u_2,c,d,v_1,v_2\}$ where
    $\{u_1,u_2,d,v_1\}$, $\{u_1,c,d,v_2\}$, and $\{u_2,c,v_1,v_2\}$ are circuits,
    and $\{u_1,u_2,c,d\}$, $\{u_1,u_2,v_1,v_2\}$, and $\{c,d,v_1,v_2\}$ are cocircuits.

    Towards a contradiction, suppose that $M \ba d$ has a non-trivial $2$-separation $(U,V)$.
    Without loss of generality, $|\{u_1,u_2,c\} \cap U| \ge 2$, and $U$ is fully closed.
    So $\{u_1,u_2,c\} \subseteq U$.
    If $U$ meets $\{v_1,v_2\}$, then $d \in \cl(U)$, so $(U \cup d,V)$ is a $2$-separation of $M$; a contradiction.
    So $\{v_1,v_2\} \subseteq V$.
    Now $c \in \cocl_{M \ba d}(V)$, so $(U-c,V \cup c)$ is also a $2$-separation of $M\ba d$.
    Recall that $|E(M)-P| \ge 3$.
    If $|V-P| \le 1$, then $|U-P| \ge 2$, so we may assume, up to relabelling, that $|V-P| \ge 2$.

    Let $C_1$ be the circuit $\{u_2,c,v_1,v_2\}$.
    Observe that $\lambda_{M \ba d}(C_1) = 2$, since $d \in \cocl(C_1)$.
    By submodularity of the connectivity function,
    \begin{align*}
      \lambda_{M\ba d}(U \cup C_1) &\le \lambda_{M\ba d}(U) + \lambda_{M\ba d}(C_1) - \lambda_{M\ba d}(U \cap C_1) \\
      &= 1 + 2 - 2 = 1,
    \end{align*}
    so $(U \cup \{v_1,v_2\}, V-\{v_1,v_2\})$ is a $2$-separation in $M \ba d$.
    But $c \in \cocl_{M \ba d}(U \cup \{v_1,v_2\})$, so $(U\cup \{v_1,v_2,c\},V-\{v_1,v_2\})$ is also a $2$-separation of $M$; a contradiction.
    So $M \ba d$ has no non-trivial $2$-separations, implying $\co(M \ba d)$ is $3$-connected, as required.

  A similar argument applies in the case that $P$ is an \pspider; we omit the details.
  \end{slproof}

  \begin{sublemma}
    The lemma holds when $P$ is a \tvamoslike\ of $M$ or $M^*$.
  \end{sublemma}
  \begin{slproof}
    It suffices to show that both $\si(M/p)$ and $\co(M\ba p)$ are $3$-connected for each $p \in P$ when $P$ is a \tvamoslike\ of $M$.
    Let $P = \{s_1,s_2,t_1,t_2,q_1,q_2\}$ where $\{s_1,s_2,q_1,q_2\}$, $\{t_1,t_2,q_1,q_2\}$, and $\{s_1,s_2,t_1,t_2\}$ are circuits, and $\{s_1,t_1,q_1,q_2\}$ and $\{s_2,t_2,q_1,q_2\}$ are the non-spanning cocircuits contained in $P$.
    A similar approach works here as in \cref{connpspider}; we outline the proof, omitting some of the details.

    Suppose $M \ba p$ has a non-trivial $2$-separation $(U,V)$, for some $p \in \{q_1,q_2\}$.
    By symmetry, we may assume $p = q_2$.
    Then $\{s_1,t_1,u\} \subseteq U$ and $\{s_2,t_2\} \subseteq V$, and $|V-P| \ge 2$. 
    By an uncrossing argument with the set $\{s_1,s_2,t_1,t_2\}$, which is $3$-separating in $M \ba q_2$, we deduce that $(U \cup \{s_2,t_2\},V-\{s_2,t_2\})$ is a $2$-separation in $M \ba q_2$, with $q_2 \in \cl_M(U \cup \{s_2,t_2\})$; a contradiction.
    So $\co(M \ba p)$ is $3$-connected for $p \in \{q_1,q_2\}$.

    Suppose $M / p$ has a non-trivial $2$-separation $(U,V)$, for some $p \in \{s_1,s_2,t_1,t_2\}$.
    By symmetry, we may assume $p = t_2$.
    Then, we may assume that $|\{q_1,q_2,t_1\} \cap U| \ge 2$ and $U$ is fully closed, but $P-t_2 \subseteq \fcl_{M/t_2}(\{q_1,q_2,t_1\})$, so $P-t_2 \subseteq U$, and thus $t_2 \in \cocl_M(U)$; a contradiction.
    So $\si(M / p)$ is $3$-connected for $p \in \{s_1,s_2,t_1,t_2\}$.

    Suppose $M / p$ has a non-trivial $2$-separation $(U,V)$, for some $p \in \{q_1,q_2\}$.
    By symmetry, we may assume $p = q_2$.
    Then, we may assume that $\{s_1,s_2,q_1\} \subseteq U$ and $\{t_1,t_2\} \subseteq V$, and $|V-P| \ge 2$. 
    By an uncrossing argument with $P-q_2$, which is $3$-separating in $M/q_2$,  we deduce that $(U \cup \{t_1,t_2\},V-\{t_1,t_2\})$ is a $2$-separation in $M/q_2$, with $q_2 \in \cocl(U \cup \{t_1,t_2\})$; a contradiction.
    So $\si(M / p)$ is $3$-connected for $p \in \{q_1,q_2\}$.

    Suppose $M \ba p$ has a non-trivial $2$-separation $(U,V)$, for some $p \in \{s_1,s_2,t_1,t_2\}$.
    By symmetry, we may assume $p = t_2$.
    Then, we may assume that $\{s_2,q_1,q_2\} \subseteq U$ and $\{s_1,t_1\} \subseteq V$. 
    The set $P-t_2$ is $3$-separating in $M/t_2$.
    By an uncrossing argument with $U$, when $|V-P| \ge 2$, or with $V$, when $|U-P| \ge 2$, we obtain a $2$-separation of $M/t_2$ where $t_2$ is in the closure of one side; a contradiction.
    So $\si(M \ba p)$ is $3$-connected for $p \in \{s_1,s_2,t_1,t_2\}$.
  \end{slproof}
\end{proof}

Recall that when $P$ is a \psep\ with $|E(M)-P| \ge 3$, and $z \in \cl(P)-P$ or $z \in \cocl(P)-P$, we say $P \cup z$ is an \emph{\auging\ of $P$}.
The next lemma shows that when a matroid~$M$ with no $N$-detachable pairs has a \psep~$P$, either $E(M)-E(N) \subseteq P$, or there is an \auging~$P \cup z$ of $P$ such that $E(M)-E(N) \subseteq P \cup z$.
The proof follows a similar approach to \cref{handlequads}, but there are some extra subtleties to handle.

Let $N$ be a $3$-connected minor of $M$.
An \augsep~$P \cup z$ is \emph{\prob\ with respect to $N$} if
\begin{enumerate}[label=\rm(\alph*)]
  \item $z \in \cl(P)$ and $z$ is $N$-contractible but not $N$-deletable, or
      $z \in \cocl(P)$ and $z$ is $N$-deletable but not $N$-contractible; and
  \item $E(M)-E(N) \subseteq P \cup z$.
\end{enumerate}

\begin{lemma}
  \label{problematiclemma}
  Let $M$ be a $3$-connected matroid, and let $N$ be a $3$-connected minor of $M$ such that $|E(N)| \ge 4$, and every triangle or triad of $M$ is \unfortunate. 
  Suppose that there exists $d \in E(M)$ such that $M \ba d$ is $3$-connected and has a cyclic $3$-separation $(Y, \{d'\}, Z)$ where $M \ba d \ba d'$ has an $N$-minor with $|Y \cap E(N)| \le 1$, and there is a subset $X$ of $Y$ such that for some $c \in \cocl_{M \ba d}(X)-X$, one of the following holds:
  \begin{enumerate}[label=\rm(\alph*)]
    \item $X \cup \{c,d\}$ is contained in a \spikelike~$P$ of $M$, where $P$ is maximal subject to $|P \cap E(N)| \le 1$,
    \item $P = X \cup \{c,d\}$ is a \twisted\ of $M$,
    \item $P = X \cup \{c,d\}$ is a \tvamoslike\ of $M$ or $M^*$,
    \item $P = X \cup \{c,d\}$ is an \pspider\ of $M$, or 
    \item $P = X \cup \{a,b,c,d\}$ is a \spider\ of $M$ with associated partition $\{X,\{a,b,c,d\}\}$ for some distinct $a,b \in E(M) - (X \cup \{c,d\})$.
  \end{enumerate}
  If $M$ has no $N$-detachable pairs, then either 
  \begin{enumerate}[label=\rm(\roman*)]
    \item $|E(M)| \le 10$ or $M$ is a \quadflower,
    \item $E(M)-E(N) \subseteq P$, or
    \item there is some $z \in E(M)-P$ such that $P \cup z$ is an \auging\ of $P$ that is problematic with respect to $N$.
  \end{enumerate}
\end{lemma}
\begin{proof}
  The set $P$ is a \psep\ of $M$ that is either a \spikelike, a \twisted, an \pspider, a \spider, or a \tvamoslike\ of $M$ or $M^*$.
  By definition, each of these \psep s has a partition $(L_1,L_2,\dotsc,L_t)$, for some $t \ge 3$, such that $|L_i|=2$ for each $i \in [t]$, and $L_i \cup L_j$ is a cocircuit for all distinct $i,j \in [t]$, except when $P$ is a \tvamoslike\ of $M$ in which case $L_i \cup L_j$ is a cocircuit for all distinct $i,j \in [t]$ except $\{i,j\}=\{1,2\}$.
  Dually, $P$ has a partition $(K_1,K_2,\dotsc,K_t)$ such that $|K_i|=2$ for each $i \in [t]$, and $K_i \cup K_j$ is a circuit for all distinct $i,j \in [t]$, except when $P$ is a \tvamoslike\ of $M^*$ in which case $K_i \cup K_j$ is a circuit for all distinct $i,j \in [t]$ except $\{i,j\}=\{1,2\}$.

  Let $(L_1,\dotsc,L_t)$ and $(K_1,\dotsc,K_t)$ be the partitions of $P$ as described.
  For the arguments that follow, we require a particularly convenient $N$-labelling, as described in \cref{preswitch}.  We start with the following:

  \begin{sublemma}
    \label{preswitching}
    There is an $N$-labelling such that $|P \cap E(N)| \le 1$ when $P$ is not a \spider, and $|P \cap E(N)| \le 2$ otherwise.
  \end{sublemma}
  \begin{slproof}
    Suppose $P$ is not a \spider.
    Clearly $|P \cap E(N)| \le 1$ when $c \in Y \cup d'$, since $|Y \cap E(N)| \le 1$ in $M \ba d \ba d'$. So assume that $c \in Z$.
    As $(Y,Z)$ is a $2$-separation of $M \ba d \ba d'$, and $c \in \cocl_{M \ba d}(X)$, it follows that $(Y \cup c, Z-c)$ is also a $2$-separation of $M \ba d \ba d'$.
    Note that $|(Z-c) \cap E(N)| >1$, since $|Y \cap E(N)| \le 1$ and $|E(N)| \ge 4$.
    Hence, by \cref{m2.73}, $|(Y \cup c) \cap E(N)| \le 1$ in $M \ba d$, and $|P \cap E(N)| \le 1$ as required.

    Now suppose $P$ is a \spider.
    If $d' \notin \{a,b,c\}$, then $(Y \cup c, Z-c)$ is a $2$-separation of $M \ba d \ba d'$, and, in the same manner, we deduce that $|(Y \cup c) \cap E(N)| \le 1$ in $M \ba d \ba d'$.
    Moreover, up to swapping $a$ and $b$, the partition $(Y, \{c\}, \{b\}, \{a\}, Z-\{a,b,c\})$ is a path of $2$-separations in $M \ba d \ba d'$ and, iteratively repeating this process, we deduce that $|(Y \cup \{a,b,c\}) \cap E(N)| \le 1$ in $M \ba d$, as required.
    On the other hand, if $d' \in \{a,b,c\}$, then $|X \cap E(N)| \le 1$ in $M \ba d \ba d'$, and $\{a,b,c\}-d'$ is a series pair in this matroid, so $|P \cap E(N)| \le 2$, as required.
  \end{slproof}

  \begin{sublemma}
    \label{preswitch}
    There is an $N$-labelling $(C,D)$ such that, up to $N$-label switches on elements in $P$, 
    \begin{enumerate}[label=\rm(\roman*)]
      \item for every $i \in [t]$ except perhaps $s \in [t]$, there is an element in $L_i$ that is $N$-labelled for contraction; and
      \item for every $i \in [t]$ except perhaps $s' \in [t]$, there is an element in $K_i$ that is $N$-labelled for deletion.
    \end{enumerate}
    Moreover, if such an $s$ exists, then $P$ is an \pspider\ or a \tvamoslike\ of $M$ or $M^*$, and $P-L_s$ is independent;
    while if such an $s'$ exists, then $P$ is an \pspider\ or a \tvamoslike\ of $M$ or $M^*$, and $P-K_{s'}$ is coindependent.
  \end{sublemma}
  \begin{slproof}
    Let $(C,D)$ be an $N$-labelling such that at most two elements of $P$ are $N$-labelled for removal.
    Suppose $(C,D)$ does not satisfy \cref{preswitch}(i).
    Then there exist distinct $i,s \in [t]$ such that $L_i \cup L_s$ does not contain an element that is $N$-labelled for contraction.
    Then $L_i \cup L_s$ contains at least two elements that are $N$-labelled for deletion.
    Apart from when $P$ is a \tvamoslike\ of $M$ and $\{i,s\}=\{1,2\}$, the set $L_i \cup L_s$ is a $4$-element cocircuit, so there is, up to an $N$-label switch, an element in $L_i \cup L_s$ that is $N$-labelled for contraction, as required.
    In the exceptional case, $|P \cap E(N)| \le 1$, so $L_1\cup L_2$ contains at least three elements that are $N$-labelled for deletion, and $L_1\cup L_2$ is contained in a $5$-element cocircuit; so, again, an element in $L_i \cup L_s$ is $N$-labelled for contraction after an $N$-label switch.
    This proves \cref{preswitch}(i).
    We obtain \cref{preswitch}(ii) by a dual argument.

    Suppose that $P$ is an \pspider\ or \tvamoslike\ of $M$, so $|P \cap E(N)| \le 1$.
    For ease of notation, let $s=t=3$; that is,
    for each $i \in \{1,2\}$ there is an element in $L_i$ that is $N$-labelled for contraction.
    Towards a contradiction, suppose that $P-L_3$ is a circuit (when $P$ is an \pspider, $P-L_3$ is a quad; whereas when $P$ is a \tvamoslike, $P-L_3$ is a coindependent circuit).
    As $|P \cap E(N)| \le 1$, there is an element $d_3 \in L_3$ that is $N$-labelled for deletion.
    Let $L_3 = \{d_3,x\}$.
    There are elements $c_1 \in L_1$ and $c_2 \in L_2$ that are $N$-labelled for contraction.
    Since $P-L_3$ is a circuit, up to an $N$-label switch we can locate an element $d_1 \in P-L_3$ that is $N$-labelled for deletion.
    Since $M\ba d_3\ba d_1$ has an $N$-minor, and $x$ is in a series pair in this matroid, we can perform an $N$-label switch so that $x \in L_3$ is $N$-labelled for contraction, as required.
    A similar argument applies in the dual.

    In the case that $P$ is a \spider\ or \spikelike, then $P-L_i$ contains a circuit for any $i \in [t]$, and a similar argument applies.
    Finally, if $P$ is a \twisted, then it is readily checked that starting from any $N$-labelling satisfying \cref{preswitch}, we can, by performing $N$-label switches, obtain an $N$-labelling where, for each $i \in [3]$, the pair $L_i$ contains an element that is $N$-labelled for contraction.
  \end{slproof}

  The next claim follows from \cref{preswitching}, \cref{probcondel} and its dual, and the fact that every triangle or triad of $M$ is \unfortunate.
  We omit the details.

  \begin{sublemma}
    \label{probtris}
    No element of $P$ is in a triangle or triad of $M$.
  \end{sublemma}

  Let $(C,D)$ be the $N$-labelling described in \cref{preswitch}. 
  Assume that $M$ has no $N$-detachable pairs.
  If $r(E(M)-P) \le 2$, then $Z$ contains a triangle, which, by \cref{probtris}, is disjoint from $P$.  So $|E(M)-P| \ge 3$.
  %
  \begin{sublemma}
    \label{con}
    Suppose that there exists some $e \in E(M)-\cl(P)$ that is $N$-labelled for contraction.
    If $M/e$ is $3$-connected, then, for each $i \in [t]-s$, there exists an element~$g_i$ such that $L_i \cup \{e,g_i\}$ is a circuit, where $g_i$ is $N$-deletable.
  \end{sublemma}
  \begin{slproof}
    Since $e \notin \cl(P)$, we have that $P$ is a \psep\ of $M/e$.
    For each $i \in [t]-s$, as $L_i$ contains an element~$x$ that is $N$-labelled for contraction, \cref{probconnmain} implies that either $M/e/x$ is $3$-connected, in which case $\{e,x\}$ is an $N$-detachable pair, or $x$ is in a triangle of $M/e$.  As we are under the assumption that $M$ has no $N$-detachable pairs, $L_i$ is contained in a triangle of $M/e$ for each $i \in [t]-s$, by \cref{probcondel}. 
    As each $L_i$ is not contained in a triangle in $M$, by \cref{probtris}, there exists an element~$g_i$ such that $L_i \cup \{e,g_i\}$ is a circuit of $M$.
    Now, as $M/e/x$ has an $N$-minor and $g_i$ is in a parallel pair in this matroid, $g_i$ is $N$-deletable. 
  \end{slproof}

  We obtain the following by a dual argument:

  \begin{sublemma}
    \label{del}
    Suppose that there exists some $e \in E(M)-\cocl(P)$ that is $N$-labelled for deletion.
    If $M\ba e$ is $3$-connected, then, for each $i \in [t]-s'$, there exists an element~$h_i$ such that $K_i \cup \{e,h_i\}$ is a cocircuit, where $h_i$ is $N$-labelled for contraction up to an $N$-label switch.
  \end{sublemma}

  Next we prove the following:
  \begin{sublemma}
    \label{conndcon}
    If there exists an element $e \in E(M)-P$ that is $N$-labelled for contraction, then either $\si(M/e)$ is not $3$-connected,
    $|E(M)| \le 10$, or $M$ is a \quadflower.
  \end{sublemma}
  \begin{slproof}
    Suppose there exists an element $e \in \cl(P)-P$ that is $N$-labelled for contraction.
    Let $Q = E(M)-(P \cup e)$, and observe that $(P, \{e\}, Q)$ is a path of $3$-separations with $|Q| \ge 2$.
    If $r(Q) \le 2$, then $e$ is an $N$-contractible element in a triangle; a contradiction.
    So $r(Q) \ge 3$, in which case the path of $3$-separations is a vertical $3$-separation, implying $\si(M/e)$ is not $3$-connected, as required.

    Suppose $e \in E(M)-\cl(P)$ and $\si(M/e)$ is $3$-connected.
    As $e$ is $N$-contractible, it is not in an \unfortunate\ triangle, so $M/e$ is $3$-connected.
    By \cref{con}, for each $i \in [t]-s$ there exists an element $g_i$ such that $L_i \cup \{e,g_i\}$ is a circuit, and $g_i$ is $N$-labelled for deletion up to an $N$-label switch.

    Suppose $\co(M \ba g_i)$ is not $3$-connected for some $i \in [t]-s$.
    We claim that there is a cyclic $3$-separation $(U,\{g_i\},V)$ of $M$ such that $L_i \subseteq U$. 
    There certainly exists some cyclic $3$-separation $(U',\{g_i\},V')$ of $M$.
    Suppose $L_i$ meets both $U'$ and $V'$; let $L_i = \{u,v\}$ with $u \in U'$ and $v \in V'$.
    If $|P \cap U'| = 1$, then $u \in \cl(V')$, so $(V'\cup u,\{g_i\},U' -u)$ is a cyclic $3$-separation of $M$ with $L_i \subseteq V' \cup u$.
    Similarly, if $|P \cap V'| = 1$, then $(U' \cup v,\{g_i\},V'-v)$ is a cyclic $3$-separation of $M$ with $L_i \subseteq U' \cup v$.
    Now assume that $|P \cap U'| \ge 2$ and $|P \cap V'| \ge 2$.
    As $|E(M)-P| \ge 3$, we may assume, without loss of generality, that $|V'-P| \ge 2$.
    By uncrossing, $(U'\cup P, \{g_i\}, V'-P)$ is a path of $3$-separations.
    If $r^*(V'-P) \le 2$, then $g_i$ is in a triad; a contradiction.
    We deduce that there exists a cyclic $3$-separation $(U, \{g_i\}, V)$ with $L_i \subseteq U$, as claimed.

    Now, if $e \in U$, then $g_i \in \cl(U)$, so $(U,V)$ is a $2$-separation of $M$; a contradiction.
    So $e \in V$. 
    As $e \in \cl(U \cup g_i)$ and $|V| \ge 3$, the sets $U \cup g_i$ and $U \cup \{g_i,e\}$ are exact $3$-separations.
    In particular, $e \in \cl(V-e)$, so $r(V-e)>2$, since $e$ is not in a triangle.
    So $(U\cup g_i, \{e\}, V-e)$ is a vertical $3$-separation, and hence
    $\si(M/e)$ is not $3$-connected; a contradiction.

    We deduce that $\co(M \ba g_i)$ is $3$-connected for each $i \in [t]-s$.
    As each $g_i$ is $N$-deletable, it is not contained in an \unfortunate\ triad, so $M \ba g_i$ is $3$-connected.
    We claim that there exists some $j \in [t]-s'$ such that $L_i \cap K_j=\emptyset$.
    The claim follows immediately from \cref{preswitch}, except when $P$ is an \pspider.
    In such a case, suppose, without loss of generality, that $L_1 \cup L_2 = K_1 \cup K_2$ is a quad; then $s,s' \in \{1,2\}$, since $P-L_s$ is independent and $P - K_{s'}$ is coindependent.
    Now $L_3=K_3$ is disjoint from $K_j$ for $j \in \{1,2\}-s'$, or $L_i$ for $i \in \{1,2\}-s$, satisfying the claim.
    After choosing such a $j$, \cref{del} implies that there is an element $h_j$ such that $K_j \cup \{g_i,h_j\}$ is a cocircuit.
    Since the circuit $L_i \cup \{e,g_i\}$ meets the cocircuit $K_j \cup \{g_i,h_j\}$, it follows, by orthogonality, that $e=h_j$.
    Then $g_i \in \cl(P \cup e) \cap \cocl(P \cup e)$ for each $i \in [t]-s$. 
    If, say, $s \notin \{1,2\}$ and $g_1 \neq g_2$, then $\lambda(P \cup \{e,g_1,g_2\}) \le 1$, so $|E(M)-P| \le 4$.
    Since the $N$-deletable element $g_i$ is not in a triad, and the $N$-contractible element $e$ is not in a triangle, $|E(M)-P| = 4$, and $E(M)-P$ is a quad.
    Thus, if $P$ is not a \spider\ or a \spikelike, then $|E(M)| \le 10$.
    In the case that $P$ is a \spider, $|E(M)| = 12$, and it is readily checked that $M$ is a \quadflower.
    In the case that $P$ is a \spikelike, then $L_i=K_i$ for all $i \in [t]$, the set $L_1 \cup \{e,g_1\}$ is a circuit, and $K_3 \cup \{e,g_1\}$ and $K_3 \cup \{e,g_2\}$ are cocircuits.
    But the latter two cocircuits imply $L_3 \cup \{e,g_1,g_2\}$ is a $5$-element coplane that intersects the circuit $L_1 \cup \{e,g_1\}$ in two elements; a contradiction.

    Now let $g_i = g_j = g$ for all $i,j \in [t]-s$.
    Suppose $P$ is a \spikelike; then we may assume that $L_i = K_i$ for all $i \in [t]$, and there is no $s \in [t]$ for which $L_s$ has no element $N$-labelled for contraction.
    Then $L_i \cup \{e,g\}$ is a quad for all $i \in [t]$.
    Moreover, $e$ is $N$-labelled for contraction and, up to an $N$-label switch, $g$ is $N$-labelled for deletion.
    Therefore $P$ is not maximal; a contradiction.

    Now $P$ is not a \spikelike. 
    We claim that $P \cup \{e,g\}$ is spanned by a $(t+1)$-element set.
    If $P$ is an \pspider\ or a \tvamoslike\ of $M$, then $L_s \subseteq \cl(P-L_s)$, so by choosing an element from $L_i$ for each $i \in [t]-s$, together with $\{e,g\}$, we have such a set.
    Otherwise, $L_i \subseteq \cl(P-L_i)$ for any $i \in [t]$, so we take $\{e,g\}$ together with an element from $L_i$ for each $i \in [t-1]$.
    So $r(P \cup \{e,g\}) \le t+1$.
    As $e \notin \cl(P)$, it now follows that $r(P) \le t$; a contradiction.
  \end{slproof}

  \begin{sublemma}
    \label{nodoubly}
    No element in $E(M)-P$ is $N$-flexible.
  \end{sublemma}
  \begin{slproof}
    Suppose $x \in E(M)-P$ is $N$-flexible.
    Either $\si(M/x)$ or $\co(M\ba x)$ is $3$-connected, by Bixby's Lemma, which contradicts \cref{conndcon} or its dual.
  \end{slproof}

  \begin{sublemma}
    \label{confar}
    No element in $E(M)-\cl(P)$ is $N$-labelled for contraction.
  \end{sublemma}
  \begin{slproof}
    Suppose there is an element $e \in E(M)-\cl(P)$ that is $N$-labelled for contraction.
    By \cref{conndcon}, $\si(M/e)$ is not $3$-connected.
    So $M$ has a vertical $3$-separation $(U,\{e\},V)$, with $|U \cap E(N)| \le 1$, by \cref{m2.73}.
    Without loss of generality, $V \cup e$ is closed.
    By the dual of \cref{doublylabel3}, at most one element in $U$ is not $N$-flexible, and if such an element~$x$ exists, then $x \in U \cap \cocl(V)$ and $e \in \cl(U - x)$.
    By \cref{nodoubly}, we deduce that $U -x \subseteq P$ if such an $x$ exists, and $U \subseteq P$ otherwise.
    In either case, $e \in \cl(P)$; a contradiction.
  \end{slproof}

  \begin{sublemma}
    \label{conclose}
    There is at most one element in $\cl(P)-P$ that is $N$-labelled for contraction.
  \end{sublemma}
  \begin{slproof}
    Suppose $e,f \in \cl(P)-P$ are both $N$-labelled for contraction.
    If $r(E(M)-P) \ge 3$, then it follows that $(P \cup f, \{e\}, E(M)-(P \cup \{e,f\}))$ is a vertical $3$-separation.
    By the dual of \cref{doublylabel3}(ii), and as $f \in \cl(E(M)-(P \cup \{e,f\})$, $f$ is also $N$-deletable, so $f$ is $N$-flexible.
    But this contradicts \cref{nodoubly}.
    So $r(E(M)-P) \le 2$.
    Since $|E(M)-P| \ge 3$, the pair $\{e,f\}$ is contained in a triangle.  As $M/e$ has an $N$-minor, it follows that $f$ is $N$-deletable, and hence $N$-flexible, contradicting \cref{nodoubly}.
  \end{slproof}

  By duality, \cref{confar,conclose} also imply the following:
  \begin{sublemma}
    No element in $E(M)-\cocl(P)$ is $N$-labelled for deletion, and 
    there is at most one element in $\cocl(P)-P$ that is $N$-labelled for deletion.
  \end{sublemma}

  Finally, suppose that there exist distinct elements $g \in \cl(P)-P$ and $h \in \cocl(P)-P$ such that $g$ is $N$-labelled for contraction and $h$ is $N$-labelled for deletion.
  Now $(P \cup g, \{h\}, E(M)-\{g,h\})$ is a cyclic $3$-separation, with $g \notin \cocl(E(M)-g)$, so $g$ is $N$-deletable by \cref{doublylabel3}(i), contradicting \cref{nodoubly}.
  That completes the proof of the lemma.
\end{proof}

\section{Proof of the theorem}
\label{finalendgamesec}

We first recall the main result from the first two papers in the series \cite[Theorem~6.1]{paper2}.

\begin{theorem}
  \label{usefulprecursor}
  Let $M$ be a $3$-connected matroid and let $N$ be a $3$-connected minor of $M$ where $|E(N)| \ge 4$, and every triangle or triad of $M$ is \unfortunate.
  Suppose, for some $d \in E(M)$, that $M \ba d$ is $3$-connected and has a cyclic $3$-separation $(Y, \{d'\}, Z)$ with $|Y| \ge 4$, where $M \ba d \ba d'$ has an $N$-minor with $|Y \cap E(N)| \le 1$.
  Then either
  \begin{enumerate}
    \item $M$ has an $N$-detachable pair; or
    \item there is a subset $X$ of $Y$ such that for some $c \in \cocl_{M \ba d}(X)-X$, one of the following holds:
      \begin{enumerate}[label=\rm(\alph*)]
        \item $X \cup \{c,d\}$ is a \spikelike\ of $M$,
        \item $X \cup \{c,d\}$ is a \twisted\ of $M$,
        \item $X \cup \{c,d\}$ is a \tvamoslike\ of $M$ or $M^*$,
        \item $X \cup \{c,d\}$ is an \pspider\ of $M$, or 
        \item $X \cup \{a,b,c,d\}$ is a \spider\ of $M$ with associated partition $\{X,\{a,b,c,d\}\}$ for some distinct $a,b \in E(M) - (X \cup \{c,d\})$.
      \end{enumerate}
  \end{enumerate}
\end{theorem}

We now prove the main result.

\begin{theorem}
  \label{usefulone}
  Let $M$ be a $3$-connected matroid with $|E(M)| \ge 11$, and let $N$ be a $3$-connected minor of $M$ such that $|E(N)| \ge 4$, every triangle or triad of $M$ is \unfortunate, and $|E(M)|-|E(N)| \ge 5$.
  Then either
  \begin{enumerate}
    \item $M$ has an $N$-detachable pair;\label{pc1}
    \item there is some $P \subseteq E(M)$ such that $E(M)-E(N) \subseteq P$, and $P$ is 
      \begin{enumerate}[label=\rm(\alph*)]
        \item a \spikelike,
        \item a \twisted, 
        \item a \tvamoslike\ of $M$ or $M^*$, 
        \item an \pspider, 
        \item a \spider,
        \item a quad $3$-separator,
      \end{enumerate}
      or an \auging\ of one of these \psep s; or\label{pc2}
    \item $|E(M)| =12$, and $M$ is either a \quadflower\ or a \tcn.\label{pc3}
  \end{enumerate}
\end{theorem}
\begin{proof}
  By \cref{step}, if \cref{pc1} does not hold, then, up to replacing $(M,N)$ by $(M^*,N^*)$, either
  \begin{enumerate}[label=\rm(\Roman*)]
    \item there exists $d \in E(M)$ such that $M \ba d$ is $3$-connected and has a cyclic $3$-separation $(Y, \{d'\}, Z)$ with $|Y| \ge 4$, where $M \ba d \ba d'$ has an $N$-minor with $|Y \cap E(N)| \le 1$; or\label{uo1}
    \item $M$ has an \planespider~$Q \cup z$ such that $z \in \cocl(Q)$ and $M \ba z$ has an $N$-minor with $|Q \cap E(N)| \le 1$; or\label{uo3}
    \item there is an $N$-labelling $(C,D)$ of $M$ such that, for every switching-equivalent $N$-labelling $(C',D')$,
      \begin{enumerate}[label=\rm(\alph*)]
        \item $M / c$ and $M \ba d$ are $3$-connected for every $c \in C'$ and $d \in D'$,
        \item each pair $\{c_1,c_2\} \subseteq C'$ is contained in a $4$-element circuit, and
        \item each pair $\{d_1,d_2\} \subseteq D'$ is contained in a $4$-element cocircuit.
      \end{enumerate}\label{uo2}
  \end{enumerate}

  In case \cref{uo2}, if neither \cref{pc2} nor \cref{pc3} holds, then \cref{uo1} or \cref{uo3} holds by \cref{weaktheoremdetailed}.
  Suppose \cref{uo1} does not hold.  Then \cref{uo3} holds and, moreover, there is no element $e \in E(M)-(Q \cup z)$ such that $\{e,z\}$ is $N$-deletable and $M \ba e$ is $3$-connected, otherwise $(Q,\{z\},E(M)-(Q \cup z))$ is a cyclic $3$-separation satisfying \cref{uo1}.
  So, by an application of \cref{handlequads}, \cref{pc2} holds.

  Now we may assume that \cref{uo1} holds.
  By \cref{usefulprecursor}, either \cref{pc1} holds, or $Y$ contains a set $X$ such that, for some $c \in \cocl_{M \ba d}(X)-X$, either $X \cup \{c,d\}$ is a \psep~$P$ that is a \twisted, a \tvamoslike, or an \pspider, or $X \cup \{c,d\}$ is contained in a \spider~$P$, or a maximal \spikelike~$P$.
  Thus, we can apply \cref{problematiclemma}, and deduce that
  in the case that \cref{pc1} does not hold,
  either 
  $M$ is a \quadflower, so \cref{pc3} holds,
  or $E(M)-E(N) \subseteq P'$ where $P'= P$ or $P'$ is an \auging\ of $P$, so \cref{pc2} holds. 
\end{proof}

Using \cref{unfortunatetri}, we can relax the condition that every triangle or triad is \unfortunate\ if we allow at most one $\Delta$-$Y$ or $Y$-$\Delta$ exchange. 

\begin{corollary}
  \label{maincorollary}
  Let $M$ be a $3$-connected matroid with $|E(M)| \ge 11$, and let $N$ be a $3$-connected minor of $M$ such that $|E(N)| \ge 4$ and $|E(M)|-|E(N)| \ge 5$.
  Then either
  \begin{enumerate}
    \item $M$ has an $N$-detachable pair;
    \item there is a matroid $M'$ obtained by performing a single $\Delta$-$Y$ or $Y$-$\Delta$ exchange on $M$ such that $M'$ has 
      an $N$-detachable pair; or
    \item there is some $P \subseteq E(M)$ such that $E(M)-E(N) \subseteq P$, and $P$ is 
      \begin{enumerate}[label=\rm(\alph*)]
        \item a \spikelike,
        \item a \twisted, 
        \item a \tvamoslike\ of $M$ or $M^*$, 
        \item an \pspider, 
        \item a \spider,
        \item a quad $3$-separator,
      \end{enumerate}
      or an \auging\ of one of these \psep s; or
    \item $|E(M)| =12$, and $M$ is either a \quadflower\ or a \tcn.
  \end{enumerate}
\end{corollary}

In particular, \cref{maincorollary} implies \cref{maintheoremdetailed}.

\section{Epilogue}
\label{secend}

\subsection*{Detachable pairs: the chain theorem}

A \emph{spike} is a matroid on $2r$ elements, for some $r \ge 3$, with a partition $(L_1,L_2,\dotsc,L_r)$ of the ground set into pairs such that $L_i \cup L_j$ is a quad for all distinct $i,j\in \seq{r}$.
Note that what we refer to as a spike is sometimes called a ``tipless spike''.
A pair of elements $\{x_1,x_2\}$ in a matroid $M$ are \emph{detachable} if either $M/x_1/x_2$ or $M \backslash x_1 \backslash x_2$ is $3$-connected.
Alan Williams proved the following chain theorem for detachable pairs in his Ph.D.\ thesis~\cite{Williams2015}.

\begin{theorem}
  \label{chain}
  Let $M$ be a $3$-connected matroid with $|E(M)| \ge 13$.
  Then either
  \begin{enumerate}
    \item $M$ has a detachable pair,
    \item there is a matroid $M'$ obtained by performing a single $\Delta$-$Y$ or $Y$-$\Delta$ exchange on $M$ such that $M'$ has a detachable pair, or
    \item $M$ is a spike.
  \end{enumerate}
\end{theorem}

His proof of this theorem takes a similar approach as the first two papers of this series~\cite{paper1,paper2}; however, a number of the results in these two papers have simpler analogues 
when there is no $N$-minor to worry about.

Alternatively, \cref{chain} can be obtained as a consequence of \cref{maintheoremdetailed}.  Tutte's Wheels-and-Whirls Theorem~\cite{tutte1966} implies that every $3$-connected matroid with at least four elements has either $U_{2,4}$ or $M(K_4)$ as a minor.  Thus, we can choose $N$ to be either $U_{2,4}$ or $M(K_4)$, in which case $|E(M)|-|E(N)| \ge 9$ or $|E(M)-|E(N)| \ge 7$ respectively, and then apply \cref{maintheoremdetailed} and analyse the remaining structures. 
We omit the details.

\subsection*{$N$-detachable pairs in graphic matroids}

We now consider the implications of \cref{maintheoremdetailed} for graphic matroids.
It is easy to check that the \psep s in case (iii) of \cref{maintheoremdetailed} cannot appear in graphic matroids.  
Hence, \cref{maintheoremdetailed} has the following corollary:

\begin{theorem}
  \label{splittergraphic}
  Let $M$ be a $3$-connected graphic matroid, and let $N$ be a $3$-connected minor of $M$ such that $|E(M)|-|E(N)| \ge 5$.
  Then either
  \begin{enumerate}
    \item $M$ has an $N$-detachable pair, or
    \item there is a matroid $M'$ obtained by performing a single $\Delta$-$Y$ or $Y$-$\Delta$ exchange on $M$ such that $M'$ has 
      an $N$-detachable pair.
  \end{enumerate}
\end{theorem}

Let $G$ be a graph on at least four edges with no isolated vertices.
It is well known that the cycle matroid $M(G)$ is $3$-connected if and only if $G$ is $3$-connected and simple.
Moreover, $\Delta$-$Y$ exchange in matroids generalises the corresponding operation in graphs,
and a matroid obtained by performing a $\Delta$-$Y$ exchange on a triangle of $M(G)$ is graphic.
However, graphic matroids are not closed under $Y$-$\Delta$ exchange;
if $M(G)$ is a graphic matroid, and $M'$ is obtained from $M(G)$ by a $Y$-$\Delta$ exchange on a triad $T^*$, then $M'$ is graphic if and only if $T^*$ is not a separating triad.
Nevertheless, we expect that a careful analysis of \cref{unfortunatetri} for graphic matroids would confirm the following conjecture:

\begin{conjecture}
  \label{splittergraph}
  Let $G$ be a simple $3$-connected graph, and let $H$ be a simple $3$-connected minor of $G$ such that $|E(G)|-|E(H)| \ge 5$.
  Then either
  \begin{enumerate}
    \item $G$ has a pair of edges $\{x,y\}$ such that either $G/x/y$ or $G \ba x \ba y$ is simple, $3$-connected, and has an isomorphic copy of $H$ as a minor, or
    \item there is a graph $G'$ obtained by performing a single $\Delta$-$Y$ or $Y$-$\Delta$ exchange on $G$ such that $G'$ has 
      a pair of edges $\{x,y\}$ for which either $G'/x/y$ or $G' \ba x \ba y$ is simple, $3$-connected, and has an isomorphic copy of $H$ as a minor.
  \end{enumerate}
\end{conjecture}

Similarly, we conjecture the following analogue of \cref{chain}.

\begin{conjecture}
  \label{chaingraph}
  Let $G$ be a simple $3$-connected graph such that $|E(G)| \ge 10$.
  Then either
  \begin{enumerate}
    \item $G$ has a pair of edges $\{x,y\}$ such that either $G/x/y$ or $G \ba x \ba y$ is a simple $3$-connected graph, or
    \item there is a graph $G'$ obtained by performing a single $\Delta$-$Y$ or $Y$-$\Delta$ exchange on $G$ such that $G'$ has a pair of edges $\{x,y\}$ for which either $G'/x/y$ or $G' \ba x \ba y$ is a simple $3$-connected graph.
  \end{enumerate}
\end{conjecture}

\subsection*{Eliminating $\Delta$-$Y$s}

A natural question is what additional structures arise if no $\Delta$-$Y$ or $Y$-$\Delta$ exchange is permitted.

As a starting point, we consider here an analogue of Tutte's Wheels-and-Whirls Theorem; that is, an analogue of \cref{chain} where case (ii) is not allowed.
Certainly, if $M$ is a wheel or a whirl, then it follows from Tutte's Wheels-and-Whirls Theorem that $M$ has no detachable pairs.
It is not hard to see that $M(K_{3,k})$ and, dually, $M^*(K_{3,k})$ also have no detachable pairs, for any $k \ge 3$.

A spike has a unique single-element extension by an element that is in the closure of each of the legs, and a unique single-element coextension by an element that is in the coclosure of each of the legs; we call a matroid with both a \emph{spike with tip and cotip}.  
If $M$ is a spike with tip and cotip, then $M$ has no detachable pairs.
In fact, these matroids are examples of a more general family of matroids with no detachable pairs, which we now describe.

Let $M$ be a spike with tip $t$ and cotip $s$, whose $q$ legs are the pairs $\{x_i,y_i\}$ for $i \in \seq{q}$, where $q \ge 2$.  Note that $M$ has rank $q+1$, and we consider the rank-$3$ matroid $Q_6$ to be the unique tipped and cotipped spike with two legs.
Let $0 \le p \le q$.
For each $i \in \seq{p}$, attach a wheel of rank at least three along the triangle $\{x_i,y_i,t\}$, where $t$ and $y_i$ are spokes of the wheel, by generalised parallel connection.
Finally, delete $x_i$ for each $i \in \seq{p}$.
We say that such a matroid is constructed by \emph{attaching wheels to a spike with tip and cotip}.
Observe in particular that $M$ consists of $q$ maximal fans, each of even size and with ends $t$ and $s$ (see \cite[Section~2]{paper1} for the definition, and connectivity properties, of fans).
It follows that any such matroid $M$ has no detachable pairs.

Finally, let $M$ be a paddle $(P_1,\dotsc,P_n)$, for $n \ge 3$, with spine $\{s,t\}$ (that is, $\cl(P_1) \cap \dotsb \cap \cl(P_n) = \{s,t\}$, see \cite{osw04}) where
for each $i\in[n]$, the set $P_i \cup \{s,t\}$ is an odd maximal fan with ends $s$ and $t$.
Such a matroid can be constructed as follows.
Let $L$ be a line containing the pair $\{s,t\}$, and consisting of at least three elements.
Attach at least three wheels, each of rank at least three, and each along some triangle containing $\{s,t\}$, by generalized parallel connection.  Finally, delete $L-\{s,t\}$.
We say that such a matroid can be constructed by \emph{attaching wheels with common spokes~$s$ and $t$}.
If we also delete $s$, we obtain a matroid with no detachable pairs.
The simplest example of such a matroid is $M(K_{3,n}')$, where $K_{3,n}'$ is obtained from $K_{3,n}$ by adding a single edge between vertices in the part of size three.

We conjecture the following:

\begin{conjecture}
  \label{conj1}
  Let $M$ be a $3$-connected matroid with $|E(M)| \ge 13$.
  Then either
  \begin{enumerate}
    \item $M$ has a detachable pair,
    \item $M$ is a spike,
    \item $M \cong M(K_{3,k})$ or $M \cong M^*(K_{3,k})$ for some $k \ge 5$,
    \item $M$ is the cycle matroid of a wheel, or a whirl,
    \item $M$ can be constructed by attaching wheels to a spike with tip and cotip, or
    \item $M$ or $M^*$ can be constructed by attaching wheels with common spokes~$s$ and $t$, and then deleting $s$.
  \end{enumerate}
\end{conjecture}

We note that the matroids in (v), and matroids similar to those in (vi), appear in the work of Oxley and Wu on matroids with exactly two non-essential elements~\cite{Oxley2000b}.
An element $e \in E(M)$ is \emph{essential} if neither $M \ba e$ nor $M/e$ is $3$-connected.
Oxley and Wu showed that the class of $3$-connected matroids with precisely two non-essential elements, one of which must be deleted and one of which must be contracted,
is the class of matroids 
that can be constructed by attaching wheels to a spike with tip and cotip~\cite[Theorem~1.4]{Oxley2000b}.
Meanwhile, the class of $3$-connected matroids with precisely two non-essential elements, both of which must be deleted, are matroids that can be constructed by attaching wheels with common spokes~\cite[Theorem~1.3]{Oxley2000b}.
A matroid of case (vi) is obtained by deleting a further element from the spine of the paddle.

\medskip
We now focus on simple $3$-connected graphs. 
By interpreting \cref{conj1} for graphic matroids, we obtain a conjecture regarding the existence of ``detachable pairs'' in simple $3$-connected graphs.

We first require some definitions; these 
follow \cite{Oxley2000b}.
Consider a copy of $K_4$ where the edges $e$ and $f$ are non-adjacent, and let $u$ and $v$ be vertices such that $u$ incident to $f$, and $v$ is incident to $e$.
A \emph{twisted wheel} is a graph that can be obtained by subdividing $e$ so that $j \ge 0$ new vertices are introduced, adding $j$ edges between each of these vertices and $u$; then subdividing $f$ so that $k \ge 0$ new vertices are introduced, and adding $k$ edges between each of the $k$ new vertices and $v$.

A \emph{multi-dimensional wheel} is constructed as follows: begin with the $3$-vertex path $uhv$, and add $k \ge 3$ parallel edges between $u$ and $v$.  For each of the $k$ parallel edges, subdivide it so that one or more new vertices are introduced, and join each new vertex to $h$.  
A graph obtained by then removing the edge $uh$ is an \emph{unhinged multi-dimensional wheel}.

We conjecture the following:

\begin{conjecture}
  Let $G$ be a simple $3$-connected graph with $|E(G)| \ge 13$. 
  Then either
  \begin{enumerate}
    \item $G$ has a pair of edges $\{x,y\}$ such that either $G\ba x\ba y$ or $G/x/y$ is a simple $3$-connected graph,
    \item $G$ is a wheel,
    \item $G \cong K_{3,k}$ for some $k \ge 5$,
    \item $G$ is a twisted wheel, or
    \item $G$ is an unhinged multi-dimensional wheel. 
  \end{enumerate}
\end{conjecture}

\bibliographystyle{abbrv}
\bibliography{lib}

\end{document}